\documentclass[reqno,11pt]{amsart}
\calclayout
\usepackage{amsmath,amsthm,amssymb,amsbsy}
\usepackage{amsfonts,amstext,amscd}
\usepackage{graphicx}
\usepackage{bm}
\usepackage{commath,mathtools,setspace}
\usepackage{paralist}
\usepackage{color}
\usepackage[T1]{fontenc}
\usepackage{geometry}
\usepackage[colorlinks=true,linkcolor=blue,citecolor=blue]{hyperref}
\usepackage[colorinlistoftodos,prependcaption,textsize=small]{todonotes}
\definecolor{red}{RGB}{255,0,0}
\definecolor{green}{RGB}{0,100,0}
\definecolor{blue}{RGB}{0,0,255}
\newtheorem{theorem}{Theorem}[section]
\newtheorem{thmx}{Theorem}
\newtheorem{lemma}[theorem]{Lemma}

\newtheorem{corollary}[theorem]{Corollary}
\newtheorem{proposition}[theorem]{Proposition}

\theoremstyle{remark}
\newtheorem{remark}[theorem]{Remark}
\newtheorem{definition}[theorem]{Definition}
\newtheorem{example}[theorem]{Example}

\renewcommand{\Re}{\mathop{\rm Re}}
\renewcommand{\Im}{\mathop{\rm Im}}

\newcommand{\dil}[1]{\text{Dil}_{#1}}

\DeclareMathOperator{\supp}{supp}

\renewcommand{\P}{\mathbb{P}}

\newcommand{\R}{\mathbb{R}}
\newcommand{\C}{\mathbb{C}}
\newcommand{\Z}{\mathbb{Z}}

\newcommand{\N}{\mathbb{N}}

\newcommand{\isdef}{\overset{\mathrm{def}}{=\joinrel=}}

\newcommand{\monicpolreal}{\mathbb P_n^*(\mathbb R)}

\newcommand{\cc}{\mathbb{C}}

\newcommand{\pp}{\mathbb{P}}

\newcommand{\rr}{\mathbb{R}}

\newcommand{\MM}{\mathcal{M}}

\newcommand{\RR}{\mathcal{R}}
\renewcommand{\SS}{\mathcal{S}}

\newcommand{\raising}[2]{\left(#1\right)^{\overline{#2}}}
\newcommand{\falling}[2]{\left(#1\right)^{\underline{#2}}}

\newmuskip\pFqmuskip

\newcommand*\pFqN[6][8]{%
  \begingroup % only local assignments
  \pFqmuskip=#1mu\relax
  \mathcode`\,=\string"8000
  \begingroup\lccode`\~=`\,
  \lowercase{\endgroup\let~}\pFqcomma
  {}_{#2}F_{#3}{\left(\genfrac..{0pt}{}{#4}{#5};#6\right)}%
  \endgroup
}
\newcommand{\pFqcomma}{\mskip\pFqmuskip}

\newcommand*\HGF[5]{%
 {\ }_{#1} F_{#2}{\left(\genfrac..{0pt}{}{#3}{#4};#5\right)}%
}

\title[Zeros via finite free convolution and MOP]{Zeros of generalized hypergeometric polynomials via finite free convolution. Applications to multiple orthogonality} 
\author[A. Mart\'{\i}nez-Finkelshtein]{Andrei Mart\'{\i}nez-Finkelshtein}

\address[AMF]{Department of Mathematics, Baylor University, TX, USA, and Department of Mathematics, University of Almer\'{\i}a, Spain}

\email{a\_martinez-finkelshtein@baylor.edu}

\author[R.~Morales]{Rafael Morales}

\address[RM]{Department of Mathematics, Baylor University, TX, USA}

\email{rafael\_morales2@baylor.edu}

\author[D.~Perales]{Daniel Perales}

\address[DP]{Department of Mathematics, Texas A\&M University, TX, USA}

\email{daniel.perales@tamu.edu}

\date{\today}



\begin{document}

\begin{abstract}
We address the problem of the weak asymptotic behavior of zeros of families of generalized hypergeometric polynomials as their degree tends to infinity. The main tool is the representation of such polynomials as a finite free convolution of simpler elements; this representation is preserved in the asymptotic regime, so we can formally write the limit zero distribution of these polynomials as a free convolution of explicitly computable measures. We derive a simple expression for the $S$-transform of the limit distribution, which turns out to be a rational function, and a representation of the Kampé de Fériet polynomials in terms of finite free convolutions. 

We apply these tools, as well as those from \cite{martinez2023real}, to the study of some well-known families of multiple orthogonal polynomials (Jacobi-Piñeiro and multiple Laguerre of the first and second kinds), obtaining results on their zeros, such as interlacing, monotonicity, and asymptotics.
\end{abstract}

\maketitle

\tableofcontents

\section{Introduction}

The definition of the generalized hypergeometric function ${ }_{i} F_j$ with $i$ numerator and $j$ denominator parameters is well known. If one of the numerator parameters is equal to a negative integer, say $ -n$, with $n \in \mathbb{N}$, then the series terminates and is a polynomial of degree $\le n$. If all its zeros are real, we usually want to know their properties, such as positivity/negativity, interlacing, and monotonicity with respect to the parameters. In the case of a sequence of such polynomials, enumerated by their degree $n$, we also want to investigate their asymptotic behavior as $n\to \infty$.

For small values of $i$ and $j$, the answer to these questions can usually be obtained by exploiting their connection to some classical families of polynomials, in many cases orthogonal, or using their other properties, such as the differential equation or integral representation. But when $i, j\ge 2$, the problem becomes more difficult due to the limited number of tools that allow us to investigate their behavior.  

The finite free convolutions (that in this paper come in two flavors, multiplicative $\boxtimes_n$ and additive $\boxplus_n$), are binary operations on polynomials, studied already by Szeg\H{o}, Schur, Walsh, and others, although under different names. They behave especially well when applied to real-rooted polynomials, preserving zero interlacing and monotonicity. Recently, such convolutions have been rediscovered as expected characteristic polynomials of a multiplication (or addition) of random matrices \cite{MR4408504}, and were also interpreted as finite analogs of the free probability \cite{marcus} (thus, named generically as finite free convolution of polynomials). 

The connection between these polynomial convolutions and free probability is revealed in the asymptotic regime, when we consider the zero-counting measure (also known in this context as the empirical root distribution) of a sequence of polynomials whose degree tends to infinity. Then the finite free convolution of polynomials turns into a free convolution of limiting distributions of their zeros \cite{arizmendi2021finite, arizmendi2018cumulants}.

In a recent paper \cite{martinez2023real}, the authors illustrated the power of finite free convolution of polynomials to prove the real-rootedness, interlacing, or monotonicity of zeros with respect to the parameters. The key tool was the representation of generalized hypergeometric polynomials as finite free convolutions of simpler ``building blocks'', much easier to study. We also briefly explained the potential of this approach in the study of asymptotic behavior.

This work is, in a certain sense, a natural continuation of \cite{martinez2023real}. Here, the main focus is precisely on the weak asymptotic behavior of zeros of families of generalized hypergeometric polynomials as their degree tends to infinity. Using the representation of a generalized hypergeometric polynomial as a finite free convolution of some simpler elements, we can formally write the limit zero distribution of these polynomials as a free convolution of simpler measures. In order to convert this observation into an effective computational tool, we derive in Section~\ref{sec:Stransform} an expression for the so-called $S$-transform of the limit zero-counting measures of the original polynomials. Unlike the case of the Cauchy transform of such a measure, which is normally an algebraic function, the $S$-transform turns out to be a rational function, easily expressible in terms of the main parameters of the problem.

In the second part of the article, we apply these tools, as well as those from \cite{martinez2023real}, to the study of some well-known families of multiple orthogonal polynomials. While writing the paper, we became aware that the recent contribution \cite{wolfs} mentions some connections of multiple orthogonal polynomials and finite free convolution using the results from \cite{martinez2023real}.

\textbf{Multiple}, also known as Hermite-Pad\'e, \textbf{orthogonal polynomials} (MOPs),  are polynomials of one variable that satisfy the orthogonality conditions with respect to several measures. They are a very useful extension of orthogonal polynomials and have recently received renewed interest because tools have become available to investigate their asymptotic behavior, and they do appear in a number of fascinating applications, see \cite{MR3525716}. 

In this paper, we consider only the case of absolutely continuous orthogonality measures on the real line. Thus, let $r$ be a positive integer and $w_1, w_2,\ldots, w_r$ be non-negative integrable functions (``weights'') on $\R$ for which all the moments
are finite. Let $\bm n = (n_1,n_2,\ldots,n_r) \in \mathbb{N}^r$ be a multi-index of size $|\bm{n}| = n_1+n_2+\cdots+n_r$.
There are two types of multiple orthogonal polynomials. \textbf{Type I multiple orthogonal polynomials} are given as a vector
$(A_{\bm{n},1},A_{\bm{n},2},\ldots,A_{\bm{n},r})$ of $r$ polynomials, where $A_{\bm{n},j}$ has degree $\leq n_j-1$, for
which the function
\begin{equation}
    \label{def:typeIfunction}
    Q_{\bm{n}}(x) \isdef \sum_{j=1}^r A_{\bm{n},j}(x)w_j(x) 
\end{equation}
is orthogonal to all polynomials of degree $\leq |\bm{n}|-2$:
\begin{equation}  \label{typeI}
  \int x^{k}  Q_{\bm{n}}(x)\, dx = \sum_{j=1}^r \int x^k A_{\bm{n},j}(x) w_j(x) \, dx = 0, \qquad 0 \leq k \leq |\bm{n}|-2.
\end{equation}
One usually adds the normalization
\begin{equation}  \label{typeInorm}
    \int x^{|\bm{n}|-1}  Q_{\bm{n}}(x)\, dx = 1 .   
\end{equation}

The \textbf{Type II multiple orthogonal polynomial} $P_{\bm{n}}$ is the monic polynomial of degree $|\bm{n}|$ that satisfies
the orthogonality conditions
\begin{equation}  \label{typeII}    
\int P_{\bm{n}}(x) x^k  w_j(x)\, dx = 0, \qquad  0 \leq k \leq n_j-1, 
\end{equation}
for $1 \leq j \leq r$. 
Both conditions \eqref{typeI}--\eqref{typeInorm} and condition \eqref{typeII} yield a corresponding linear system of $|\bm{n}|$ equations in the $|\bm{n}|$ unknowns, either
coefficients of polynomials $(A_{\bm{n},1},\dots,A_{\bm{n},r})$ or coefficients of the monic polynomial $P_{\bm{n}}$.
The matrices of these linear systems are each other's transpose and contain moments of the $r$ weights $(w_1,\ldots,w_r)$.
A solution of these linear systems may not exist or may not be unique. One needs extra assumptions on the weights
$(w_1,\ldots,w_r)$ for a solution to exist and to be unique. If a unique solution exists for a multi-index $\bm{n}$ then
the multi-index is said to be \textbf{normal}. If all multi-indices are normal, then the system $(w_1,\ldots,w_r)$ is said to be a \textbf{perfect system}.

The weights $(w_1,\ldots,w_r)$ form an \textbf{AT-system} on an interval $\Delta\subset \R$ for a multi-index $\bm n \in \N^r$ if for any polynomials $A_{\bm{n},j}$, $j=1, \ldots, r$, in \eqref{def:typeIfunction}, satisfying the mentioned degree constraints and not all equal to $0$, the function $Q_{\bm{n}}$ has at most $|\bm{n}|-1$ zeros on $\Delta$, see, e.g.~\cite[Chapter 23]{MR2542683}. It is known that if $(w_1,\ldots,w_r)$ is an AT-system for every multi-index $\bm n \in \N^r$, the system is perfect. An example of such systems are the \textbf{Nikishin systems}, fact proved in \cite{MR2852293, MR2832123}. Since this is not central to our discussion, for the definition of such systems we refer the reader to \cite{MR3311759, MR3319505, MR4238535}.

For every AT-system and for any $\bm{n} \in \mathbb{N}^r$, the Type I function for $Q_{\bm{n}}$ defined by \eqref{def:typeIfunction} and satisfying \eqref{typeI}, has exactly $|\bm{n}|-1$ sign changes in $\Delta$, while the Type II multiple orthogonal polynomial $P_{\bm n}$, satisfying \eqref{typeII}, has $|\bm{n}|$ simple zeros on $\Delta$. These zeros exhibit an interlacing property, meaning that there is always a zero of $P_{\bm{n}}$ between two consecutive zeros of $P_{\bm{n}+\bm{e}_k}$, for each $0 \leq k \leq r-1$, where $\bm{e}_k \in \mathbb{N}^r$ is the multi-index that has all entries equal to 0 except the entry of index $k$ which is equal to 1.  
 
Various families of special MOPs have been found, extending the classical orthogonal polynomials, but also related to completely new special functions \cite{MR1662713}, \cite[Ch.~23]{MR2542683}. 

One of the well-known AT-systems is given by the weights 
\begin{equation}
    \label{JPweights}
    w_j(x)=x^{\alpha_j}(1-x)^\beta, \quad j=1, 2,\dots, r,
\end{equation}
on the interval $\Delta=[0,1]$. 
Here $\alpha_1, \ldots, \alpha_r, \beta>-1$ and $\alpha_i-\alpha_j \notin \mathbb{Z}$ for $i \neq j$. The corresponding polynomials are known as the \textbf{Jacobi-Piñeiro polynomials}, and are studied in Sections \ref{subsec:JPTypeI} and \ref{subsec:JPTypeII}.

Two different AT-systems correspond to multiple Laguerre weights on $[0,+\infty)$, see, e.g.~\cite[\S 23.4]{MR2542683}. First, we can take
\begin{equation}
    \label{Lkind1weights}
    w_j(x)=x^{\alpha_j}e^{-x}, \quad j=1, 2,\dots, r,
\end{equation}
where again $\alpha_1, \ldots, \alpha_r >-1$ and  $\alpha_i-\alpha_j \notin \mathbb{Z}$ for $i \neq j$. The corresponding polynomials are known as \textbf{ multiple Laguerre polynomials of the first kind}, and their behavior is discussed in Sections \ref{subsec:MLTypeI} and \ref{subsec:MLTypeII}. 

Another option is to define the weights
\begin{equation}
    \label{Lkind2weights}
    w_j(x)=x^{\alpha}e^{-c_jx}, \quad j=1, 2,\dots, r,
\end{equation}
where $\alpha>-1$, with all $c_i>0$, and such that $c_i\neq c_j$ for $i\neq j$. The corresponding polynomials are known as \textbf{multiple Laguerre polynomials of the second kind}, and their properties are discussed in Sections \ref{sec:TIML2k} and \ref{sec:T2ML2k}. For Type I, they are a special case of the Kampé de Fériet polynomials, so, in Section~\ref{sec:KdF}, we prove the more general fact that they can be represented as a finite free additive convolution of hypergeometric polynomials.

The asymptotic behavior of MOPs is a highly non-trivial subject. Although the study of the Hermite-Padé polynomials goes back to the original works of Hermite, see~\cite{Hermite1, Hermite2}, as well as \cite{MR1130396}, the first important asymptotic result appeared in the work of Kalyagin \cite{MR0562212}. In the 1980s, the ground-breaking works of Aptekarev, Gonchar, Rakhmanov, and Stahl, made clear that the asymptotics of the Type I form \eqref{def:typeIfunction} (but not of the individual entries $A_{\bm n, i}$) and of the Type II polynomials $P_{\bm n}$ can be described in terms of a vector equilibrium problem for logarithmic potentials, \cite{MR1240781, MR0651757, MR0807734}. However, solving such a problem is usually a formidable task.

Another approach is via the higher-order nearest-neighbor recurrence relations by MOPs, established in the work of Van Assche \cite{MR2832734} (see also \cite[Ch.~23]{MR2542683}) that allow to derive an algebraic equation on a weighted distribution of zeros. An important ingredient of this method is the expression (or, at least, the behavior) of the recurrence coefficients. A systematic study of a large number of classical and semi-classical families of orthogonality weights for MOP, and in particular, of their recurrence coefficients, was carried out in a number of contributions by Van Assche and his collaborators in recent years; see, e.g.~\cite{MR1990569}. A limitation of this approach is that it allows one to address only the quasi-diagonal case (step line) for the multi-indices $\bm n$.

Finally, another recent and formidable tool for asymptotic analysis is the non-linear steepest descent method of Deift and Zhou, applied to the Riemann-Hilbert characterization of MOPs \cite{MR2006283}. This technique renders extremely precise asymptotic information
(see, e.g.~\cite{MR2103904, MR2849479, MR3525716, MR3939592}, to mention a few), but at a very high cost of being technically challenging. 

In this paper, we address the problem of the asymptotic zero distribution of the three families of MOP mentioned above, when the degree $|\bm n| \to \infty$, also allowing a linear dependence of the parameters $\alpha$'s and $\beta$'s on $\bm n$. As recent investigations show, these polynomials are hypergeometric, so we can use the free convolution approach at a relatively low cost. We stress that one of the appeals of this technique is its simplicity. Alternatives such as the general methods described above, or using the differential equation or the integral representation \cite{zhou2024asymptotic} of these MOPs are usually much more involved. It is worth mentioning also that the zero asymptotics of the Type II multiple Laguerre polynomials of the second kind was obtained in terms of free convolution of measure in \cite{hardy2015} using arguments from random matrix theory.

Since a representation of a real-rooted polynomial in terms of free convolution can involve polynomials with complex zeros, the results on the asymptotic regime from \cite{arizmendi2021finite,arizmendi2018cumulants} cannot be directly applied. Taking advantage of the fact the proofs in those references are essentially algebraic, based on the behavior of moments and cumulants, we adapt these arguments to analyze the finite free convolutions in the asymptotic regime for measures compactly supported on the complex plane and not necessarily on the real line; see Section~\ref{sec:finiteinAsymptotics} and Appendix~\ref{sec:appendix}.

As a by-product of the representation of the hypergeometric polynomials in terms of finite free convolutions, we also derive some zero monotonicity and interlacing properties for the mentioned families of MOP. They appear to be new in the majority of cases; they are especially interesting for the polynomials $A_{\bm n,i}$ in \eqref{def:typeIfunction}, where not much is known.

\section{Preliminaries} \label{sec:prelim}

\subsection{Notation}\ 

In what follows, $\P_n$ stands for all algebraic polynomials of degree $\le n$, and $\P \isdef \cup_{n\ge 0} \P_n$.
Also, for $K\subset \C$, we denote by $\P_n(K) $  the subset of polynomials of degree $\le n$  with all zeros in $K$. In particular, $\P_n(\R)$ denotes the family of real-rooted polynomials of degree $\le n$, and $\P_n(\R_{\ge 0})$ is the subset of $\P_n(\R)$ of polynomials having only roots in $\R_{\ge 0} \isdef [0,+\infty)$, e.g real and non-negative. 

The \textbf{rising factorial} (also, \textbf{Pochhammer's symbol\footnote{\, Another standard notation for the raising factorial is $(a)_j$. We prefer to use the notation defined here.}})  for $a\ne 0$ and $j\in \mathbb Z_{\ge 0}\isdef \N \cup \{0\}$ is  
$$
\raising{a}{j} \isdef a(a+1)\dots(a+j-1)= \frac{\Gamma(a+j)}{\Gamma(a)}, \quad \raising{a}{0} \isdef  1,$$
while the \textbf{falling factorial} is defined as
$$
\falling{a}{j} \isdef  a(a-1)\dots(a-j+1)= \raising{a-j+1}{j}, \quad \falling{a}{0}\isdef  1.
$$
If $\bm a =(a_1, \dots, a_i)\in \R^i$ is a vector (tuple), we understand by
$$
\raising{\bm a}{k} =\prod_{s=1}^i \raising{a_s}{k}.
$$

For $n\in \N$, we will use the notation
\begin{equation}
    \label{notationZn}
    \Z_n \isdef  \{0,  1, 2,\dots,  n-1  \}, \quad (-\Z_n) \isdef \{0, -1,-2,\dots, -n+1  \}.
\end{equation}

Finally, we will write
$$
p \simeq q
$$
to indicate that polynomials $p$ and $q$ coincide up to a non-zero multiplicative constant.
We will also use the standard notation of the theory of orthogonal polynomials,
\begin{equation}
    \label{defReversed}
p^*(z)\isdef z^n\, \overline{p(1/\overline{z})}
\end{equation}
for the \textbf{reversed} polynomial of $p$. Clearly, for polynomials $p$ with real coefficients, $p^*(x)= x ^n\, p(1/x)$, in which case $p^*$ is usually called the \textbf{reciprocal} of $p$.

\subsection{Hypergeometric polynomials}\

For $\bm a =(a_1, \dots, a_i)\in \R^i$ and $\bm b =(b_1, \dots, b_j)\in \R^j$, a generalized \textbf{hypergeometric series} \cite{MR2656096, MR2723248} is an expression 
\begin{equation*}  
\HGF{i+1}{j}{a_0,\ \bm a}{\bm b}{x}=\sum_{k=0}^\infty  \frac{\raising{a_0}{k} \raising{\bm a }{k} }{\raising{\bm b }{k}  } \frac{x^k}{k !}.
\end{equation*}
In the particular case when $a_0$ is a negative integer, the series is terminating:
\begin{equation}
    \label{consequenceconnectionRandL2}
 \HGF{i+1}{j}{-n,\ \bm a  }{\bm b }{x}=\sum_{k=0}^n  \frac{\raising{-n}{k} \raising{\bm a }{k} }{\raising{\bm b }{k}  } \frac{x^k}{k !}
\end{equation}
is a (generalized) \textbf{hypergeometric polynomial} of degree $ \le n$, as long as 
\begin{equation}  \label{assumptionsHGFB}
    b_1,\dots b_j \notin \left(- \Z_{n+1}\right).
\end{equation}
Additionally, the polynomial in \eqref{consequenceconnectionRandL2} is of degree exactly $n$ if and only if
\begin{equation}\label{assumptionsHGFNew}
  a_1,\dots, a_i \notin \left(- \Z_n\right).
\end{equation} 
In what follows, we always assume that conditions \eqref{assumptionsHGFB}--\eqref{assumptionsHGFNew} hold. From a direct computation, we get the following equation for the derivative of a hypergeometric polynomial,
 \begin{equation}\label{derivativeHyper}
     \frac{d}{dx}\left( \HGF{i+1}{j}{-n,\ \bm a  }{\bm b }{x}\right)\simeq\HGF{i+1}{j}{-n+1,\ \bm a+1  }{\bm b+1 }{x}.
 \end{equation}

We will also make use of a multivariable generalization of the Kampé de Fériet series, e.g.~\cite[Section 1.4, Eq.~(24)--(25)]{MR0834385}: for $k=0,1,2,\dots, r$, consider $i_k,j_k\in \N$ and tuples of parameters $\bm a_k\in \rr^{i_k}$,  $\bm b_k\in \rr^{j_k}$; then 
$$
F_{j_0: j_1 ;\dots; j_r}^{i_0: i_1 ;\dots; i_r}\left[\left.\begin{array}{c}
\bm a_0:\bm a_1 ; \dots ; \bm a_r \\
\bm b_0:\bm b_1 ;\dots ; \bm b_r
\end{array} \right\rvert\, x_1,\dots,x_r \right] \isdef \sum_{l_1=0}^{\infty}\cdots  \sum_{l_r=0}^{\infty} \frac{\raising{\bm a_0}{l_1+\dots +l_r}  \raising{\bm a_1}{l_1} \cdots \raising{\bm a_r}{ l_r}  }{\raising{\bm b_0}{l_1+\dots +l_r}  \raising{\bm b_1}{l_1} \cdots \raising{\bm b_r}{ l_r}}   \frac{x_1^{l_1}}{l_1!} \cdots \frac{x_r^{l_r}}{l_r!} ,
$$
along with the Kampé de Fériet polynomials: for $n\in \Z_{\ge 0}$,
\begin{equation}
    \label{KdFpolyns.mult}
    \begin{split}
 F_{j_0: j_1 ;\dots; j_r}^{i_0+1: i_1 ;\dots; i_r} & \left[\left.\begin{array}{c}
-n,\bm a_0:\bm a_1 ; \dots ; \bm a_r \\
\bm b_0:\bm b_1 ;\dots ; \bm b_r
\end{array} \right\rvert\,  x_1,\dots,x_r \right] \\
& = \sum_{k=0}^{n}\frac{\raising{-n}{k}\raising{\bm a_0}{k} }{\raising{\bm b_0}{k}}\sum_{l_1+\cdots+l_r=k} \frac{\raising{\bm a_1}{l_1}x_1^{l_1}}{\raising{\bm b_1}{l_1} l_1!}   \cdots \frac{\raising{\bm a_r}{ l_r} x_r^{l_r}}{\raising{\bm b_r}{ l_r} l_r!} \\
& = \sum_{l_1+\cdots+l_r+m=n} \frac{\raising{-n}{n-
m}\raising{\bm a_0}{n-m} }{\raising{\bm b_0}{n-m}}\frac{\raising{\bm a_1}{l_1}x_1^{l_1}}{\raising{\bm b_1}{l_1} l_1!}   \cdots \frac{\raising{\bm a_r}{ l_r} x_r^{l_r}}{\raising{\bm b_r}{ l_r} l_r!}.
  \end{split}
\end{equation}

\subsection{Finite free convolutions}\ \label{sec:finitefreeconv}

\begin{definition}[\cite{MR4408504}] \label{def:MultiplicativeConv}
Given two polynomials, $p$ and $q$, of degree at most $n$, the 
\textbf{$n$-th multiplicative finite free convolution} of $p$ and $q$, denoted as $p\boxtimes_n q$, is a polynomial of degree at most $n$, which can be defined in terms of the coefficients of polynomials written in the form  
\begin{equation}
    \label{defFMULTCONV}
    p(x)=\sum_{j=0}^n x^{n-j}(-1)^j e_j(p) \quad \text { and } \quad q(x)=\sum_{j=0}^n x^{n-j}(-1)^j e_j(q).
\end{equation}
Namely,
$$
[p\boxtimes_n q](x) \isdef \sum_{k=0}^n x^{n-k}(-1)^k e_k(p\boxtimes_n q),
$$
with
\begin{equation}
    \label{coeffMultConv}
  e_k(p\boxtimes_n q) \isdef  \binom{n}{k}^{-1} e_k(p) e_k(q).
\end{equation}
\end{definition}
In particular, if $p, q $ are of degree $n$ and monic, then also $p\boxtimes_n q $ has the same property. 

The multiplicative finite free convolution is a bi-linear operator  from $\P_n \times \P_n$ to $\P_n$: if $p,q,r\in \P_n$,  and $\alpha\in\rr$, then 
\begin{equation*}
    (\alpha p+q)\boxtimes_n r= \alpha (p\boxtimes_n r)+q\boxtimes_n r.
\end{equation*}

From Definition \ref{def:MultiplicativeConv} it easily follows that multiplicative convolution with the polynomial $(x-\alpha)^n$ is equivalent to a  \emph{dilation} of the roots by $\alpha$:
\begin{equation}
    \label{identityMult2}
    \dil{\alpha}(p) \isdef \alpha^n p\left(\frac{x}{\alpha}\right)=p(x) \boxtimes_n (x-\alpha)^n , \quad \alpha\neq 0.
\end{equation}
In consequence,
\begin{equation}
    \label{identityMult3a}
    (\dil{\alpha} p) \boxtimes_n q = p \boxtimes_n (\dil{\alpha} q)= \dil{\alpha} \left[ p \boxtimes_n q\right], \quad \alpha\neq 0,
\end{equation}
and
\begin{equation}
    \label{identityMult3b}
\dil{\alpha}\dil{\beta} (p)= \dil{\alpha\beta} (p), \quad \alpha,\beta\neq 0.
\end{equation}

The following result was proved in \cite[Theorem 3.1]{martinez2023real}:
\begin{thmx} 
\label{thm:multiplicativeConvHG}
If $n\in \mathbb Z_{\ge 0}$, and  
$$p(x)=\HGF{i_1+1}{j_1}{-n, \bm a_1}{\bm b_1}{x}, \qquad q(x)=\HGF{i_2+1}{j_2}{-n, \bm a_2}{\bm b_2}{x},$$
where the parameters $\bm a_1,\bm a_2,\bm b_1,\bm b_2 $ are tuples (of sizes $i_1,i_2,j_i,j_2$, respectively), then their $n$-th free multiplicative convolution is given by
$$[p\boxtimes_n q](x)=\HGF{i_1+i_2+1}{j_1+j_2}{-n, \bm a_1, \bm a_2}{\bm b_1,\bm b_2}{x}.$$
\end{thmx}

\begin{definition}[\cite{MR4408504}]
Given two polynomials, $p$ and $q$, of degree at most $n$, the \textbf{$n$-th additive finite free convolution} of $p$ and $q$, denoted as $p\boxplus_n q$, is a polynomial of degree at most $n$, defined in terms of the coefficients of polynomials written in the form 
\begin{equation}
    \label{defFADDCONV}
    p(x)=\sum_{j=0}^n x^{n-j}(-1)^j e_j(p) \quad \text { and } \quad q(x)=\sum_{j=0}^n x^{n-j}(-1)^j e_j(q).
\end{equation}
Namely,
$$
[p\boxplus_n q](x) \isdef \sum_{k=0}^n x^{n-k}(-1)^k e_k(p\boxplus_n q),
$$
with
\begin{equation}
    \label{coeffAdditiveConv}
e_k(p\boxplus_n q) \isdef \falling{n}{k} \sum_{i+j=k} \frac{e_i(p)}{\falling{n}{i}}\frac{ e_j(q)}{\falling{n}{j}} 
\end{equation}
(and thus, $e_0(p\boxplus_n q)=e_0(p)e_0(  q)$).
\end{definition}
 
The additive finite free convolution is a bi-linear operator  from $\P_n \times \P_n$ to $\P_n$: if $p,q,r\in \P_n$,  and $\alpha\in\rr$, then 
\begin{equation*}
    (\alpha p+q)\boxplus_n r= \alpha (p\boxplus_n r)+q\boxplus_n r.
\end{equation*}
Moreover,
    \begin{equation}
        \label{shift}
        p(x) \boxplus_n (x-\alpha)^n = p(x-\alpha), \quad p\in \P_n.
    \end{equation}
and $p \boxplus_n q =0$ if and only if $\deg(p)+\deg(q) < n$, or if $\deg p=n$ then $q\equiv 0$.
This also shows that the inverse of any $p\in \P_n$ under the additive (finite free) convolution $\boxplus_n$ is unique. 

In analogy to \eqref{identityMult3a}, we have
\begin{equation}
    \label{identityAdd3a}
    (\dil{\alpha} p) \boxplus_n (\dil{\alpha}q) \simeq \dil{\alpha} \left(  p \boxplus_n q\right), \quad \alpha\neq 0.
\end{equation}

The following result appears, although in a slightly different form, in \cite[Theorem 3.4]{martinez2023real}:
\begin{thmx}
\label{thm.additive.conv.HG}
Let $p$ and $q$ be hypergeometric polynomials of the following form:
$$
p(x)= \HGF{i_1+1}{j_1}{-n,\ \bm a_1}{\bm b_1}{(-1)^{l_1}x}, \qquad q(x)=\HGF{i_2+1}{j_2}{-n,\ \bm a_2}{\bm b_2}{(-1)^{l_2}x},
$$
where $l_1,l_2\in\{0, 1\}$ the parameters $\bm a_1,\bm a_2,\bm b_1,\bm b_2 $ are tuples (of sizes $i_1,i_2,j_i,j_2$, respectively). 
Then, with the notation $\mathfrak d  =d/dx$, 
\begin{equation*}
    \begin{split}
 [p\boxplus_n q](x)  \simeq   
 \HGF{j_1}{i_1}{-\bm b_1-n+1}{-\bm a_1-n+1}{(-1)^{i_1+j_1+l_1+1}\mathfrak d  }  \HGF{j_2}{i_2}{-\bm b_2-n+1}{-\bm a_2-n+1}{(-1)^{i_2+j_2+l_2+1}\mathfrak d  } [x^n].
    \end{split}
\end{equation*} 
\end{thmx}
From here, an immediate consequence is the following auxiliary result:
\begin{lemma} \label{lemma:trick}
Assume that a polynomial $p$ of degree $n$  can be represented as a product of hypergeometric functions,
$$
 p(x)= \HGF{i_1}{j_1}{-\bm b_1-n+1}{-\bm a_1-n+1}{(-1)^{i_1+j_1+l_1+1}x} \HGF{i_2}{j_2}{ -\bm b_2-n+1}{-\bm a_2-n+1}{(-1)^{i_2+j_2+l_2+1}x},
$$
where $l_1,l_2\in\{0, 1\}$. Then, for the reciprocal polynomial $p^*$,
$$
p^*(x)\simeq  \HGF{2}{0}{-n,1}{\cdot}{x} \boxtimes_n\left[\HGF{i_1+1}{j_1}{-n,\ \bm a_1}{\bm b_1}{ (-1)^{l_1}x} \boxplus_n
\HGF{i_2+1}{j_2}{-n,\ \bm a_2}{\bm b_2}{ (-1)^{l_2}x}\right].$$
\end{lemma}
\begin{proof}
From the hypothesis viewed as a differential function we get
$$p(\mathfrak d)= \HGF{i_1}{j_1}{-\bm b_1-n+1}{-\bm a_1-n+1}{(-1)^{i_1+j_1+l_1+1}\mathfrak d } \HGF{i_2}{j_2}{ -\bm b_2-n+1}{-\bm a_2-n+1}{(-1)^{i_2+j_2+l_2+1}\mathfrak d }.
$$
By Theorem~\ref{thm.additive.conv.HG}, 
\begin{equation}\label{additive}
    p\left(\mathfrak d  \right)[x^n]\simeq\HGF{i_1+1}{j_1}{-n,\ \bm a_1}{\bm b_1}{ (-1)^{l_1}x} \boxplus_n
\HGF{i_2+1}{j_2}{-n,\ \bm a_2}{\bm b_2}{ (-1)^{l_2}x}.
\end{equation}

On the other hand, for a polynomial 
$$
q(x)=\sum_{j=0}^n x^{n-j}(-1)^j e_j(q)
$$
it holds that
\begin{equation}\label{reciprocal}
   q\left(\mathfrak d  \right)[x^n]=\sum_{k=0}^n(-1)^k\falling{n}{n-k} e_k(q)x^k=(-1)^n n!\sum_{k=0}^n(-1)^k\frac{e_k(q^*)}{(n-k)!}x^{n-k}. 
\end{equation}
Applying the multiplicative convolution by 
$$
\HGF{2}{0}{-n, \,1}{\cdot}{x} = \sum_{k=0}^n (-1)^k\, \binom{n}{k} (n-k)! \, x^{n-k}
$$  
on both sides of \eqref{reciprocal} and using the definition \eqref{coeffMultConv}, we get the identity
\begin{equation*}
  \HGF{2}{0}{-n, \,1}{\cdot}{x}  \boxtimes_n  q\left(\mathfrak d  \right)[x^n]    \simeq q^*.
\end{equation*}
Applying it to \eqref{additive} it yields 
\begin{equation*}
   p^*\simeq \HGF{2}{0}{-n,1}{\cdot}{x}
\boxtimes_n\left[\HGF{i_1+1}{j_1}{-n,\ \bm a_1}{\bm b_1}{ (-1)^{l_1}x} \boxplus_n
\HGF{i_2+1}{j_2}{-n,\ \bm a_2}{\bm b_2}{ (-1)^{l_2}x}\right],
\end{equation*}
which concludes the proof.
\end{proof}

\subsection{Real roots, interlacing, and free finite convolution} \label{sec:realrootednessfreeconv}\

A very important fact is that in many circumstances the finite free convolution of two polynomials with real roots also has all its roots real. Here, we use the notation introduced at the beginning of Section \ref{sec:prelim}. 

\begin{proposition}[Szeg\H{o} \cite{szego1922}, Walsh \cite{Walsh1922}, see also \cite{martinez2023real}]
    \label{prop:realrootedness}
    Let $p, q \in \P_n$. Then
    \begin{enumerate}[(i)]
        \item   $p,q\in \P_n(\rr) \ \Rightarrow \ p\boxplus_n q\in \P_n(\rr)$.
        \item  $p\in \P_n(\rr),\ q\in \P_n(\rr_{\geq 0}) \ \Rightarrow \ p\boxtimes_n q\in \P_n(\rr)$.
        \item  $p,q\in \P_n(\rr_{\geq 0}) \ \Rightarrow \ p\boxtimes_n q\in \P_n(\rr_{\geq 0})$.
        \item $p,q\in \P_n(\rr_{\leq 0}) \ \Rightarrow \ p\boxtimes_n q\in \P_n(\rr_{\geq 0})$
\item $p\in \P_n(\rr_{\leq 0}),\ q\in \P_n(\rr_{\geq 0}) \ \Rightarrow \ p\boxtimes_n q\in \P_n(\rr_{\leq 0})$.
    \end{enumerate}
\end{proposition}
If we replace above the sets $\R_{\ge 0}$ and $\R_{\le 0}$ by strict inclusions, $\R_{>0}$ and $\R_{<0}$, respectively, the statements of Proposition~\ref{prop:realrootedness} remain valid.

\begin{definition}[Interlacing] 
Let
$$
p(x)=e_0(p)\prod_{j=1}^n \left(x-\lambda_j(p)\right) \in \P_n(\R), \quad \lambda_1(p) \leq \dots \leq \lambda_n(p),
$$
and 
$$
q(x)=e_0(q)\prod_{j=1}^m \left(x-\lambda_j(q)\right) \in \P_m(\R), \quad \lambda_1(q) \leq \dots \leq \lambda_m(q).
$$
We say that $q$ \textbf{interlaces} $p$ (or, equivalently, that \textbf{zeros of $q$ interlace zeros of $p$}, see, e.g.,~\cite{MR3051099}), and denote it $p \preccurlyeq q$, if 
\begin{equation} \label{interlacing1}
    m=n \quad \text{and} \quad \lambda_1(p) \leq \lambda_1(q) \leq \lambda_2(p) \leq \lambda_2(q) \leq \cdots \leq  \lambda_n(p) \leq \lambda_n(q),
\end{equation}
or if
\begin{equation} \label{interlacing2}
    m=n-1 \quad \text{and} \quad \lambda_1(p) \leq \lambda_1(q) \leq \lambda_2(p) \leq \lambda_2(q) \leq \cdots \leq  \lambda_{n-1}(p) \leq \lambda_{n-1}(q) \le \lambda_n(p).
\end{equation}
Furthermore, we use the notation $p \prec q$ when all inequalities in \eqref{interlacing1} or \eqref{interlacing2} are strict.
\end{definition}

From the real-root preservation and the linearity of the free finite convolution one easily obtains the following interlacing-preservation property:
\begin{proposition}[Preservation of interlacing]
\label{lem:preservinginterlacingMult}
If $p, \widetilde{p} \in \P_n(\R)$ be of degree exactly $n$ such that $p \preccurlyeq \widetilde{p}$, then for a polynomial $q$ of degree $n$,
$$
q \in \P_n(\R) \quad  \Rightarrow \quad p\boxplus_n q \preccurlyeq \widetilde{p} \boxplus_n q
$$
and 
$$
q \in \P_n(\R_{\ge 0}) \quad  \Rightarrow \quad p\boxtimes_n q \preccurlyeq \widetilde{p} \boxtimes_n q.
$$
The same statements hold if we replace all $\preccurlyeq$ by $\prec$.
\end{proposition}
For a proof of this result, see \cite[Proposition 2.11]{martinez2023real}. 
 
\begin{remark}\label{preserveinterlacing.negativezeros}
Noticing that $p \preccurlyeq q$ if and only if $q(-x) \preccurlyeq p(-x)$ and that $p(x) \boxtimes_n q(-x)=[p\boxtimes_n q] (-x)$, we can easily extend Proposition \ref{lem:preservinginterlacingMult} to include polynomials with negative real roots. Namely, if $p, \widetilde{p}\in \monicpolreal$ and $p \preccurlyeq \widetilde{p}$, then for a polynomial $q$ of degree $n$, 
$$
q \in \P_n(\R_{\le 0}) \quad  \Rightarrow \quad \widetilde{p} \boxtimes_n q \preccurlyeq p \boxtimes_n q.
$$
\end{remark}

\subsection{Integral transforms and free convolution of measures} \

We follow the standard notation from the literature on free probability (see,  e.g.~\cite{arizmendi2021finite, MR4408504, MR3585560, MR2266879}) and define several transforms for a Borel measure $\mu$.

The \textbf{Cauchy transform} of $\mu$,
	\begin{equation} \label{def:Cauchy}
		\mathcal G_\mu(z)\isdef  \int \frac{d\mu(y)}{z-y},
	\end{equation}
is well defined and analytic on $\C \setminus \supp(\mu)$, even in the case of a signed measure. If $\mu$ is compactly supported, then $\mathcal G_\mu$ has a Laurent expansion
	\begin{equation}
	    \label{expansionCauchyTr}
     \mathcal G_\mu(z) = \sum_{k=0}^\infty \frac{m_k}{z^{k+1}} ,
	\end{equation}
convergent in a neighborhood of infinity, where 
	$$
	m_j =\int z^j \, d\mu(z), \quad j=0, 1, 2,\dots
	$$
are the \textbf{moments} of the measure $\mu$. Unless specified otherwise, we assume in what follows that $\mu$ is a compactly supported positive probability measure, so that
	\begin{equation}
	    \label{normCondition}
     m_0=1.
	\end{equation}

The \textbf{$K$-transform} of the compactly supported probability measure $\mu$ is the functional inverse of its Cauchy transform, that is,
	$$
	\mathcal K_\mu(w)\isdef   \mathcal G^{-1}_\mu(w),
	$$
analytic in a punctured neighborhood of the origin, while its \textbf{$R$-transform} 
	$$
	\mathcal R_\mu(w)\isdef    \mathcal K_\mu(w)-\frac{1}{w} =\sum_{n=0}^\infty \kappa_{n+1} w^n,
	$$
is analytic in a neighborhood of the origin\footnote{\, Clearly, the analyticity of $\mathcal R_\mu$ at the origin for a compactly supported measure is equivalent to condition \eqref{normCondition}.}. The coefficients $\kappa_{n}$ are called the \textbf{free cumulants} of $\mu$. Notice that $\mathcal R_\mu \equiv 0$ if and only if $$\mathcal G_\mu(z)=1/z$$ in a neighborhood of infinity; this is the case of $\mu=\delta_0$, the Dirac delta or unit mass point at the origin, but not only: any unit Lebesgue measure on a circle or a disk centered at the origin has this property.
	
Closely related to the Cauchy transform is the \textbf{generating function of the moments} (for short, the \textbf{$M$-transform}) of a compactly-supported probability measure $\mu$,
 \begin{equation}
     \label{defMtransform}
     	\mathcal M_\mu(z)\isdef \frac{1}{z}  \mathcal G_\mu\left(\frac{1}{z}\right)-1 = \int \frac{zy \, d\mu(y)}{1-zy},
 \end{equation}
 for which the expansion 
\begin{equation}
    \label{expansionMtransform}
    \mathcal M_\mu(z) = \sum_{k=1}^\infty  m_k \, z^{k} 
\end{equation}
converges in a neighborhood of the origin. 
We finally define the \textbf{$S$-transform} of $\mu $ as
	\begin{equation}
	    \label{defStranform}
     \mathcal S_\mu(w)\isdef   \frac{w+1}{w} \mathcal M_\mu^{-1}(w),
	\end{equation}
	where $\mathcal M_\mu^{-1}$ is the functional inverse of $\mathcal M_\mu$ (when it exists). Notice that once again $\mathcal M_\mu\equiv 0$ if and only if $\mathcal G_\mu(z)=1/z$; furthermore, if $\mu$ is compactly supported and $m_1\neq 0$ then $\mathcal S_\mu$ is analytic in a neighborhood of origin and nonvanishing at $w=0$, with $ \mathcal S_\mu(0) =1/m_1. $
We can guarantee that $m_1\ne 0$ if $\mu$ is a positive probability measure supported on the positive (or negative) semiaxis. 

All these transforms determine each other via the identities
\begin{equation}
\label{eq.moment.Rtransform}
\left(z\MM_\mu(z)+z\right)\RR_\mu \left(z\MM_\mu(z)+z\right)=\MM_\mu(z), \qquad \mathcal G_\mu \left( \frac{z+1}{z \mathcal S_\mu(z)}\right) = z \mathcal S_\mu(z),
\end{equation}
which follow directly from the definitions above. Combining \eqref{defStranform} and \eqref{eq.moment.Rtransform} we obtain also a direct relation between the $R$-transform and the $S$-transforms (see for instance \cite[Equation (16.8), Definition 18.15 and Remark 18.16]{MR2266879}):
\begin{equation}
  \label{RtransformStranform}
w \mathcal S_\mu(w)=\left( w \mathcal R_\mu(w)\right)^{-1},
\end{equation}
where $\left( w \mathcal R_\mu(w)\right)^{-1}$ is the functional inverse of $w\mathcal R_\mu(w)$ (when it exists).

We will later use the well-known fact, see for instance \cite[Proposition 3.13]{HS07}, that for a measure $\mu$ such that $\mu(\{0\})=0$ it holds that
\begin{equation}
    \label{measureReversed}
S_\mu(z) S_{\mu^{*}}(-z-1)=1,  
\end{equation}
where $\mu^{*}$ is the \textbf{reversed measure} of $\mu$ (equivalently, the push-forward measure under the mapping $t \mapsto 1/t$), so that its density is $d\mu^*(t)=  t^{-2}d\mu( t^{-1})$). The identity \eqref{measureReversed} is valid in a punctured neighborhood of origin, in the domain of analyticity of both functions on the left-hand side.

\begin{remark}
    \label{rem:momentsequences}
When $\mu$ is not compactly supported, expressions \eqref{def:Cauchy} and \eqref{defMtransform} make sense and define analytic functions in the open set $\C\setminus \supp (\mu)$ and in its image by $z\mapsto 1/z$, respectively. Formula \eqref{defStranform} is also well-posed, at least where $\mathcal M_\mu$ is locally invertible. In this case, we cannot expect the series in \eqref{expansionCauchyTr} and \eqref{expansionMtransform} to converge.

On the other hand, the formulas in \eqref{expansionCauchyTr} and \eqref{expansionMtransform} can be considered as formal power series associated with a ``moment'' sequence $ m_0, m_1, \dots,  $ of complex numbers. In that case, if we denote \textit{the sequence} by $\mu$, we can use these formulas as a \textit{definition} of $\mathcal G_{\mu}$ and $\mathcal M_{\mu}$, as well as define $\mathcal R_{\mu} $ and $\mathcal S_\mu$ through \eqref{eq.moment.Rtransform}. In other words, all the transforms can be defined for any complex sequence but only as formal power series (also called moment series), with no analyticity or convergence assumed a priori. 

As we will see in Section~\ref{sec:finiteinAsymptotics}, the advantage of this approach is that many asymptotic results can be established even for formal moment sequences, without assuming any underlying measure. 
\end{remark}

Using these definitions, we can introduce two important operations on compactly supported probability measures $\mu$ and $\nu$ (or, as we have just discussed, on their moment sequences):
\begin{enumerate}
	\item[(i)] the \textbf{free additive convolution} $\mu \boxplus \nu$ via the identity
	\begin{equation}
	    \label{def:additiveconvolutionmeasures}
     \mathcal R_{\mu \boxplus \nu}(w) = \mathcal R_{\mu  }(w) +  \mathcal R_{ \nu}(w) ;
	\end{equation}
	\item[(ii)] the \textbf{free multiplicative convolution} $\mu \boxtimes \nu$ via the identity
	\begin{equation}
	    \label{def:multiplicativeconvolutionmeasures}
	\mathcal S_{\mu \boxtimes \nu}(w) = \mathcal S_{\mu  }(w)  \mathcal S_{ \nu}(w) .
	\end{equation}
\end{enumerate} 
Clearly, identities \eqref{def:additiveconvolutionmeasures}--\eqref{def:multiplicativeconvolutionmeasures} do not define the resulting measures $\mu \boxplus \nu$ and $\mu \boxtimes \nu$ uniquely, unless they are determined by the analytic expression of their $R-$ and $S-$transforms, respectively. This is the case of probability measures compactly supported on $\R$.

%%%%%%%%%%%%%%%%%%%%%%%%%%%%%%%%%%%%%%%%%%%%%%%%%%%%%%
\section{General results} \label{sec:generalresults}
%%%%%%%%%%%%%%%%%%%%%%%%%%%%%%%%%%%%%%%%%%%%%%%%%%%%%%

\subsection{Kampé de Fériet polynomials and finite free convolution} \label{sec:KdF}\

Recall the definition of the Kampé de Fériet polynomials in \eqref{KdFpolyns.mult}. 
The following factorization holds:
\begin{theorem} \label{thm:multipleKdF}
For $k=0,1,\dots, r$, let $i_k,j_k\in \N$, $c_k\neq 0$, and $\bm a_k\in \rr^{i_k}$, $\bm b_k\in \rr^{j_k}$. Then, for $n\in \N$,
\begin{align*}
   &  F_{j_0: j_1 ;\dots; j_r}^{i_0+1: i_1 ;\dots; i_r}\left[\left.\begin{array}{c}
-n,\bm a_0:\bm a_1 ; \dots ; \bm a_r \\
\bm b_0:\bm b_1 ;\dots ; \bm b_r
\end{array} \right\rvert\, c_1 x,c_2x,\dots,c_rx \right] 
  \simeq \left[q_0\boxtimes_n \left(q_1 \boxplus_n \cdots \boxplus_n q_r \right) \right]^*(x), 
\end{align*}
where
\begin{align*}
q_0(x)& \isdef \HGF{j_0+1}{i_0}{-n, -\bm b_0-n+1}{ -\bm a_0-n+1}{(-1)^{i_0+j_0}x},\\
q_l(x)& \isdef \HGF{j_l+1}{i_l}{-n, -\bm b_l-n+1}{ -\bm a_l-n+1}{(-1)^{i_l+j_l}\frac{x}{c_l}}, \quad \text{for }l=1,\dots,r .
\end{align*}
The proportionality constant can be computed explicitly in terms of the parameters.
\end{theorem}
\begin{proof}
Using the notation in \eqref{defFMULTCONV}, we can write the coefficients of $q_0,q_1,\dots, q_r$ (up to a multiplicative constant) as
$$
e_k(q_0)=\binom{n}{k}\frac{ \raising{\bm a_0}{k}}{\raising{\bm b_0}{k}}, \qquad 
e_k(q_l)=\binom{n}{k}\frac{ \raising{\bm a_l}{k}} { \raising{\bm b_l}{k}} \, c_l^k  \qquad \text{for }l=1,\dots, r.
$$
By \eqref{coeffAdditiveConv}, 
$$
e_k(q_1\boxplus_n \cdots \boxplus_n q_r)= \falling{n}{k} \sum_{l_1+\dots+l_r=k} \frac{c_1^{l_1} \raising{\bm a_1}{l_1} } {l_1! \raising{\bm b_1}{l_1}} \cdots \frac{c_r^{l_r} \raising{\bm a_r}{l_r} } {l_r! \raising{\bm b_r}{l_r}} ,
$$
and by \eqref{coeffMultConv}, 
$$e_k(p\boxtimes_n(q_1\boxplus_n \cdots \boxplus_n q_r))= \frac{ \raising{\bm a_0}{k}}{\raising{\bm b_0}{k}}\falling{n}{k} \sum_{l_1+\dots+l_r=k} \frac{c_1^{l_1} \raising{\bm a_1}{l_1} } {l_1! \raising{\bm b_1}{l_1}} \cdots \frac{c_r^{l_r} \raising{\bm a_r}{l_r} } {l_r! \raising{\bm b_r}{l_r}} .
$$
Thus, 
\begin{align*}
p\boxtimes_n(q_1\boxplus_n \cdots \boxplus_n q_r) 
&\simeq  \sum_{k=0}^n x^{n-k}(-1)^k \frac{ \raising{\bm a_0}{k}}{\raising{\bm b_0}{k}} \falling{n}{k}  \sum_{l_1+\dots+l_r=k} \frac{c_1^{l_1} \raising{\bm a_1}{l_1} } {l_1! \raising{\bm b_1}{l_1}} \cdots \frac{c_r^{l_r} \raising{\bm a_r}{l_r} } {l_r! \raising{\bm b_r}{l_r}}  \\
&= \sum_{k=0}^n x^{n-k} \sum_{l_1+\dots+ l_r= k}  \frac{\raising{-n}{k} \raising{\bm a_0}{k}}{\raising{\bm b_0}{k}}\frac{c_1^{l_1} \raising{\bm a_1}{l_1} } {l_1! \raising{\bm b_1}{l_1}} \cdots \frac{c_r^{l_r} \raising{\bm a_r}{l_r} } {l_r! \raising{\bm b_r}{l_r}} .
\end{align*}
The result follows from \eqref{KdFpolyns.mult} and the definition of the reversed polynomials \eqref{defReversed}.
\end{proof}

Another representation is obtained if we fix all but one variable of a Kampé de Fériet polynomial:
\begin{theorem} \label{thm:multipleKdFfactorization}
For $k=0,1,\dots, r$, let $i_k,j_k\in \N$,  $c_k\neq 0$, and  $\bm a_k\in \rr^{i_k}$, $\bm b_k\in \rr^{j_k}$. Then, for $n\in \N$, 
\begin{align*}
   &  F_{j_0: j_1 ;\dots; j_r}^{i_0+1: i_1 ;\dots; i_r}\left[\left.\begin{array}{c}
-n,\bm a_0:\bm a_1 ; \dots ; \bm a_r \\
\bm b_0:\bm b_1 ;\dots ; \bm b_r
\end{array} \right\rvert\, c_1 x,c_2,\dots,c_r \right] 
  \simeq \left[q_1 \boxtimes_n (q_0\boxplus_n q_2\boxplus_n \cdots \boxplus_n q_r)\right](x), 
\end{align*}
where
\begin{align*}
q_0(x)& \isdef \HGF{i_0+1}{j_0}{-n, \bm a_0}{ \bm b_0}{-x}, \qquad  q_1(x) \isdef \HGF{i_1+1}{j_1}{-n, \bm a_1}{ \bm b_1}{-c_1x},\quad \text{and} \\
q_l(x)& \isdef \HGF{j_l+1}{i_l}{-n, -\bm b_l-n+1}{ -\bm a_l-n+1}{(-1)^{i_l+j_l}\frac{x}{c_l}}  \qquad \text{for }l=2,\dots, r.
\end{align*}
The proportionality constant can be computed explicitly in terms of the parameters.
\end{theorem}
\begin{proof}
Using the notation in \eqref{defFMULTCONV}, we can write the coefficients of $q_0,q_1,\dots, q_r$, (up to a multiplicative constant) as
\begin{align*}
e_k(p)&=\binom{n}{k}(-1)^{n-k}\frac{ \raising{\bm a_0}{n-k}}{\raising{\bm b_0}{n-k}}, \qquad 
e_k(q_1) =\binom{n}{k}(-1)^k\frac{c_1^{n-k} \raising{\bm a_1}{n-k}}{\raising{\bm b_1}{n-k}}, \\
e_k(q_l)&=\binom{n}{k}\frac{ c_l^k  \raising{\bm a_l}{k}} { \raising{\bm b_l}{k}} \qquad \text{for }l=2,\dots, r.
\end{align*}
By \eqref{coeffAdditiveConv}, 
\begin{align*}
e_l(q_0\boxplus_n q_2\boxplus_n \cdots \boxplus_n q_r)&= \falling{n}{l} \sum_{k+l_2+\dots+l_r=l} \frac{(-1)^{n-k}\raising{\bm a_0}{n-k}}{k!\raising{\bm b_0}{n-k}}\frac{ c_2^{l_2}\raising{\bm a_2}{l_2}} {l_2! \raising{\bm b_2}{l_2}} \cdots \frac{ c_r^{l_r}\raising{\bm a_r}{l_r}} {l_r! \raising{\bm b_r}{l_r}} \\
&= \frac{1}{(n-l)!} \sum_{k+l_2+\dots+l_r=l} \frac{\raising{-n}{n-k}\raising{\bm a_0}{n-k}}{\raising{\bm b_0}{n-k}}\frac{ c_2^{l_2}\raising{\bm a_2}{l_2}} {l_2! \raising{\bm b_2}{l_2}} \cdots \frac{ c_r^{l_r}\raising{\bm a_r}{l_r}} {l_r! \raising{\bm b_r}{l_r}}.
\end{align*}
From \eqref{coeffMultConv} we get that 
\begin{align}
&e_l\left[q_1 \boxtimes_n (q_0\boxplus_n q_2\boxplus_n \cdots \boxplus_n q_r)\right] \nonumber \\
&=\frac{ (-1)^lc_1^{n-l}}{(n-l)!}\frac{ \raising{\bm a_1}{n-l}}{\raising{\bm b_1}{n-l}}\sum_{k+l_2+\dots+l_r=l} \frac{\raising{-n}{n-k}\raising{\bm a_0}{n-k}}{\raising{\bm b_0}{n-k}}\frac{ c_2^{l_2}\raising{\bm a_2}{l_2}} {l_2! \raising{\bm b_2}{l_2}} \cdots \frac{ c_r^{l_r}\raising{\bm a_r}{l_r}} {l_r! \raising{\bm b_r}{l_r}} .\label{eq.proof.KdeF.general}
\end{align}

On the other hand, using the definition of the multivariable Kampé de Fériet polynomials \eqref{KdFpolyns.mult} and a change in the order of the summation, we obtain
\begin{align*}
& F_{j_0: j_1 ;\dots; j_r}^{i_0+1: i_1 ;\dots; i_r}\left[\left.\begin{array}{c}
-n,\bm a_0:\bm a_1 ; \dots ; \bm a_r \\
\bm b_0:\bm b_1 ;\dots ; \bm b_r
\end{array} \right\rvert\, c_1 x,c_2,\dots,c_r \right] \\
&=\sum_{l_1+\dots+ l_r+m=n}  \frac{\raising{-n}{n-m} \raising{\bm a_0}{n-m}}{\raising{\bm b_0}{n-m}}\frac{ \raising{\bm a_1}{l_1}x^{l_1}c_1^{l_1}} { \raising{\bm b_1}{l_1}l_1!}  \frac{ c_2^{l_2}\raising{\bm a_2}{l_2}} {l_2! \raising{\bm b_2}{l_2}} \cdots \frac{ c_r^{l_r}\raising{\bm a_r}{l_r}} {l_r! \raising{\bm b_r}{l_r}}  \\
&= \sum_{l_1=0}^n\frac{x^{l_1}c_1^{l_1}}{l_1!} \frac{ \raising{\bm a_1}{l_1}} { \raising{\bm b_1}{l_1}} 
\sum_{m+l_2+\dots+ l_r= n-l_1}  \frac{\raising{-n}{n-m} \raising{\bm a_0}{n-m}}{\raising{\bm b_0}{n-m}} \frac{ c_2^{l_2}\raising{\bm a_2}{l_2}} {l_2! \raising{\bm b_2}{l_2}} \cdots \frac{ c_r^{l_r}\raising{\bm a_r}{l_r}} {l_r! \raising{\bm b_r}{l_r}}.
\end{align*}
The coefficients of this polynomial, written in the form \eqref{defFMULTCONV}, match those in \eqref{eq.proof.KdeF.general} after changing variables $l=n-l_1$. The assertion follows.
\end{proof}
\begin{remark}
A straightforward consequence of both factorizations is that if $(-1)^{i_0+j_0}=(-1)^{i_1+j_1}$, then we get the following relation:

$$\left( F_{j_0: j_1 ;\dots; j_r}^{i_0+1: i_1 ;\dots; i_r}\left[\left.\begin{array}{c}
-n,\bm a_0:\bm a_1 ; \dots ; \bm a_r \\
\bm b_0:\bm b_1 ;\dots ; \bm b_r
\end{array} \right\rvert\,(-1)^{i_1+j_1}c_1 x,c_2x,\dots,c_r x\right] \right)^*$$
$$\simeq F_{i_1: i_0 ;j_2;\dots; j_r}^{j_1+1: j_0 ;i_2;\dots; i_r}\left[\left.\begin{array}{c}
-n,-\bm b_1-n+1:-\bm b_0-n+1 ; \bm a_2 ; \dots ; \bm a_r \\
-\bm a_1-n+1:-\bm a_0-n+1  ; \bm b_2\dots ; \bm b_r
\end{array} \right\rvert\,(-1)^{i_1+j_1}\frac{x}{c_1},\frac{c_2}{c_1},\dots,\frac{c_r}{c_1} \right] .
$$
By the symmetry in the variables of the polynomial, and using the previous relation twice, we get that if $(-1)^{i_0+j_0}=(-1)^{i_1+j_1}=(-1)^{i_2+j_2}$, then
$$F_{i_1: i_0 ;j_2;\dots; j_r}^{j_1+1: j_0 ;i_2;\dots; i_r}\left[\left.\begin{array}{c}
-n,-\bm b_1-n+1:-\bm b_0-n+1 ; \bm a_2 ; \dots ; \bm a_r \\
-\bm a_1-n+1:-\bm a_0-n+1  ; \bm b_2\dots ; \bm b_r
\end{array} \right\rvert\,\frac{x}{c_1},\frac{c_2}{c_1},\dots,\frac{c_r}{c_1} \right] 
$$
$$\simeq F_{i_2: i_0; j_1; j_3;\dots; j_r}^{j_2+1: ;j_0 ; i_1 ; i_3;\dots; i_r}\left[\left.\begin{array}{c}
-n,-\bm b_2-n+1:-\bm b_0-n+1 ; \bm a_1  ; \bm a_3; \dots ; \bm a_r \\
-\bm a_2-n+1:-\bm a_0-n+1  ; \bm b_1 ; \bm b_3;\dots ; \bm b_r
\end{array} \right\rvert\,\frac{x}{c_2},\frac{c_1}{c_2},\dots,\frac{c_r}{c_2}\right] .
$$
We are not aware of this kind of relations in the literature.
\end{remark}

\subsection{Finite free convolutions in the asymptotic regime} \label{sec:finiteinAsymptotics}
\

We have already seen that free finite convolution is useful when studying the zeros of polynomials. An additional crucial advantage is that finite free convolutions tend to free convolution in the asymptotic regime when the degree $n\to \infty$.

The connection between the convolutions of polynomials and free probability (reason for the name of ``finite free'' convolutions) was first noticed by Marcus, Spielman, and Srivastava in \cite{MR4408504} when they used Voiculescu's $R$- and $S$-transform to improve the bounds on the largest root of a convolution of two real-rooted polynomials. This connection was explored further in \cite{marcus}, where Marcus defined a finite analog of the $R$- and $S$-transform that are related to Voiculescu's transforms in the limit. Using finite free cumulants, Arizmendi and Perales \cite{arizmendi2018cumulants} showed that in the asymptotic regime, a finite free additive convolution becomes a free additive convolution. This was later extended to the multiplicative convolution by Arizmendi, Garza-Vargas and Perales \cite{arizmendi2021finite}.

Since the proofs of these facts in \cite{arizmendi2018cumulants,arizmendi2021finite} are essentially algebraic, it turns out that they are applicable in a broader context of moment sequences instead of measures; see Remark~\ref{rem:momentsequences}. More precisely, we need to consider the normalized zero-counting measures for a sequence of polynomials, real-rooted or not, under the only assumptions that their moments converge.

Given a polynomial $p$ of degree $n$ and roots $\lambda_j(p)$, $j=1, \dots, n$ (not necessarily all distinct or real),  its (normalized) \textbf{zero counting measure} (or \textbf{empirical root distribution} of $p$) is
\begin{equation}
    \label{def:zerocountingmeasure}
    \chi (p)\isdef \frac{1}{n}\sum_{j=1}^n \delta_{\lambda_j(p)},
\end{equation}
where $\delta_z$ is the Dirac delta (unit mass) placed at the point $z$. The corresponding moments of $\chi(p)$ (which we also call the moments of $p$, stretching the terminology a bit) are
$$
m_k(p)\isdef \frac{1}{n}\sum_{j=1}^n \lambda_j^k(p)=\int x^k \, d\chi(p), \quad k=0, 1, 2,\dots
$$

As mentioned above, the connection between finite and standard free probability is revealed in the asymptotic regime, when we let the degree $n\to \infty$. We say that the sequence of polynomials $ \mathfrak p= \left(p_n\right)_{n=1}^{\infty}$ such that each $p_n$ is real-rooted and of degree exactly $n$ \textbf{converges in moments} if all finite limits
\begin{equation}
    \label{def:convergenceMoments}
m_k=\lim_{n\to\infty} m_k(p_n), \qquad  k= 0, 1,2,\dots
\end{equation}
exist. If the supports of $\chi(p_n)$ are contained in a compact set of the complex plane, then the sequence $\chi(p_n)$ is weakly compact and $m_0, m_1, \dots$ are the moment of any of its accumulation points. In this case, \eqref{def:convergenceMoments} is equivalent to existence of a probability measure $\nu(\mathfrak p)$, compactly supported on $\C$,  with all its moments finite such that
$$
\lim_{n\to\infty} m_k(p_n)=m_k(\nu (\mathfrak p)), \qquad  k= 0, 1,2,\dots
$$
If we know additionally that $\chi(p_n)$ are on $\R$, with their supports contained in the same compact, and if the moment problem for $\nu (\mathfrak p)$ is determined, then \eqref{def:convergenceMoments} is equivalent to the weak-* convergence of the sequence $\chi (p_n)$ to $\nu (\mathfrak p)$.

For real-rooted polynomials, the following proposition is a direct consequence of \cite[Corollary 5.5]{arizmendi2018cumulants} and \cite[Theorem 1.4]{arizmendi2021finite}:

\begin{proposition}  
    \label{prop:finiteAsymptotics}
    Let $\mathfrak p:=\left(p_n\right)_{n=1}^{\infty}$ and $\mathfrak  q:=\left(q_n\right)_{n=1}^{\infty}$ be two sequences of real-rooted polynomials as above, and let $\nu(\mathfrak p)$ and $\nu(\mathfrak q)$ be two compactly supported probability Borel measures on $\R$ such that $\mathfrak p$ (respectively, $\mathfrak q$) converges in moments to $\nu(\mathfrak p)$ (respectively, $\nu(\mathfrak q)$). Then
    \begin{enumerate}
        \item[(i)]  the sequence $\left(p_n \boxplus_n q_n\right)_{n=1}^{\infty}$ converges in moments to $\nu(\mathfrak p) \boxplus \nu(\mathfrak q)$;
        \item[(ii)] if, additionally, for all sufficiently large $n$, $p_n, q_n\subset \pp_n(\rr_{>0})$ (or if $p_n, q_n\subset \pp_n(\rr_{<0})$) then the sequence $\left(p_n \boxtimes_n q_n\right)_{n=1}^{\infty}$ converges in moments to $\nu(\mathfrak p) \boxtimes \nu(\mathfrak q)$.
    \end{enumerate}
\end{proposition}
In other words, in the case of real-rooted polynomials, as $n\to \infty$, we can replace the finite free convolution of polynomials by the standard free convolution of measures supported on $\R$. The motivation to consider only real-rooted polynomials is given by their application in free probability. However, the proof of Proposition~\ref{prop:finiteAsymptotics} is basically algebraic, based on the notion of convergence in moments, and can be easily extended to more general cases, such as polynomials with uniformly bounded, but not necessarily real, roots. 

Namely, we have the following result:
\begin{theorem}
\label{prop:finiteAsymptotics.updated}
Let $\mathfrak p:=\left(p_n\right)_{n=1}^{\infty}$ and $\mathfrak  q:=\left(q_n\right)_{n=1}^{\infty}$ be two sequences of polynomials, such that $\deg p_n = \deg q_n=n$, $n\in \Z_{\ge 0}$, and all zeros of both $\mathfrak p$ and $\mathfrak q$ are uniformly bounded.

If, additionally, all finite limits
$$
\lim_{n\to\infty} m_k(p_n)=\alpha_k\in \C \qquad \text{and}\qquad \lim_{n\to\infty} m_k(q_n)=\beta_k\in \C, \qquad  k= 0, 1,2,\dots,
$$
exist, then
\begin{enumerate}
\item[(i)] all moments of the sequence of polynomials $\left(p_n \boxplus_n q_n\right)_{n=1}^{\infty}$ have a finite limit,
$$
\gamma_k\isdef \lim_{n\to\infty} m_k(p_n \boxplus_n q_n)\in \cc  \qquad k= 0, 1,2,\dots,
$$
and the $R$-transform associated to the sequence $\left(\gamma_k\right)_{k=1}^{\infty}$
satisfies
$$\RR_\gamma(z)=\RR_{\alpha}(z)+\RR_{\beta}(z).$$
\item[(ii)] All moments of the sequence $\left(p_n \boxtimes_n q_n\right)_{n=1}^{\infty}$ have a limit
$$
\theta_k:=\lim_{n\to\infty} m_k(p_n \boxtimes_n q_n)\in\cc \qquad k= 0, 1,2,\dots,
$$
and the $S$-transform associated to the sequence $\left(\theta_k\right)_{k=1}^{\infty}$
satisfies
$$\SS_\theta(z)=\SS_{\alpha}(z)\SS_{\beta}(z).$$
\end{enumerate}
\end{theorem}
\begin{remark}
    For the notion of the $S$- and $R$-transforms associated to a sequence, see Remark \ref{rem:momentsequences}.
\end{remark}

This theorem can be established following the arguments leading to Proposition \ref{prop:finiteAsymptotics}. As noted, the proofs of the results in \cite{arizmendi2018cumulants, arizmendi2021finite} rely on the fact that the combinatorial structure behind the finite convolutions of polynomials tends (as $n\to \infty$) to the combinatorial formulas that relate the coefficients of the series in the $M$-, $R$-, and $S$-transforms of the limit sequence. This fact is established regardless of whether this sequence is in fact a moment sequence of a probability measure. When this is the case, the notions of the $R$- and other transforms for a sequence and for the underlying measure match, see,  for instance, \cite[Lectures 16 and 18]{MR2266879}.

The reader interested in further details can check Appendix~\ref{sec:appendix}.

\subsection{{S}-transform for a hypergeometric polynomial} \label{sec:Stransform}

In this section, we show that although the Cauchy transform of any limit of zero-counting measures of hypergeometric polynomials are, generally speaking, algebraic functions whose explicit expressions are not available, their $S$-transforms are straightforward rational functions.

Let us first discuss the case of $_2F_1$ polynomials.
\begin{proposition}
    \label{prop:asymptotics2F1}
  Let  
$$
p_n(x) = \HGF{2}{1}{-n  , \alpha_{n} }{\beta_n}{x},
$$
be a sequence of hypergeometric polynomials such that the finite limits
$$
\lim_n \frac{\alpha_n}{n}=A, \quad \lim_n \frac{\beta_n}{n}=B
$$
exist. If additionally
\begin{equation}
    \label{mainassumptionsCase2F1}
    A  \notin [-1,0),   \quad B  \neq -1, \quad \text{and} \quad A\neq B,
\end{equation}
then any weak-* limit $\mu$ of the normalized zero-counting measures $\chi(p_n)$ is a positive probability measure compactly supported on $\C$, for which in a neighborhood of the origin,
\begin{equation}
    \label{Sfor2F1}
    \mathcal S_\mu(z)= \frac{ z+A+1}{ z+B+1  }.
\end{equation}
\end{proposition}
\begin{proof}
    Denote by $\mu_n =\chi(p_n)$ the probability zero-counting measure of $p_n$.

We use well-known identities expressing $_2 F_1$ polynomials in standard normalization in terms of Jacobi polynomials,
\begin{align} \label{HPJacobiPfla1}
\raising{\beta}{n}\HGF{ 2}{1}{-n, \alpha}{\beta }{x} 
& =  n!  \, P^{(\beta-1,-n+\alpha-\beta)}_{n}(1-2x) =(-1)^n  n! \, P^{(-n+\alpha-\beta,\beta-1)}_{n}(2x-1) .
\end{align}
Jacobi polynomials also exhibit several transformation formulas in the cases when they can have multiple zeros (at $\pm 1$) or have a degree reduction, such as 
$$
P_n^{(-k-1, \beta)}(z)=\frac{\Gamma(n+\beta+1)}{\Gamma(n+\beta-k)} \frac{(n-k-1) !}{n !}\left(\frac{z-1}{2}\right)^{k+1} P_{n-k-1}^{(k+1, \beta)}(z), \qquad k\in \Z_n,
$$
and several more. For a more detailed discussion, see \cite[\S 4.22]{szego1975orthogonal} or \cite{MR2124460}. From these identities it follows that the assumption $A\notin [-1,0)$ is sufficient to guarantee that the zeros of $p_n$ are uniformly bounded, while with $B\neq -1$, its zero-counting measures $\mu_n$ do not collapse to $\delta_0$. 

Thus, standard arguments on weak compactness of the sequence $ \mu_n$ show that under these assumptions, there exist $\Lambda \subset \mathbb{N}$ and a unit measure $\mu\neq \delta_0$ such that
$$
\mu_n \xrightarrow{*} \mu  , \quad n \in \Lambda .
$$
As a consequence, $ \mathcal G_{\chi(p_n)} \to  \mathcal G_{\mu}$, $n \in \Lambda$, in a neighborhood of infinity.

Polynomials $p_n$ satisfy the ODE
$$
z(1-z) p_n''(z) + \left( \beta_n - (\alpha_n-n+1)z\right)p_n'(z)+n\alpha_n p_n(z)=0.
$$
Rewriting it in terms of $h_n=p_n'/p_n$ and noticing that $h_n/n \to G = \mathcal G_{\mu}$, we get an equation for $G$,
\begin{equation}
    \label{eq:algebraic2F1forG}
    z(1-z) G^2(z) + \left(  B - (A- 1)z\right)G(z)+ A=0.
\end{equation}
Observe again that for $B\neq -1$, $G(z)\ne 1/z$. 

With the notation $w=\mathcal S_{\mu}(z)$, the second identity in \eqref{eq.moment.Rtransform} takes the form
\begin{equation}
    \label{StransformAlt}
    G  \left( \frac{z+1}{z w}\right) = z w.
\end{equation}
Thus, making the change of variable $z\mapsto (z+1)/(w z)$ in \eqref{eq:algebraic2F1forG} and using the identity above, we arrive at 
$$
z (w (z+B+1)-(z+A+1))=0,
$$
or
$$
w = \frac{ z+A+1}{ z+B+1  },
$$
which proves \eqref{Sfor2F1}.
\end{proof}

\begin{remark}
    \label{rem:pathologicalcases}
It is interesting to observe that expression \eqref{Sfor2F1} is meaningful even if we drop the assumptions \eqref{mainassumptionsCase2F1}. It is worth discussing then the possible consequences of relaxing these restrictions.

As we have seen in the proof, if $B=-1$, then $\mathcal G_\mu(z)=1/z$, which happens, for instance, if $\mu=\delta_0$. If the zeros of $p_n$'s are real, it means that they asymptotically collapse to the origin.

The consequence of dropping the assumption $A\neq B$ is also clear: if $A=B$, we have $\mathcal G_\mu(z)=1/(z-1)$, which happens, for example, if $\mu=\delta_1$.

If $-1<A<0$, the explicit formula \eqref{consequenceconnectionRandL2} shows that selecting the sequence $\alpha_n$ appropriately, we can make the degree of $p_n$ be of $\mathcal O( -An)$. This would mean that non-trivial part of the measures $\mu_n$ ``escapes'' to infinity. Any possible weak-* limit could have a bounded support, but it will be of a total mass $|A|$. Notice that it would mean that for this $\mu$, $m_0=-A<1$, which is compatible with the algebraic equation \eqref{eq:algebraic2F1forG}.
  
Finally, if $A=-1$, the $S$-transform in \eqref{Sfor2F1} is still rational, but now it vanishes at $z=0$. In this case, one of the branch points of the Cauchy transform $\mathcal G_\mu$, solution of \eqref{eq:algebraic2F1forG}, is at infinity.  This shows that if the limit measure $\mu$ exists, it is a probability measure with unbounded support. 
\end{remark}

Now we formulate the general result of this section:
\begin{theorem}
    \label{thm:StranfHyper2F1}
    For $s,t \in \Z_{\ge 0}$, let $\bm a_n \in \R^s$, $\bm b_n\in \R^t$ be such that finite limits
\begin{equation}
    \label{assumptionlimits}
    \lim_n \frac{\bm a_n}{n}=\bm A=(A_1,\dots A_s)\in \R^s, \quad \lim_n \frac{\bm b_n}{n}=\bm B=(B_1,\dots, B_t)\in \R^t
\end{equation}
exist. 
Assume additionally that
\begin{equation}
    \label{mainassumptionsAB}
    A_j \notin [-1,0),  \quad B_k \neq -1, \quad A_j \neq B_k, \qquad j\in \{1, \dots, s\}, \quad k\in \{ 1, \dots, t\}.
\end{equation}
Then any weak-* limit $\mu$ of the normalized zero-counting measures $\chi(p_n)$ of the sequence
\begin{equation}
    \label{sequencePn}
      p_n(x)=\HGF{s+1}{t}{-n, \bm a_n}{\bm b_n }{ n^{t-s} x}
\end{equation}
is a positive probability measure compactly supported on $\C$, for which in a neighborhood of the origin,
\begin{equation}
    \label{SfunctionForHyper}
   \mathcal S_\mu(z)= \frac{ \prod_{i=1}^s  (z+A_i+1)}{ \prod_{j=1}^t (z+B_j+1)  }.
\end{equation}
\end{theorem}
\begin{proof}
Consider the sequence
 $$
q_n(x) = \HGF{1}{1}{-n    }{\beta_n}{x},
$$
such that
$$
 \lim_n \frac{\beta_n}{n}=B.
$$
Polynomials $q_n$ are, in fact, Laguerre polynomials, 
$$
q_n(x) = \frac{n!}{\raising{\beta_n}{n}}\, L_n^{(\beta_n-1)}(x).
$$
We are interested in the sequence $\mu_n =\chi(\widehat q_n)$ of probability zero-counting measures for rescaled polynomials
$$
\widehat q_n(x) = \HGF{1}{1}{-n    }{\beta_n}{n x}.
$$
It is well known that for each fixed $\beta$, the sequence of monic Laguerre polynomials $\left\{\widetilde L_n^{\left(\beta\right)}\right\}, n \in \mathbb{N}$, satisfies the three-term recurrence relation
$$
\widetilde L_{n+1}^{\left(\beta\right)}(x)=\left(x-(2 n+\beta+1) \right) \widetilde L_n^{\left(\beta\right)}(x)+ n (n+\beta) \widetilde L_{n-1}^{\left(\beta\right)}(x).
$$
In particular, the monic polynomials
$$
n^{-n}\widetilde L_n^{(\beta-1)}(nx)
$$
satisfy
$$
n^{-n-1} L_{n+1}^{\left(\beta-1\right)}(nx)=\left(x-\frac{2 n+\beta}{n} \right) n^{-n} L_n^{\left(\beta-1\right)}(nx)+   \frac{n+\beta-1}{n} \left( n^{-n+1} L_{n-1}^{\left(\beta-1\right)}(nx)\right) .
$$
Writing these relations as an eigenvalue problem for a Jacobi matrix and applying  Gershgorin's theorem, it is easy to prove that under the assumption that the finite limit $B$ of $\beta_n/n$ exists, the zeros of  $\widehat q_n$  are uniformly bounded. Moreover, the well-known identity
$$
L_n^{(-k)}(x)=\frac{(n-k) !}{n !}(-x)^{k} L_{n-k}^{(k)}(x), \quad k\in \Z_n,
$$
see e.g.~\cite[Section 18.5]{MR2723248} or \cite[Section 2.3]{martinez2023real}, indicates that $\mu_n$'s do not converge to $\delta_0$, unless $B\neq -1$. This is indeed the case; see the discussion in \cite{MR1858305} or \cite{MR1931612}. 

Polynomials $q_n$ are solution of Kummer's differential 
$$
z q_n''(z)+(\beta_n-z) q_n'(z)+n q_n(z)=0,
$$
see, e.g.~\cite[Chapter 7]{MR2682403}. Hence, $\widehat q_n$ 
satisfy
\begin{equation}
    \label{odeLaguerre}
    z \widehat q_n''(z)+(\beta_n-nz) \widehat q_n'(z)+n^2 \widehat q_n(z)=0.
\end{equation}

As before, by weak compactness there exist $\Lambda \subset \mathbb{N}$ and a unit measure $\mu$ such that
$$
\mu_n \xrightarrow{*} \mu, \quad n \in \Lambda .
$$
As a consequence, $ \mathcal G_{\chi(q_n)} \to  \mathcal G_{\mu}$, $n \in \Lambda$, in a neighborhood of infinity. Rewriting \eqref{odeLaguerre} in terms of $h_n=\widehat q_n'/ \widehat q_n$ and using that $h_n/n \to G = \mathcal G_{\mu}$, we get an equation for $G$,
$$
z  G^2(z) + \left(  B -  z\right)G(z)+ 1=0
$$
(notice again that for $B\neq -1$, $G(x)\ne 1/x$). Now we proceed as in the proof of Proposition~\ref{prop:asymptotics2F1}: making the change of variable $z\mapsto (z+1)/(w z)$ in this quadratic equation on $G$ and using \eqref{StransformAlt}, we get for $w=\mathcal S_{\mu}(z)$ the equation
$$
z \left(w (z+B+1)-1 \right) =0,
$$
and we conclude that 
\begin{equation}
    \label{Sfor1F1}
  \mathcal S_\mu(z)= \frac{ 1}{ z+B+1  }.
\end{equation}

Same arguments, or using identity \eqref{measureReversed} and the fact that  
$$
\HGF{2}{0}{-n , \alpha_n   }{\cdot}{x} = (-x)^n \raising{ \alpha_n }{n}\HGF{1}{1}{-n    }{-n-\alpha_n+1}{-1/x}
$$
(see, e.g., \cite[Lemma 2.1]{martinez2023real}) allows us to derive that for the rescaled polynomials
$$
\widehat q_n(x) = \HGF{2}{0}{-n, \alpha_n    }{\cdot}{x/n}
$$
under assumption 
$$
 \lim_n \frac{\alpha_n}{n}=A\notin  [-1,0),
$$
for any accumulation point $\mu$ of the normalized zero-counting measures $\chi(\widehat q_n)$ it holds that
\begin{equation}
    \label{Sfor2F0}
  \mathcal S_\mu(z)=   z+A+1  .
\end{equation}
  
\medskip 
  
We ``assemble''  the general case \eqref{sequencePn}, using Proposition \ref{prop:asymptotics2F1} and the building blocks above, appealing to the identity \eqref{identityMult2} and Theorem~\ref{thm:multiplicativeConvHG}.  

Indeed, the polynomial in \eqref{sequencePn} can be written as
$$ 
      p_n(x)=\HGF{s+1}{t}{-n, \bm a_n}{\bm b_n }{ n^{t-s} x} \simeq \dil{n^{t-s}} \left[ \HGF{s+1}{t}{-n, \bm a_n}{\bm b_n }{ x} \right]. 
$$

Denote $\bm a_n =\left(a_1^{(n)}, \dots, a_s^{(n)}\right) $, $\bm b_n =\left(b_1^{(n)}, \dots, b_t^{(n)}\right) $. 
If $s=t$, by Theorem~\ref{thm:multiplicativeConvHG},
\begin{equation}
    \label{cases=ta}
    p_n(x)=\HGF{2}{1}{-n, a_1^{(n)}}{b_1^{(n)} }{   x} \boxtimes_n \HGF{2}{1}{-n, a_1^{(n)}}{b_1^{(n)} }{  x} \boxtimes_n\dots \boxtimes_n  \HGF{2}{1}{-n, a_s^{(n)}}{b_s^{(n)} }{  x},
\end{equation}
and in this case, \eqref{SfunctionForHyper} follows from \eqref{Sfor2F1} and Theorem~\ref{prop:finiteAsymptotics.updated}.

Assume that $s<t$. Then
\begin{align*}
     p_n(x) & \simeq q_{n,1}(x) \boxtimes_n q_{n,2}(x) ,
\end{align*}
where $q_{n,1}(x)$ is the polynomial in the right hand side of \eqref{cases=ta}, while
\begin{align*}
q_{n,2}(x) &= \dil{n^{s-t}} \left[ \HGF{1}{t-s}{-n }{b_{s+1}^{(n)} , \dots, b_{t}^{(n)} }{   x} \right]\\
& =  \dil{n^{-1}} \left[\HGF{1}{1}{-n }{b_{s+1}^{(n)}  }{   x} \right] \boxtimes_n \dots \boxtimes_n
 \dil{n^{-1}} \left[ \HGF{1}{1}{-n }{b_{t}^{(n)}  }{   x} \right] \\
& =  \HGF{1}{1}{-n }{b_{s+1}^{(n)}  }{ n  x} \boxtimes_n \dots \boxtimes_n  \HGF{1}{1}{-n }{b_{t}^{(n)}  }{ n  x}.
\end{align*}
It remains to use \eqref{Sfor2F1}, \eqref{Sfor1F1}, and Theorems~\ref{thm:multiplicativeConvHG} and \ref{prop:finiteAsymptotics.updated} to obtain \eqref{SfunctionForHyper}.

The case $s>t$ is analyzed in the same way, using \eqref{Sfor2F0}.
\end{proof}

\begin{remark}
    From the proof above it is clear that \eqref{SfunctionForHyper} follows directly from \eqref{Sfor1F1}--\eqref{Sfor2F0} and Theorem~\ref{thm:multiplicativeConvHG}. However, we found the analysis of the case of $_2F_1$ polynomials illuminating and having an independent interest. 

    We can analyze the ``pathological'' situations of Theorem \ref{thm:StranfHyper2F1} that arise if we relax the conditions \eqref{mainassumptionsAB} using the arguments from Remark \ref{rem:pathologicalcases}.  

    A result, equivalent to a particular case of the statement of Theorem~\ref{thm:StranfHyper2F1}, can be found in \cite{Neuschel}.
\end{remark}
\begin{theorem}
    \label{cor:CauchyTransf}
    Under the assumptions of Theorem~\ref{thm:StranfHyper2F1}, the Cauchy transform $\mathcal G_\mu(u)$ of the limit measure $\mu$ in a neighborhood of infinity is an algebraic function $y/u$, where $y=y(u)$ satisfies the equation
    \begin{equation}
    \label{eqCTBis}
    y \prod_{j=1}^t (y +B_j ) = u (y-1)   \prod_{i=1}^s  (y+A_i ),
\end{equation}
whose Riemann surface is a ramified covering of $\C$ of genus $0$. 
\end{theorem}
\begin{proof}
    By \eqref{defStranform} and \eqref{SfunctionForHyper} we can find $z=\mathcal M_\mu(u)$ by solving
$$
	u =  \frac{z}{z+1} \frac{ \prod_{i=1}^s  (z+A_i+1)}{ \prod_{j=1}^t (z+B_j+1)  }.
	$$
Recalling the definition \eqref{defMtransform} of the $M$-transform in terms of the Cauchy transform  $\mathcal G_\mu$ of $\mu$, we have that
$$
z=\frac{1}{u}  \mathcal G_\mu\left(\frac{1}{u}\right)-1,
$$
and we can rewrite the equation above as
$$
u =  \frac{u^{-1}  \mathcal G_\mu\left(u^{-1}\right)-1}{u^{-1}  \mathcal G_\mu\left(u^{-1}\right)} \frac{ \prod_{i=1}^s  (u^{-1}\mathcal G_\mu\left(u^{-1}\right)+A_i )}{ \prod_{j=1}^t (u^{-1}\mathcal G_\mu\left(u^{-1}\right)+B_j )  }.
$$
Replacing $u\mapsto 1/u$ and denoting $y= u \mathcal G_\mu(u)$ we arrive at equation \eqref{eqCTBis}. Since it defines $u$ as a rational function of $y$, the Riemann surface of $y=y(u)$ (and hence, of $w=w(u)$) has genus $0$.
\end{proof}

For future reference, it will be convenient to formulate the following immediate consequence:
\begin{corollary}
    \label{cor:shifted}
  Let $c, d\in \R$, $c\neq 0$. Then, under the assumptions of Theorem~\ref{thm:StranfHyper2F1}, any weak-* limit $\mu$ of the normalized zero-counting measures $\chi(p_n)$ of the sequence 
$$
    p_n(x)=\HGF{s+1}{t}{-n, \bm a_n}{\bm b_n }{ n^{t-s} (c x+d)}
$$
is a positive probability measure $\mu$ compactly supported on $\C$, whose $S$-transform $w=\mathcal S_\mu(z)$ satisfies the equation
\begin{equation}
    \label{SforpFqshifted}
w    \prod_{j=1}^t \left[  d  u w + c(u+1+B_j)    \right] 
 =  c^{t-s} \left[  d   w +c  \right]   \prod_{i=1}^s  \left( d  u w + c(u+1 + A_i) \right).
\end{equation}
\end{corollary}
\begin{proof}
Let
$$
   p_n(x)=\HGF{s+1}{t}{-n, \bm a_n}{\bm b_n }{ n^{t-s} x}
$$
and let $\widetilde \mu$ be a weak-* limit of the normalized zero-counting measures $\chi(q_n)$. It is easy to see that the Cauchy transforms of $\widetilde \mu$ and $\mu$ are related by  
$$
\mathcal G_{  \mu}(z) = c\, \mathcal G_{ \widetilde \mu}(cz+d).
$$
Hence, using \eqref{eqCTBis}, we find that $G=\mathcal G_{  \mu}$ satisfies the equation
$$  
 G \prod_{j=1}^t \left(\left(cu+d \right)  G + c B_j \right) = c^{t-s}  \left(\left(c u+d \right)  G-c\right)   \prod_{i=1}^s  \left(\left(cu+d \right)  G+c A_i \right).
$$
 The change of variable $u\mapsto (u+1)/(  w u)$, along with the identity \eqref{StransformAlt}, yield
\eqref{SforpFqshifted} with $ w=\mathcal S_{  \mu}(z)$.
\end{proof}

%%%%%%%%%%%%%%%%%%%%%%%%%%%%%%%%%%%%%%%%%%%%%%%%%%%%%%
\section{Type I Multiple orthogonal polynomials} \label{sec:MOPtype1}
%%%%%%%%%%%%%%%%%%%%%%%%%%%%%%%%%%%%%%%%%%%%%%%%%%%%%%

Recall that Type I multiple orthogonal polynomials are defined by the orthogonality and normalization conditions \eqref{def:typeIfunction}--\eqref{typeInorm}. The potential-theoretic techniques allow to describe the  asymptotics of the zeros of function $Q_{\bm n}$ in \eqref{def:typeIfunction}. The recently obtained expressions for the actual type I polynomials allow us to perform this analysis for polynomial coefficients of $Q_{\bm n}$.

%%%%%%%%%%%%%%%%
\subsection{Type I Jacobi-Piñeiro polynomials}\ \label{subsec:JPTypeI}
%%%%%%%%%%%%%%%%

The AT-system of Jacobi-Piñeiro weights was introduced in \eqref{JPweights}: for $\beta>-1$, $\bm{n}=(n_1,\dots,n_r)\in \N^r$, $\bm \alpha=(\alpha_1,\dots,\alpha_r)$ such that, without loss of generality,
\begin{equation}
    \label{assumptionsonalphas}
     \alpha_1>\alpha_2>\cdots>\alpha_r>-1  \quad \text{and} \quad   \alpha_i-\alpha_j\not\in\Z \text{ for } i\not=j,
\end{equation}
the Type I Jacobi-Piñeiro polynomials of Type I, $P^{(\bm \alpha;\beta)}_{\bm{n},i}$, $\deg P^{(\bm \alpha;\beta)}_{\bm{n},i}  \leq n_j-1$, $i=1, \dots, r$, are given by the orthogonality conditions on the function 
\begin{equation}
\label{def:typeIfunctionJPI}
     Q_{\bm{n}}^{(\bm \alpha;\beta)}(x) \isdef \sum_{j=1}^rP^{(\bm \alpha;\beta)}_{\bm{n},i}(x)w_j(x),\quad w_j(x)\isdef  x^{\alpha_j}(1-x)^\beta,
\end{equation}
namely, 
\begin{equation}  \label{JPtypeI}
 \int_0^1        Q_{\bm{n}}^{(\bm \alpha;\beta)}(x) x^{ k}  \, dx = 0, \qquad 0 \leq k \leq |\bm{n}|-2,
\end{equation}
together with the normalization
$$
 \int_0^1   Q_{\bm{n}}^{(\bm \alpha;\beta)}(x) x^{ |\bm{n}|-1}  \, dx = 1.
$$

Recently, it was shown that these are hypergeometric polynomials. In order to write the explicit expression, we borrow notation from \cite{JPOP23}: for $\bm{a}\in \mathbb{R}^r$, let $\bm{a}^{(i)}\in \mathbb{R}^{r-1}$ stand for the vector obtained from $\bm a$ by deleting its $i$-th entry; also, in a slight abuse of notation, we also understand that adding a scalar to a multi-index means adding this scalar to each one of its components. Then, see \cite{JPOP23},
\begin{equation}\label{JPTypeI.General}
 P^{(\bm \alpha;\beta)}_{\bm{n},i}(x)= c^{(\bm \alpha;\beta)}_{\bm{n},i} 
 \HGF{r+1}{r}{-n_i+1,\alpha_i+\beta+|\bm{n}|,\alpha_i+1-\bm \alpha^{(i)}-\bm{n}^{(i)}}{\alpha_i+1,\alpha_i+1-\bm \alpha^{(i)}}{x},
\end{equation}
where
$$
c^{(\bm \alpha;\beta)}_{\bm{n},i} \isdef (-1)^{\left| \bm{n}\right|-1}\frac{\Gamma(\alpha_i+\beta+|\bm{n}|) \prod_{k=1}^r 
 \raising{\alpha_k+\beta+|\bm{n}|}{n_k} }{(n_i-1)!\,\Gamma(\beta+|\bm{n}|)\Gamma(\alpha_i+1) \prod_{\substack{1\leq k \leq r\\ k\neq i}}\raising{\alpha_{k}-\alpha_i}{n_k}}
$$
is just a normalizing constant that does not affect the zeros of the polynomial. For $r=2$, this formula was obtained first in \cite{branquinho2023hahn}.

Using this expression and Theorem \ref{thm:multiplicativeConvHG}, we get the following representation for polynomials \eqref{JPtypeI} for $i=1,\dots,r$:
\begin{equation}\label{JPTypeI.General.Decomposition}
 P^{(\bm \alpha;\beta)}_{\bm{n},i}=c^{(\bm \alpha;\beta)}_{\bm{n},i} p_{i,1}\boxtimes_{n_i-1} p_{i,2}\boxtimes_{n_i-1}\cdots\boxtimes_{n_i-1} p_{i,r},
\end{equation}
where 
\begin{equation}\label{p.JP.General}
  p_{i,j}(x)=\begin{cases}
   \displaystyle    \HGF{2}{1}{-n_i+1,\alpha_i-\alpha_j-n_{j}+1}{\alpha_i-\alpha_{j}+1}{x}, &  j\in \{1,\cdots,r\}\setminus \{i\},  \\[4mm]
   \displaystyle  \HGF{2}{1}{-n_i+1,\alpha_i+\beta+|\bm{n}|}{\alpha_i+1}{x} , & j=i.
       \end{cases}
\end{equation}

\subsubsection{Real zeros: monotonicity and interlacing}\ \label{sec:zerosJPTypeI}

Since the Jacobi-Piñeiro weights with $\bm \alpha$ satisfying \eqref{assumptionsonalphas} form an AT-system, it is known that the corresponding function $Q_{\bm{n}}^{(\bm \alpha;\beta)}$ in \eqref{def:typeIfunctionJPI} has $|\bm n|-1$ zeros on $(0,1)$ for every multi-index $\bm n$, see, e.g.~\cite[Theorem 2.3]{MR2247778}. Although the potential-theoretic description of the asymptotics of $Q_{\bm{n}}^{(\bm \alpha;\beta)}$ is relatively well studied, the behavior of the zeros of the Type I Jacobi-Piñeiro polynomials $P^{(\bm \alpha;\beta)}_{\bm{n},i}$, $i=1, \dots, r$, is much less known.  

Let us denote the open intervals $\Delta_1, \dots, \Delta_r$, with
\begin{equation}
    \label{deltas}
\Delta_1=(0,1), \quad \Delta_{k}=\begin{cases}
   \R_{<0}, & \text{if $2\le k\le r$ is even,} \\
    \R_{>0}, & \text{if $2\le k\le r$ is odd.}
\end{cases}  
\end{equation}
A simple calculation shows that for $-1<\alpha <0$,
\begin{equation}
    \label{integral1}
    \int_{\Delta_k} \frac{|t|^\alpha}{t-x}dt=(-1)^k \frac{\pi}{\sin(\pi \alpha)}\, |x|^\alpha, \quad x\in \Delta_{k-1}, \quad k\ge 2.
\end{equation}
Let
$$
 s_{k}(x)\isdef 
   |x|^{\alpha_k-\alpha_{k-1}}, \quad x\in \Delta_{k-1}, \quad 2\le k \le r.
$$
If together with \eqref{assumptionsonalphas} we assume that
\begin{equation}
    \label{assumptionNikishin}
  \alpha_1  -1<\alpha_{r} =\min_{1\le k \le r} \alpha_k,
\end{equation}
then by \eqref{integral1}, each $s_k$ is equal, up to a constant factor, to its Cauchy transform on $\Delta_k$:
$$
s_k(x) = \text{const}_k \int_{\Delta_k} \frac{|t|^{\alpha_k-\alpha_{k-1}}}{t-x}\, dt, \quad x\in \Delta_{k-1}, \quad 2\le k \le r,
$$
which means that the weights $w_j$ in \eqref{def:typeIfunctionJPI} form a Nikishin system. By \cite{MR3311759}, under the additional constraint that if the sequence of multi-indices $\bm n=(n_1,\dots, n_r)$ is such that
$$
\left( \max_{1\le k \le r} n_k\right) - \left( \min_{1\le k \le r} n_k\right)
$$
is uniformly bounded, then the number of zeros of each $P^{(\bm \alpha;\beta)}_{\bm{n},i}$, $i=1,\dots, r$, outside $\overline \Delta_r$ is also uniformly bounded, and the zeros accumulate on $\overline \Delta_r\cup \{\infty\}$ as $|\bm n|\to \infty$. 

If we drop the restriction \eqref{assumptionNikishin}, the weights $w_i$ still form a generalized Nikishin system, with
$$
w_{j+1} (x)= x^{m_j} s_j(x)w_j(x), \quad j=1,\dots, r-1, 
$$
for some $m_j\in \Z$. Such a situation is considered in \cite{MR3354972}; the results therein are not immediately applicable due to the assumptions the authors impose on the polynomial factors. However, we conjecture that the zeros of the Jacobi-Piñeiro polynomials $P^{(\bm \alpha;\beta)}_{\bm{n},i}$ accumulate on $\overline \Delta_r\cup \{\infty\}$ as $|\bm n|\to \infty$ under very mild assumptions on the sequence $\bm n$.

Using the approach that we have already described in \cite{martinez2023real}, we can give some additional sufficient conditions for real-rootedness of Jacobi-Piñeiro polynomials, as well as deduce the interlacing and monotonicity of zeros of polynomials with respect to the parameters.

\begin{theorem}\label{Properties.JP}
Let $\beta>-1$ and $\bm \alpha=(\alpha_1,\dots,\alpha_r)$ satisfy \eqref{assumptionsonalphas}. Assume that for a multi-index $\bm n\in\N^{r}$ and for a specific index $i\in \{1, 2,\dots, r\}$,  
\begin{equation}\label{condition1.real.JP}
    \alpha_1-1<\alpha_i<\min_{1\le j\le r }\{\alpha_j+n_j\}-n_i+1.
\end{equation}
Then $P^{(\bm \alpha;\beta)}_{\bm n,i}\in \P_{n_i-1}(\Delta_r)$.

Additionally, for $0< t\leq 2$,
\begin{enumerate}[(i)]
    \item $P^{(\bm \alpha+t;\beta)}_{\bm n,i}\preccurlyeq  P^{(\bm\alpha;\beta)}_{\bm n,i}$ and $ P^{(\bm \alpha;\beta)}_{\bm n,i}\preccurlyeq P^{(\bm\alpha;\beta+t)}_{\bm n,i}$, if $r$ is even, 
    \item $P^{(\bm \alpha;\beta)}_{\bm n,i}\preccurlyeq P^{(\bm\alpha+t;\beta)}_{\bm n,i}$ and $ P^{(\bm \alpha;\beta+t)}_{\bm n,i}\preccurlyeq P^{(\bm\alpha;\beta)}_{\bm n,i}$, if $r$ is odd.
\end{enumerate}
\end{theorem}
Observe that if condition \eqref{condition1.real.JP} is satisfied for certain multi-indices $\bm \alpha $ and $\bm n\in\N^{r}$, it is also satisfied for $\bm \alpha +t$, with $t>0$.
\begin{proof}
Let us show first that  
$$
p_{i,i}\in\mathbb{P}(\mathbb{R}_{>0}) \quad \text{and} \quad  p_{i,j}(x)\in\P(\R_{<0}), \quad j\in \{1,\dots, r\}\backslash\{i\}.
$$
We can establish it using some results from \cite{dominici2013real}, many of them re-proved in \cite{martinez2023real}.
    Indeed, since by assumptions, $\alpha_i, \beta>-1$, $\bm n\in \N^r$, and thus,
    \begin{equation*} 
        \alpha_i+\beta+|\bm n|>\alpha_i+n_i-1,  
    \end{equation*}
we have by \cite[Theorem 1, \textit{(ii)}]{dominici2013real} that $p_{i,i}\in\P_{n_i-1}((0,1))$. 

Similarly, by \eqref{assumptionsonalphas} and \eqref{condition1.real.JP}, for every $j\in \{1,\dots, r\}\backslash\{i\}$,
$$ 
     \alpha_j-1<\alpha_i<\alpha_j+n_j-n_i+1,
$$
or equivalently,
$$  
             -n_{i}+2> \alpha_{i}-\alpha_{j}-n_{j}+1,\quad \alpha_{i}-\alpha_{j}+1>0  .
$$  
It remains to use \cite[Theorem 1, \textit{(i)}]{dominici2013real} to conclude that $p_{i,j}\in\P_{n_i-1}(\R_{<0})$. 
By Proposition \ref{prop:realrootedness}, for
\begin{equation} \label{defq_i}
 q_i:=p_{i,1}\boxtimes_{n_{i}-1}\cdots\boxtimes_{n_{i}-1}  p_{i,i-1}\boxtimes_{n_{i}-1} p_{i,i+1} \boxtimes_{n_{i}-1} \cdots\boxtimes_{n_{i}-1} p_{{i},r},
\end{equation}
we have that
\begin{equation}
    \label{eq:locationq}
q_i\in 
    \begin{dcases}
        \P(\R_{>0}) & \text{for $ r$  odd,} \\
        \P(\R_{<0}) & \text{for  $r$  even.} 
    \end{dcases}
\end{equation}
Since the decomposition \eqref{JPTypeI.General.Decomposition} of $P^{(\bm \alpha;\beta)}_{\bm n,i}$ can be written as
\begin{equation}
    \label{decompJPIrevisited}
 P^{(\bm \alpha;\beta)}_{\bm{n},i}=c^{(\bm \alpha;\beta)}_{\bm{n},i} p_{i,i}\boxtimes_{n_i-1} q_{i},
\end{equation}
it remain to Proposition \ref{prop:realrootedness} again to establish the first assertion of the theorem (about the location of the zeros).

A key observation to prove interlacing is that polynomials $p_{i,j}$ in \eqref{p.JP.General}, with $j\neq i$, are independent of the parameter $\beta>-1$, and invariant with respect to the scalar shift $\bm \alpha \mapsto \bm \alpha +t$, $t>0$. Thus, the same property is shared by the polynomial $q_i$ in \eqref{defq_i}. 

On the other hand, by \cite[Eq. (61)]{martinez2023real}, for $0\leq t_1,t_2\leq 2$ and $t_1+t_2\not=0$.
    \begin{equation}
        \HGF{2}{1}{-n_{i}+1,\alpha_{i}+\beta+|\bm{n}|+t_1}{\alpha_{i}+1}{x} \prec \HGF{2}{1}{-n_{i}+1,\alpha_{i}+\beta+|\bm{n}|+t_2}{\alpha_{i}+1+t_2}{x}. 
    \end{equation}

For $r$ odd, using   \eqref{eq:locationq}  and Proposition \ref{lem:preservinginterlacingMult}, we get that
\begin{equation}\label{odd.interlacing.JP}
\begin{split}
    \HGF{2}{1}{-n_{i}+1,\alpha_{i}+\beta+|\bm{n}|+t_1}{\alpha_{i}+1}{x}&\boxtimes_{n_{i}-1}q_i \\ & \preccurlyeq \HGF{2}{1}{-n_{i}+1,\alpha_{i}+\beta+|\bm{n}|+t_2}{\alpha_{i}+1+t_2}{x}\boxtimes_{n_{i}-1}q_i ,
\end{split}
\end{equation}
while for $r$ even, using now Remark \ref{preserveinterlacing.negativezeros}, we obtain 
\begin{equation}\label{even.interlacing.JP}
\begin{split}
    \HGF{2}{1}{-n_{i}+1,\alpha_{i}+\beta+|\bm{n}|+t_2}{\alpha_{i}+1+t_2}{x}&\boxtimes_{n_{i}-1}q_i \\ & \preccurlyeq \HGF{2}{1}{-n_{i}+1,\alpha_{i}+\beta+|\bm{n}|+t_1}{\alpha_{i}+1}{x}\boxtimes_{n_{i}-1}q_i .
\end{split}
\end{equation}
In particular, setting in \eqref{odd.interlacing.JP}-\eqref{even.interlacing.JP}, $t_1=0$ and $t_2=t$, and noticing that a multiplicative constant does not affect interlacing, we conclude that
$$ 
    P^{(\bm \alpha;\beta)}_{\bm n,i}\preccurlyeq P^{(\bm\alpha+t;\beta)}_{\bm n,i} \quad \text{for }r\text{ even} \quad \text{and} \quad P^{(\bm \alpha+t;\beta)}_{\bm n,i}\preccurlyeq P^{(\bm\alpha;\beta)}_{\bm n,i} \quad \text{for }r\text{ odd},
$$ 
while taking $t_1=t$ and $t_2=0$ in \eqref{odd.interlacing.JP}-\eqref{even.interlacing.JP} yields
$$
    P^{(\bm \alpha;\beta+t)}_{\bm n,i}\preccurlyeq P^{(\bm\alpha;\beta)}_{\bm n,i} \quad \text{for }r\text{ even} \quad \text{and} \quad P^{(\bm \alpha;\beta)}_{\bm n,i}\preccurlyeq P^{(\bm\alpha;\beta+t)}_{\bm n,i} \quad \text{for }r\text{ odd}.
$$
\end{proof}

\begin{remark}
The statement about zero interlacing contained in Theorem \ref{Properties.JP} implies that for $r$ even, the zeros of the polynomial $P^{(\bm \alpha+t;\beta)}_{\bm n,i}$ are decreasing with respect to $t$ and increasing with respect to $\beta$; this monotonicity is exactly the opposite if $r$ is odd.

As it follows from the proof of Theorem~\ref{Properties.JP}, conditions \eqref{condition1.real.JP} can be weakened. For instance, if instead of \eqref{condition1.real.JP} we assume the existence of an index $j\neq i$ such that
\begin{equation}\label{condition.real}
   \alpha_j-2<\alpha_i<\alpha_j+n_j-n_i+2,\quad \max_{k\not=i,j}\{\alpha_k\}-1<\alpha_i<\min_{k\not=i,j}\{\alpha_k+n_k\}-n_i+1 ,
\end{equation}
then we arrive at a weaker conclusion that $P^{(\bm \alpha;\beta)}_{\bm n,i} \in \P_{n_i-1}(\R)$.

Notice that additionally to \textit{(i)--(ii)} we could state that   $ P^{(\bm \alpha;\beta)}_{\bm n,i}\preccurlyeq P^{(\bm\alpha+\bm e_i;\beta+1)}_{\bm n-\bm e_i,i}$, where $\bm{e}_i \in \mathbb{N}^r$ is the multi-index that has all entries equal to 0 except the entry of index $i$ which is equal to 1.  This observation is based on the simple fact that, by \eqref{derivativeHyper},
\begin{equation*}
    \frac{d}{dx}\left(P^{(\bm \alpha;\beta)}_{\bm{n},i}(x) \right)\simeq P^{(\bm \alpha+\bm e_i;\beta+1)}_{\bm{n}-\bm e_i ,i}(x).
\end{equation*}
\end{remark}

If we try to enforce the assumptions \eqref{condition1.real.JP} to be valid for every index $i$, we obtain the following simple consequence of Theorem \ref{Properties.JP}:
\begin{corollary}
    \label{JP.realroots.all}
Let $\beta>-1$, $\bm \alpha=(\alpha_1,\dots,\alpha_r)$ satisfy \eqref{assumptionsonalphas} and \eqref{assumptionNikishin}. Then 
for any multi-index $\bm n\in\N^{r}$ of the form
\begin{equation}
    \label{assumptionCor43}
  \bm n=(n, \dots, n, n+1, \dots, n+1)
\end{equation}
(that is, such that for $i<j$, $0\le n_j-n_i \le 1$), conclusions of Theorem \ref{Properties.JP} hold for all  $i\in \{1, \dots, r\}$.
\end{corollary}
From our discussion above it is clear that the assumptions of Corollary \ref{JP.realroots.all} reduce the situation to a step-line (as the one studied already in \cite{MR0602272}) for a Nikishin system of weights. In this sense, the reality of zeros is not surprising, although the interlacing properties seem to be new even in this case.
\begin{proof}
Fix $i\in \{1, \dots, r\}$. Taking into account \eqref{assumptionsonalphas}, condition \eqref{assumptionNikishin} implies that $\alpha_1-1<\alpha_i$.

On the other hand, for any index $j<i$, by \eqref{assumptionCor43}, $n_i \le n_j+1$, and 
$$
-1<\alpha_i-\alpha_j <0 \le n_j -n_i +1.
$$
Analogously, for any index $j>i$, by \eqref{assumptionCor43}, $n_i \le n_j $, and 
$$
0<\alpha_i-\alpha_j <1 \le n_j -n_i +1.
$$
These two sets of inequalities are equivalent to \eqref{condition1.real.JP}. On the other hand, they hold for every $i\in \{1, \dots, r\}$ if and only if $0\le n_j-n_i \le 1$ for $i<j$, and the assertion follows.
\end{proof}

\subsubsection{Zero asymptotics of Jacobi-Piñeiro polynomials of Type I}\ \label{sec:asymptoticsTypeIJP}

We can describe the weak zero asymptotics of a sequence of Type I Jacobi-Pi\~neiro polynomials,
$$
p_{\bm n, i}= P^{(\bm \alpha_{\bm n};\beta_{\bm n})}_{\bm{n},i} , \quad i=1, \dots, r,
$$
where $\bm \alpha_{\bm n} = \left(\alpha_1^{(\bm n)}, \dots, \alpha_r^{(\bm n)} \right)$, $\bm n =(n_1, \dots, n_r)\in \N^r$, under assumption
\begin{equation}
    \label{asymptJPIassumptions}
    \lim_{|\bm n|\to \infty} \frac{\alpha_i^{(\bm n)}}{|\bm n|}=A_i\ge 0, \quad \lim_{|\bm n|\to \infty} \frac{\beta_{(\bm n)}}{|\bm n|}=B\ge 0,  \quad \text{and} \quad \lim_{|\bm n|\to \infty} \frac{n_i}{|\bm n|}=\theta_i>0, \qquad i=1, \dots, r.
\end{equation}
Clearly,   $\theta_1+\dots + \theta_r=1$. Formula \eqref{JPTypeI.General} indicates that we need to consider the parameters 
\begin{equation}
    \label{defABfrak}
\mathfrak a_j^{(i)} \isdef \begin{cases}
    (A_i+B+1)/\theta_i, & \text{if } j=i, \\
    (A_i-A_j-\theta_j)/\theta_i & \text{if } j\ne i,
\end{cases}\quad \text{and}\quad \mathfrak b_k^{(i)} \isdef \begin{cases}
     A_i / \theta_i, & \text{if } k=i, \\
    (A_i-A_k)/\theta_i & \text{if } k\ne i.
\end{cases}
\end{equation}
Restrictions $\mathfrak a_j^{(i)}\notin [-1,0)$ and $\mathfrak b_k^{(i)}\neq -1$ imply that additionally to \eqref{asymptJPIassumptions} we need to assume that
\begin{equation}
    \label{asymptJPIassumptions2}
    A_j \neq A_i +\theta_i, \quad A_j +\theta_j \notin (A_i, A_i + \theta_i], \qquad j\in \{1, \dots, r\}\setminus \{i\}.
\end{equation}
Notice that the assumption that $\mathfrak a_j^{(i)} \neq \mathfrak b_k^{(i)}$ holds automatically.

By Theorem~\ref{thm:StranfHyper2F1}, for any weak-* accumulation point $\mu$ of the normalized zero-counting measures $\chi(p_{\bm n, i})$,  
\begin{equation}
    \label{SfunctionForJPI}
   \mathcal S_\mu(z)= \prod_{j=1}^r  \frac{ z+\mathfrak a_j^{(i)}+1 }{ z+\mathfrak b_j^{(i)}+1  },
\end{equation}
and by Theorem \ref{cor:CauchyTransf}, its Cauchy transform $\mathcal G_\mu(u)$ in a neighborhood of infinity is an algebraic function $y/u$, where $y=y(u)$ satisfies the equation
$$  
    y \prod_{j=1}^r (y +\mathfrak b_j^{(i)}  ) = u (y-1)   \prod_{i=1}^r  (y+\mathfrak a_j^{(i)}  ).
$$  

Let us discuss the case when $\bm \alpha$ and $\beta$ do not depend on $\bm n$, so that
\begin{equation}
    \label{defABfrakConstCase}
\mathfrak a_j^{(i)} \isdef \begin{cases}
    1/\theta_i, & \text{if } j=i, \\
     -\theta_j/\theta_i & \text{if } j\ne i,
\end{cases}\quad \text{and}\quad \mathfrak b_j^{(i)} =0,
\end{equation}
and constrains \eqref{asymptJPIassumptions2} boil down to
\begin{equation}
    \label{assumptThetas}
\theta_j \notin (0,   \theta_i], \qquad j\in \{1, \dots, r\}\setminus \{i\}.
\end{equation}
If we assume that there exists an index $i\in \{1, \dots, r\}$ such that
$$
\theta_i < \min \left\{ \theta_j:\, j\in \{1, \dots, r\}\setminus \{i\} \right\},
$$
then for all sufficiently large values of $|\bm n|$, the inequality in the right-hand side of \eqref{condition1.real.JP} holds. Thus, imposing additionally that $|\alpha_i-\alpha_j|<1$ for all $j$'s, we conclude that any weak-* accumulation point $\mu$ of the normalized zero-counting measures $\chi(p_{\bm n, i})$ is a compactly supported probability measure on $\R_{\ge 0}$, if $r$ is odd, and on $\R_{\le 0}$, otherwise, and such that
 $$
   \mathcal S_\mu(z)= \frac{z+1+1/\theta_i}{(z+1)^r}\prod_{j\ne i}   (z+ 1 - \theta_j/\theta_i ),
$$
while $y=u \mathcal G_\mu(u)$ is a solution of 
$$
y^{r+1} = u  \left(y+\frac{1}{\theta_i} \right)\prod_{j=1}^r \left(y-\frac{\theta_j}{\theta_i} \right).
$$

In the asymptotically diagonal situation, when all $\theta_j=1/r$, we have $
\mathfrak a_j^{(i)} = -1$ if $j\ne i$, which is one of the degenerate cases discussed in Remark \ref{rem:pathologicalcases}. However, the expression for the $S$-transform of the limit measure is still valid:
 $$
   \mathcal S_\mu(z)= \frac{z^{r-1}}{(z+1)^r}  (z+ 1 +1/r  ) ,
$$
as well as for the Cauchy transform $\mathcal G_\mu(u)=y/u$, where $y=y(u)$ solves  
\begin{equation}
    \label{eqCauchyTrTypeIJPdiagonal}
    y^{r+1} = u (y-1)^r \left(y+r \right).
\end{equation}  
We conjecture that, in this case, the support of any limit measure $\mu$ is unbounded.
If the parameters $\bm \alpha$ satisfy \eqref{condition1.real.JP}, it follows from Theorem \ref{Properties.JP} that all zeros of the corresponding polynomial $P^{(\bm \alpha;\beta)}_{\bm n,i}$ are on $\Delta_r$. Furthermore, as follows from the method of proof of \cite[Theorem 3.2]{MR3354972}, even when \eqref{condition1.real.JP} is not satisfied, the zeros still accumulate on $\Delta_r\cup \{\infty\}$.

\begin{example}\label{example44}
If $r=2$ and  $\theta_1 = \theta< 1/2$, the Cauchy transform $\mathcal G_\mu(u)$ of the limit measure $\mu$ is given by $y/u$, where $y=y(u)$ solves the equation
\begin{equation}\label{cubic}
    y^3 = u (y-1)    \left(y+\frac{1}{\theta}  \right)\left(y-\frac{ 1}{\theta} +1 \right) =u\left( y^3 -\left(1+\nu \right) y + \nu\right), \quad \nu=\frac{1}{\theta } \left( \frac{1}{\theta }-1\right)>0.
\end{equation} 
We can take advantage that this is a depressed cubic and use the Cardano formula for its solution, selecting the correct branch noticing that, by the definition of the Cauchy transform, $y(u)>0$ for $u>0$. Namely, consider the functions
$$
f(u)=\left(\frac{u}{1-u}  \right)^{1/3},
$$
analytic in $\C\setminus\left[(-\infty, 0]\cup [1,+\infty)\right]$, whose single-valued branch is fixed by requiring $f(x)>0$ for $x\in (0,1)$. Then we can write
$$
y(u)= \sqrt[3]{\frac{\nu}{2} }\, f(u) \left( \left(\sqrt{1+ \frac{\kappa \,u}{1-u}}+1 \right)^{1/3} -\left(\sqrt{1+ \frac{\kappa \,u}{1-u}}-1 \right)^{1/3} \right) ,
$$
where
$$
\kappa =\kappa(\nu)=\frac{4}{27}\, \frac{(1+\nu)^3}{\nu^2}
$$
and with the branches of the roots taken always positive for $u\in (1,1+\varepsilon)$ for a sufficiently small $\varepsilon>0$. Notice that $\kappa(\nu)$ is a strictly increasing function on $[2,+\infty)$, with $\kappa(2)=1$.

In particular, for $u=-x<0$, denote by $y_\pm$ the boundary values of $y$ on $\R_{<0}$. By the symmetry principle,  $y_-=\overline{y_+ }$ on $\R_{<0}$. 

By the Sokhotski-Plemelj Formula (or Stieltjes' inversion theorem), the density of the limit measure $\mu$ can be recovered as
\begin{align*}
    \mu'(-x) & =\frac{1}{2\pi i(-x)} \left(y_- (-x)-y_+ (-x) \right) = \frac{1}{ \pi x } \Im \left(y_+ (-x) \right)  .
\end{align*}
We have that
$$
\left(\sqrt{1+ \frac{\kappa \,u}{1-u}}+1 \right)^{1/3}_+ \bigg|_{u=-x} =  \left(\sqrt{1- \frac{\kappa \, x}{1+x}}+1 \right)^{1/3} >0 \text{ for } x\in (-c^*, 0),
$$
where $c^*=c^*(\nu)>0$ solves $\kappa \, c^*/(1+c^*)=1$, that is,
\begin{equation}
    \label{defBeta}
 c^* = \frac{1}{\kappa-1}  = 27 \left( \frac{  \theta  (1-\theta  )  }{(1-2 \theta )  (2-\theta  )  (1+\theta )}\right)^2.
\end{equation}
Analogously,
$$
\left(\sqrt{1+ \frac{\kappa \,u}{1-u}}-1 \right)^{1/3}_+ \bigg|_{u=-x} =  \left(1- \sqrt{1- \frac{\kappa \, x}{1+x}}  \right)^{1/3}e^{2\pi i /3} \text{ for } x\in (-c^*, 0).
$$
Finally,
$$
f_+(-x)=\sqrt[3]{\frac{x}{1+x}}\, e^{2\pi i /3}, \quad x<0.
$$
Putting all together, we get that $\mu$ is supported on $[-c^*, 0]$, $\beta$ given in \eqref{defBeta}, with
\begin{equation}
    \label{densityJPI}
    \begin{split}
     \mu'(-x)  = & \frac{\sqrt{3}}{2\pi }\, \sqrt[3]{\frac{\nu}{2} } \,  \frac{ \left( \sqrt{1+x}+ \sqrt{(c^*-x)/c^* }  \right)^{1/3} - \left(\sqrt{1+x}-\sqrt{(c^*-x)/c^* }  \right)^{1/3}  }{x^{2/3} \sqrt{1+x}}. 
     \end{split}
\end{equation}
Symbolic integration allows us to check that $\mu'(-x)$ is the density of a positive probability measure on $[-\beta, 0]$, and that $\mu'(-x)=\frac{\sqrt{3} \nu^{1/3}}{2\pi x^{2/3}}(1+o(1))$ as $x\to 0+$, and with a square root decay as $x\to c^*-$.  

\begin{figure}
    \centering
    \includegraphics[width=.45\textwidth]{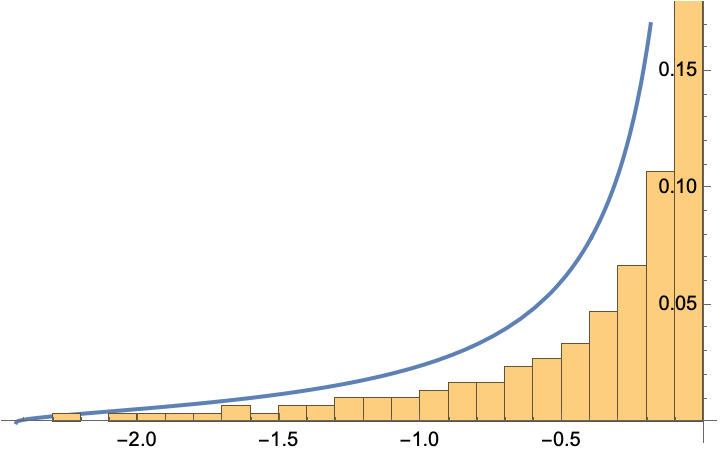}\hfill \includegraphics[width=.45\textwidth]{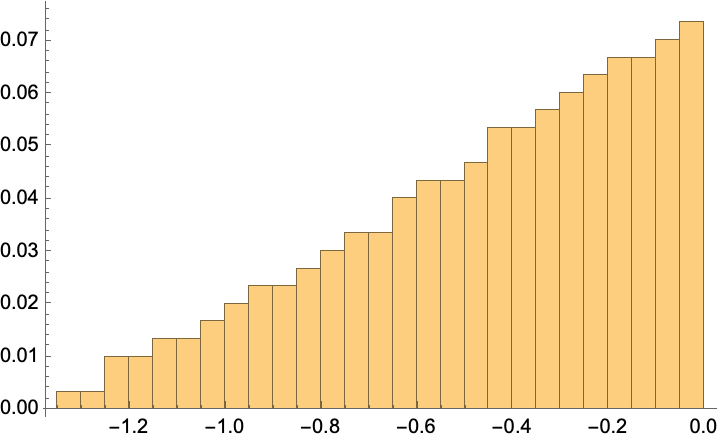}
    \caption{Left: histogram of the zeros of the polynomial $P^{(\bm \alpha ;\beta )}_{\bm{n},1}$, with $\bm \alpha=(1/2, 3/7)$, $\beta=1$, and $\bm n = (300, 600)$, along with the corresponding density \eqref{densityJPI}. Notice that for $\theta=1/3$, measure $\mu$ is supported on $[- 2.43,0]$. The smallest zero of $P^{(\bm \alpha ;\beta )}_{\bm{n},1}$ in this case is approximately $-2.2$. Right: histogram of the cube roots of the same zeros.}
    \label{fig:1}
\end{figure}

As an illustration, we plot the histogram of the zeros, all real and negative, of the polynomial $P^{(\bm \alpha; \beta)}_{\bm{n},1}$, with $\bm \alpha=(1/2, 3/7)$, $\beta=1$, and $\bm n = (300, 900)$, along with the corresponding density \eqref{densityJPI}, see Figure~\ref{fig:1}. For $\theta=1/3$, the asymptotic distribution of zeros is supported on $[- 2.43,0]$.

The diagonal case can be considered taking $\theta \to (1/2)_-$, when \eqref{densityJPI} becomes
$$ 
    \mu'(-x)  = \frac{\sqrt{3}}{2\pi  } \,  \frac{ \left(    \sqrt{1+x}+1\right)^{1/3} - \left(    \sqrt{1+x}-1\right)^{1/3}  }{x^{2/3} \sqrt{1+x} }, \quad x>0.
$$
This is the density of a positive probability measure on $\R_{>0}$, such that $\mu'(-x)=\frac{\sqrt{3}}{\pi (2x)^{2/3}}(1+o(1))$ as $x\to 0+$. 
\end{example}

%%%%%%%%%%%%%%%%
\subsection{Type I Multiple Laguerre polynomials of the first kind}\ \label{subsec:MLTypeI}
%%%%%%%%%%%%%%%%

Let $\bm \alpha=(\alpha_1,\dots,\alpha_r)$ be a multi-index satisfying conditions \eqref{assumptionsonalphas}. Then for $\bm{n}=(n_1,\dots,n_r)\in \N^r$,  the type I multiple Laguerre polynomials of the first kind, $L_{\bm{n},i}^{(\bm\alpha)}$,  $\deg L_{\bm{n},i}^{(\bm\alpha)}\leq n_j-1$, $i=1, \dots, r$, are given by the orthogonality conditions on the function 
\begin{equation}
\label{def:typeIfunctionMLI}
     Q_{\bm{n}}^{(\bm \alpha)}(x) \isdef \sum_{j=1}^r L_{\bm{n},i}^{(\bm\alpha)}(x)\,  x^{\alpha_j}e^{-x},
\end{equation}
namely, 
\begin{equation}  \label{MLtypeI}
 \int_0^{+\infty}        Q_{\bm{n}}^{(\bm \alpha)}(x) x^{ k}  \, dx = 0, \qquad 0 \leq k \leq |\bm{n}|-2,
\end{equation}
together with the normalization
$$
 \int_0^{+\infty}   Q_{\bm{n}}^{(\bm \alpha)}(x) x^{ |\bm{n}|-1}  \, dx = 1.
$$

In \cite{JPOP23}, it was shown that these are hypergeometric polynomials: using the notation introduced before \eqref{JPTypeI.General}, 
\begin{equation}\label{LaguerreFKType1.General}
     L_{\bm{n},i}^{(\bm\alpha)}(x)=c_{\bm{n},i}^{(\bm\alpha)} \HGF{r}{r}{-n_i+1,\alpha_i+1-\bm \alpha^{(i)}-\bm{n}^{(i)}}{\alpha_i+1,\alpha_i+1-\bm \alpha^{(i)}}{x}
\end{equation}
where 
$$c_{\bm{n},i}^{(\bm\alpha)} =(-1)^{|\bm{n}|-1}\left( (n_i-1)!\Gamma(\alpha_i+1)\prod_{\substack{ 1\leq i\leq r \\ j\neq i}}\raising{\alpha_{j}-\alpha_i}{n_{j}}\right)^{-1}. 
$$
For $r=2$, this formula was previously derived in \cite{branquinho2023hahn}).
    
Using this expression and Theorem \ref{thm:multiplicativeConvHG}, we get the following representation for polynomials \eqref{LaguerreFKType1.General} for $i=1,\dots,r$:
\begin{equation}\label{L.TypeI.General.Decomposition}
 L^{(\bm \alpha)}_{\bm{n},i}=c^{(\bm \alpha)}_{\bm{n},i}  p_{i,1}\boxtimes_{n_i-1} p_{i,2}\boxtimes_{n_i-1}\cdots\boxtimes_{n_i-1} p_{i,r},
\end{equation}
where
\begin{equation}\label{p.L.General}
  p_{i,j}(x)=\begin{cases}
   \displaystyle    \HGF{2}{1}{-n_i+1,\alpha_i-\alpha_j-n_{j}+1}{\alpha_i-\alpha_{j}+1}{x}, &  j\in \{1,\cdots,r\}\setminus \{i\},  \\[4mm]
   \displaystyle  \HGF{1}{1}{-n_i+1}{\alpha_i+1}{x} , & j=i.
       \end{cases}
\end{equation}

Notice that for $i\neq j$ the $p_{i,j}$ are exactly the same polynomials appearing in \eqref{p.JP.General}, thus we have the relation 
 \begin{equation}\label{JP-L.relation}
     L^{(\bm \alpha)}_{\bm n,i}   \simeq  \left(v_i(x)\boxtimes_{n_i-1} P^{(\bm \alpha;\beta)}_{\bm{n},i}\right)
 \end{equation}
with
\begin{equation}\label{q.polynomial}
    v_i(x)=\HGF{1}{1}{-n_i+1}{\alpha_i+\beta+|\bm n|}{x} = \frac{(n_i-1)!}{\raising{\alpha_i+\beta+|\bm n|}{n_i-1}} \, L_{n_i-1}^{(\alpha_i+\beta+|\bm n|-1)}(x) , 
\end{equation}
where $L_n^{(\alpha)}$ is the Laguerre polynomial. In other words, the Type I multiple Laguerre polynomial of the first kind can be obtained as a finite multiplicative convolution of the Type I Jacobi-Piñeiro polynomials with the same parameters and a standard Laguerre polynomial.

\subsubsection{Real zeros: monotonicity and interlacing}\

The discussion about real zeros carried out in Section~\ref{sec:zerosJPTypeI} applies to Type I multiple Laguerre polynomials, with the only modification that now in \eqref{deltas},  $\Delta_1=(0,+\infty)$. In particular, under assumption \eqref{assumptionNikishin}, these polynomials form a Nikishin system, and thus, for close-to-diagonal multi-indices $\bm n$ their zeros will be on $\Delta_r$. 

In addition, we have the following
\begin{theorem}\label{Properties.LTypeI}
Let $\bm \alpha=(\alpha_1,\dots,\alpha_r)$ satisfy \eqref{assumptionsonalphas}. Assume that for a multi-index $\bm n\in\N^{r}$ and for a specific index $i\in \{1, 2,\dots, r\}$,  condition \eqref{condition1.real.JP} is satisfied. Then $L^{(\bm \alpha)}_{\bm n,i}\in \P_{n_i-1}(\Delta_r)$. 

Additionally, for $0< t\leq 2$,
\begin{enumerate}[(i)]
    \item $L^{(\bm \alpha+t)}_{\bm n,i}\preccurlyeq L^{(\bm\alpha)}_{\bm n,i}$,  if $r$ is even, and 
    \item $L^{(\bm \alpha)}_{\bm n,i}\preccurlyeq L^{(\bm\alpha+t)}_{\bm n,i}$, if $r$ is odd.
\end{enumerate}
\end{theorem}
\begin{proof}
The statement follows immediately from representation \eqref{JP-L.relation} and  Theorem \ref{Properties.JP}, and the fact that all zeros of the Laguerre polynomial in \eqref{q.polynomial} are positive.
\end{proof}
Additionally, as for the Jacobi-Pi\~neiro polynomials, from \eqref{derivativeHyper} it follows that $ L^{(\bm \alpha)}_{\bm n,i}\preccurlyeq L^{(\bm\alpha+\bm e_i)}_{\bm n-\bm e_i,i}$.

\begin{corollary}
    \label{TypeILkind1.realroots.all}
Let  $\bm \alpha=(\alpha_1,\dots,\alpha_r)$ satisfy \eqref{assumptionsonalphas} and \eqref{assumptionNikishin}. Then for any multi-index $\bm n\in\N^{r}$ on the step-line, i.e., of the form \eqref{assumptionCor43}, conclusions of Theorem \ref{Properties.LTypeI} hold for all  $i\in \{1, \dots, r\}$.
\end{corollary}

\subsubsection{Zero asymptotics of Type I Multiple Laguerre polynomials of first kind}\

As we did for the Type I Jacobi-Pi\~neiro polynomials, we can describe the weak zero asymptotics of a sequence of rescaled Type I Multiple Laguerre polynomials of first kind, namely, 
$$
p_{\bm n, i}(x)= L^{(\bm \alpha_{\bm n} )}_{\bm{n},i} (n_i \, x), \quad i=1, \dots, r,
$$
where $\bm \alpha_{\bm n} = \left(\alpha_1^{(\bm n)}, \dots, \alpha_r^{(\bm n)} \right)$, $\bm n =(n_1, \dots, n_r)\in \N^r$, under assumption
\begin{equation}
    \label{asymptMLIassumptions}
    \lim_{|\bm n|\to \infty} \frac{\alpha_i^{(\bm n)}}{|\bm n|}=A_i\ge 0,    \quad \text{and} \quad \lim_{|\bm n|\to \infty} \frac{n_i}{|\bm n|}=\theta_i>0, \qquad i=1, \dots, r.
\end{equation}
Clearly,   $\theta_1+\dots + \theta_r=1$. Formula \eqref{LaguerreFKType1.General} indicates that with the parameters $\mathfrak a_j^{(i)}$ and $\mathfrak b_j^{(i)}$ defined in \eqref{defABfrak}, satisfying the same restrictions \eqref{asymptJPIassumptions2}, namely,
$$   
    A_j \neq A_i +\theta_i, \quad A_j +\theta_j \notin (A_i, A_i + \theta_i], \qquad j\in \{1, \dots, r\}\setminus \{i\},
$$
Theorem~\ref{thm:StranfHyper2F1} implies that for any weak-* accumulation point $\mu$ of the normalized zero-counting measures $\chi(p_{\bm n, i})$,  
\begin{equation}
    \label{SfunctionForIML}
   \mathcal S_\mu(z)=  \frac{ \prod_{j=1, \, j\neq i}^r  \left( z+\mathfrak a_j^{(i)}+1 \right)}{ \prod_{j=1}^r  \left( z+\mathfrak b_j^{(i)}+1\right)  },
\end{equation}
and by Theorem \ref{cor:CauchyTransf}, the Cauchy transform $\mathcal G_\mu(u)$ of the limit measure $\mu$ in a neighborhood of infinity is an algebraic function $y/u$, where $y=y(u)$ satisfies the equation
$$  
    y \prod_{j=1}^r (y +\mathfrak b_j^{(i)}  ) = u (y-1)   \prod_{i=1, \, j\neq i}^r  (y+\mathfrak a_j^{(i)}  ).
$$  

Let us discuss the case when $\bm \alpha$ does not depend on $\bm n$, so that identities \eqref{defABfrakConstCase} hold, and constrains \eqref{asymptJPIassumptions2} once again boil down to \eqref{assumptThetas}.
If we assume that there exists an index $i\in \{1, \dots, r\}$ such that
$$
\theta_i < \min \left\{ \theta_j:\, j\in \{1, \dots, r\}\setminus \{i\} \right\},
$$
imposing additionally that $|\alpha_i-\alpha_j|<1$ for all $j$'s, then as in Section \ref{sec:asymptoticsTypeIJP} we conclude that any weak-* accumulation point $\mu$ of the normalized zero-counting measures $\chi(p_{\bm n, i})$ is a compactly supported probability measure on $\R_{\ge 0}$, if $r$ is odd, and on $\R_{\le 0}$, otherwise, and such that
 $$
   \mathcal S_\mu(z)= \frac{ 1 }{(z+1)^r}\prod_{j\ne i}   (z+ 1 - \theta_j/\theta_i ).
$$
\begin{example}\label{exampler=2TypeILag}
Like in Example \ref{example44}, if $r=2$, and  $\theta_1 = \theta< 1/2$, the Cauchy transform $\mathcal G_\mu(u)$ of the limit measure $\mu$ is given by $y/u$, where $y=y(u)$ solves the equation
$$  
    y^3 = u (y-1)     \left(y-\frac{ 1}{\theta} +1 \right).
$$  
With the change $y=3w-u/3$ we get
$$
27 w^3+  \left(\frac{3}{\theta}    -u \right)u w+\left(\frac{u  }{3\theta} +1 -\frac{1}{\theta}-\frac{2 u^2}{27}\right) u =0.
$$
Proceeding as in Example \ref{example44}, we can use Cardano's formula and select the right branch of the solution observing that $w>0$ for $u>0$. The Sokhotski–Plemelj formula shows then that the limit measure $\mu$ lives on the interval $[-c^*,0]$, with 
$$
c^* =\frac{9 (1-\theta ) \theta -2+2 \sqrt{(1-3 (1-\theta ) \theta )^3} }{ \theta (1-2 \theta )^2
  }>0.
$$
See Figure~\ref{fig:2} for a comparison of some actual values of the smallest zeros with the predicted value of $-c^*$.
\begin{figure}
    \centering
    \includegraphics[width=.65\textwidth]{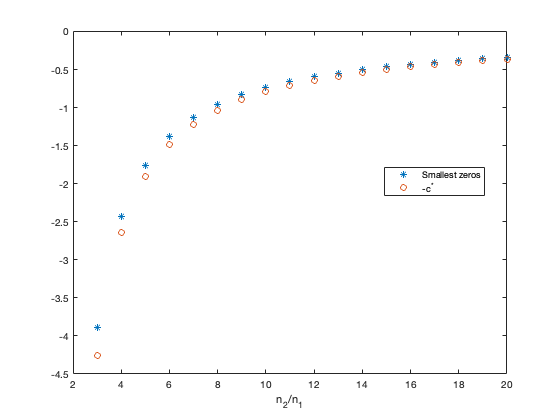}
    \caption{Values of the smallest zeros (marked with $*$) and of the corresponding values of $-c^*$ (left end-point of the limit zero distribution, marked with 'o') for $L^{(\bm \alpha)}_{(n_1,n_2),1}$, with $\bm \alpha=(1/2, 3/7)$, $ n_1 =250$, and different values of $n_2/n_1$. }
    \label{fig:2}
\end{figure}
\end{example}

\subsection{Type I Multiple Laguerre of the second kind}\label{sec:TIML2k}\

Let $\alpha>-1$ and $\bm c=(c_1,\dots ,c_r)\in \R_{>0}^r$ be a multi-index such that $c_i\neq c_j$ for $i\neq j$. Then, for $\bm{n}=(n_1,\dots,n_r)\in \N^r$,  the type I multiple Laguerre polynomials of the second kind, $\mathcal L_{\bm{n},i}^{(\alpha, \bm c)}$,  $\deg \mathcal L_{\bm{n},i}^{(\alpha, \bm c)}\leq n_j-1$, $i=1, \dots, r$, are given by the orthogonality conditions on the function 
\begin{equation*}
     Q_{\bm{n}}^{(\alpha, \bm c)}(x) \isdef \sum_{j=1}^r \mathcal L_{\bm{n},i}^{(\alpha, \bm c)}(x)\,  x^{\alpha}e^{-c_jx},
\end{equation*}
namely, 
\begin{equation}  \label{MLtypeIkind2bis}
 \int_0^{+\infty}        Q_{\bm{n}}^{(\alpha, \bm c)}(x) x^{ k}  \, dx = 0, \qquad 0 \leq k \leq |\bm{n}|-2,
\end{equation}
together with the normalization
$$
 \int_0^{+\infty}   Q_{\bm{n}}^{(\alpha, \bm c)}(x) x^{ |\bm{n}|-1}  \, dx = 1.
$$
Notice that in this case, the weights $w_j$ in \eqref{Lkind2weights} no longer form a Nikishin system, although they are an AT-system, see, e.g.~\cite[\S 23.4]{MR2542683}. 

It was shown\footnote{\, For an alternative formula for the case of $r=2$, see also \cite[Proposition 8.1]{branquinho2023hahn}.} in \cite[Theorem 4.1]{JPOP24} that the Type I multiple Laguerre polynomials of the second kind are multivariable Kampé de Fériet polynomials: for $\bm c=(c_1, \dots, c_r)$ and $i=1,\dots,r$,
\begin{equation}
    \label{eq:KampeforLaguerre}
\mathcal L_{\bm{n},i}^{(\alpha, \bm c)}(x)\simeq 
F_{ 1: 0 ; 1;\cdots;1}^{1: 0 ;0;\cdots ;0}\left[\left.\begin{array}{c}
-n_i+1: \cdot \ ; \hat{\bm n}_i \\
 \alpha +1+|\bm n|-n_i:\cdot \ ;\cdots \ ; \cdot 
\end{array} \right\rvert\, c_i x, \bm d_i\right] ,
\qquad \bm d_i\isdef\frac{c_i}{c_i-\hat{\bm c}_i}\neq 0.
\end{equation}
Hence, as an application of Theorem~\ref{thm:multipleKdFfactorization} we conclude that 
\begin{proposition} \label{prop:zerosTypeILagsecondKind}
The Type I Multiple Laguerre polynomials of the second kind can be expressed as an additive convolution:
\begin{equation}
    \label{decompTypeILaguerrekind1}
\mathcal L_{\bm{n},i}^{(\alpha, \bm c)}\simeq 
p_{i,1}\boxplus_{n_i-1} \cdots \boxplus_{n_i-1}  p_{i,r}, 
\end{equation}
where
\begin{equation}\label{p.JP.Generalnew}
  p_{i,j}(x)=\begin{cases}
   \displaystyle    \HGF{1}{1}{-n_i+1}{-n_j-n_i+2}{(c_i-c_j)x} , &  j\in \{1,\cdots,r\}\setminus \{i\},  \\[4mm]
   \displaystyle  \HGF{1}{1}{-n_i+1}{ \alpha +1+|\bm n|-n_i}{c_ix} , & j=i.
       \end{cases}
\end{equation}
\end{proposition}

\begin{proof}
By Theorem~\ref{thm:multipleKdFfactorization} (and the symmetry of the entries) we know that
$$
\mathcal L_{\bm{n},i}^{(\alpha, \bm c)}\simeq q_i \boxtimes_n (q_0\boxplus_n q_1\boxplus_n \cdots \boxplus_n q_{i-1}\boxplus_n q_{i+1}\boxplus_n \cdots \boxplus_n q_r), 
$$
where
\begin{align*}
q_0(x)& \isdef \HGF{1}{1}{-n}{ \alpha +1+|\bm n|-n_i}{-x}, \qquad  q_i(x) \isdef \HGF{1}{0}{-n}{\cdot}{-c_ix},\quad \text{and} \\
q_j(x)& \isdef \HGF{1}{1}{-n}{-n_j-n_i+2}{-\frac{x(c_i-c_j)}{c_i}}  \qquad \text{for }j\in\{1,\dots, r\}\setminus\{i\}.
\end{align*}
Notice that the multiplicative convolution by $q_i$ is just a dilation by $\frac{1}{-c_i}$, and, by \eqref{identityAdd3a}, this dilation can be distributed among each term in the additive convolution. Therefore, if we let 
$p_{i,i}=\dil{\frac{1}{-c_i}}q_0$ and $p_{i,j}\dil{\frac{1}{-c_i}}=q_j$ for $j\neq i$, we conclude the desired result.
\end{proof}

\subsubsection{Zero asymptotics of Type I Multiple Laguerre polynomials of the second kind}\ \label{sec:asymptTypeILagueresecond}

Unfortunately, the decomposition in Proposition \ref{prop:zerosTypeILagsecondKind} does not allow us to conclude that the zeros of $\mathcal L_{\bm{n},i}^{(\alpha, \bm c)}$ are real. However, we can still obtain information about any accumulation point of the normalized zero-counting measures of these polynomials. 

Let $\bm n =(n_1, \dots, n_r)\in \N^r$ and $\alpha_{\bm n} >-1$ be such that
\begin{equation}
    \label{asymptMLIkind2assumptions}
    \lim_{|\bm n|\to \infty} \frac{\alpha_{\bm n}}{|\bm n|}=A\ge 0,    \quad \text{and} \quad \lim_{|\bm n|\to \infty} \frac{n_i}{|\bm n|}=\theta_i>0, \qquad i=1,\dots,r.
\end{equation}
Moreover, assume that $\bm c_{\bm n}=\left(c_1^{(\bm n)}, \dots, c_r^{(\bm n)}\right)\in \R_{>0}^r$, with $c_i^{(\bm n)}\neq c_j^{(\bm n)}$ for $i\neq j$, depends on the multi-index $\bm n$ in such a way that the limit
\begin{equation}
    \label{asymptMLIkind2assumptions2}
    \lim_{\bm n} \bm c_{\bm n} = \bm c =(c_1, \dots, c_r)\in \R_{>0}^r
\end{equation}
exists, with all $c_i\neq c_j$ for $i\neq j$. In consequence, the vectors $\bm d_i$ in \eqref{eq:KampeforLaguerre} depend also on the multi-indices $\bm n$, but have finite limits that we will denote again by $\bm d_i$: in other words, $\bm d_i^{(\bm n)} \to \bm d_i$, $i=1,\dots,r$.
 
\begin{remark} \label{rem:scalingLaguerre}
It is easy to see that a simple re-scaling $x\mapsto x/N$, $N=N(\bm n)$, in \eqref{MLtypeIkind2bis} allows us to tackle the zero asymptotics also in the case of parameters $\bm c$ satisfying
$$
 \lim_{\bm n} \frac{\bm c_{\bm n}}{N} = \bm c = (c_1, \dots, c_r)\in \R_{>0}^r
$$
Notice that under this assumption, the limits $\bm d_i^{(\bm n)} \to \bm d_i$, $i=1,\dots,r$, still hold.
\end{remark}

\begin{theorem}
Under the assumptions above, the Cauchy transform $y=\mathcal G_\mu(u)$ of the limiting distribution of zeros of the rescaled polynomials $ \mathcal L_{\bm{n},i}^{(\alpha, \bm c)}(n_i \, x)$, satisfies the equation
\begin{equation}
    \label{eq:CauchyTypeILagKind2}
 u=\frac{1}{y}-\frac{A+1}{\theta_i(y-c_i)}+\sum_{\substack{1\leq j\leq r\\ j\neq i}} \frac{\theta_j}{\theta_i(y-c_i+c_j)}.
\end{equation}
\end{theorem}
\begin{proof}
By the representation \eqref{p.JP.Generalnew} and  \eqref{identityAdd3a}, it follows that
\begin{align*}
  \mathcal  L_{\bm{n},i}^{(\alpha_{\bm n},\bm{c}_{\bm n})}((n_i-1)x)&=
\dil{\frac{1}{n_i-1}}(p_{i,1}\boxplus_{n_i-1}p_{i,2}\boxplus_{n_i-1}\dots \boxplus_{n_i-1} p_{i,r})(x)\\
& = \left[(\dil{\frac{1}{n_i-1}} p_{i,1}) \boxplus_{n_i-1}(\dil{\frac{1}{n_i-1}} p_{i,2})\boxplus_{n_i-1}\dots \boxplus_{n_i-1} (\dil{\frac{1}{n_i-1}} p_{i,r})\right](x).
\end{align*}
From \eqref{identityMult2}, 
\begin{align*}
&\left[\dil{\frac{1}{n_i-1}} p_{i,j}\right](x)  =p_{i,j}((n_i-1) x) \\
& \simeq \begin{cases}
   \displaystyle    \HGF{1}{1}{-n_i+1}{-n_j-n_i+2}{(c_i-c_j)(n_i-1)x} , &  j\in \{1,\cdots,r\}\setminus \{i\},  \\[4mm]
   \displaystyle  \HGF{1}{1}{-n_i+1}{ \alpha_{\bm n} +1+|\bm n|-n_i}{c_i(n_i-1)x} , & j=i, 
       \end{cases}
\\
&\simeq \begin{cases}
   \displaystyle    \left(x-\frac{1}{c_i-c_j}\right)^{n_i-1} \boxtimes_{n_i-1} \HGF{1}{1}{-n_i+1}{-n_j-n_i+2}{(n_i-1)x} , &  j\in \{1,\cdots,r\}\setminus \{i\},  \\[4mm]
   \displaystyle   \left(x-\frac{1}{c_i}\right)^{n_i-1} \boxtimes_{n_i-1}\HGF{1}{1}{-n_i+1}{ \alpha_{\bm n} +1+|\bm n|-n_i}{(n_i-1)x} , & j=i.
       \end{cases} 
\end{align*}
Under assumptions \eqref{asymptMLIkind2assumptions}--\eqref{asymptMLIkind2assumptions2}, the normalized zero count measure of the first polynomial on the right-hand side tends to a Dirac delta (mass point) at $ 1/(c_i-c_j)$, if $j\neq i$, and at $1/c_i$ for $j=i$. Since the $S$-transform of $\delta_a$, $a\neq 0$, is the constant function $1/a$, applying Theorem \ref{thm:StranfHyper2F1}, we conclude that the weak-* limit $\nu_{i,j}$ of the normalized zero-counting measures of the scaled polynomials $p_{i,j}((n_i-1) x)$ is a positive probability measure, compactly supported on the real line, for which 
\begin{equation}
   \mathcal S_{\nu_{i,j}}(y)= \begin{cases} 
   \dfrac{c_i-c_j}{y- \theta_j/\theta_i},&  j\in \{1,\cdots,r\}\setminus \{i\},  \\[4mm] 
   \dfrac{c_i}{y+( A+1)/ \theta_i}, & j=i.
       \end{cases} 
\end{equation}
From \eqref{RtransformStranform} it follows that its $R$-transform is given by, 
$$
\mathcal R_{\nu_{i,j}}(y)=\begin{cases}\dfrac{-\theta_j}{\theta_i(c_i-c_j-y)},&  j\in \{1,\cdots,r\}\setminus \{i\},  \\[4mm] 
\dfrac{A+1}{\theta_i(c_i-y)}, & j=i.
       \end{cases} 
$$

Applying \eqref{def:additiveconvolutionmeasures}, we see that the normalized zero-counting measure of the rescaled $\mathcal  L_{\bm{n},i}^{(\alpha_{\bm n},\bm{c}_{\bm n})}$ converges in moments to a measure $\mu$, whose $R$-transform is
\begin{align*}
 \mathcal R_{\mu}(y)&=-\frac{A+1}{\theta_i(y-c_i)}+\sum_{\substack{1\leq j\leq r\\ j\neq i}} \frac{\theta_j}{\theta_i(y-c_i+c_j)}.
\end{align*}
Thus, its $K$-transform is 
$$
 \mathcal K_{\mu}(y)=\frac{1}{u}-\frac{A+1}{\theta_i(y-c_i)}+\sum_{\substack{1\leq j\leq r\\ j\neq i}} \frac{\theta_j}{\theta_i(y-c_i+c_j)},
$$
so that the Cauchy transform $y=\mathcal G_\mu(u)$ of $\mu$ is the solution of the algebraic equation
$$
 u=\frac{1}{y}-\frac{A+1}{\theta_i(y-c_i)}+\sum_{\substack{1\leq j\leq r\\ j\neq i}} \frac{\theta_j}{\theta_i(y-c_i+c_j)},
$$
as stated.
\end{proof}

%%%%%%%%%%%%%%%%%%%%%%%%%%%%%%%%%%%%%%%%%%%%%%%%%%%%%%
\section{Type II Multiple orthogonal polynomials} \label{sec:MOPtype2}
%%%%%%%%%%%%%%%%%%%%%%%%%%%%%%%%%%%%%%%%%%%%%%%%%%%%%%

Recall that Type II multiple orthogonal polynomial $P_{\bm{n}}$ are monic polynomial of degree $|\bm{n}|$ that satisfy 
the orthogonality conditions \eqref{typeII}, namely,
$$  
\int P_{\bm{n}}(x) x^k  w_j(x)\, dx = 0, \qquad  0 \leq k \leq n_j-1, \quad 1 \leq j \leq r.
 $$
 These polynomials are better studied than their Type I counterparts. In particular, it is known that their asymptotic zero distribution can be described via a solution of a vector equilibrium problem. Now we can perform this analysis using our free probability tools.

%%%%%%%%%%%%%%%%
\subsection{Type II Jacobi-Piñeiro polynomials}\ \label{subsec:JPTypeII}
%%%%%%%%%%%%%%%%

We consider again the AT-system of Jacobi-Piñeiro weights on $[0,1]$ that satisfy the conditions
\begin{equation}\label{assumptionsonalphasTypeII}
    \alpha_1,\dots,\alpha_r>-1  \quad \text{and} \quad   \alpha_i-\alpha_j\not\in\Z \text{ for } i\not=j. 
\end{equation}
 The Type II (monic) Jacobi-Piñeiro polynomials, $P^{(\bm \alpha;\beta)}_{\bm{n}}$, $\deg P^{(\bm \alpha;\beta)}_{\bm{n}}  = |\bm n|$,  are given by the orthogonality conditions  
\begin{equation}  \label{JPtypeII}
 \int_0^1        P_{\bm{n}}^{(\bm \alpha;\beta)}(x) \,  x^{\alpha_j+k}(1-x)^\beta  \, dx = 0, \qquad 0 \leq k \leq n_j-1, \quad j=1, \dots, r.
\end{equation}
Once again, these are hypergeometric polynomials (see \cite[\S 23.3.2]{MR2542683}):
\begin{equation} \label{JacobiPineiroDef}
\begin{split}
 (-1)^{|\bm{n}|} 
 \prod_{j=1}^r &
\frac{ \raising{|\bm{n}| +\alpha_j+\beta+1}{n_j}}{\raising{\alpha_j+1}{n_j}} \, (1-x)^\beta P_{\bm n}^{(\bm{a} ; \beta)}(x) =  
\HGF{r+1}{r}{
-|\bm{n}|-\beta, \bm \alpha+\bm n+1 }{
\bm \alpha+1}{x}.
\end{split}
\end{equation}

Taking $\beta\in\Z_{\geq 0}$  on \eqref{JacobiPineiroDef} and using Theorem \ref{thm:multiplicativeConvHG} from Section \ref{sec:finitefreeconv}, we obtain the following representation:
\begin{equation}\label{JacobiPineiro}
 (1-x)^\beta P_{\bm n}^{(\bm{a} ; \beta)}(x)\simeq p_{1}(x) \boxtimes_{|\bm n|+\beta} \dots \boxtimes_{|\bm n|+\beta} p_{r}(x),
\end{equation}
where
\begin{equation} \label{pPolyns1}
\begin{split}
p_{j}(x)&= \frac{\raising{\alpha_j+1}{n_j}}{(n_j)!}\, \HGF{2}{1}{-|\bm{n}|-\beta , \alpha_j+n_j+1}{\alpha_j+1}{x}   \\
&= (1-x)^{|\bm n|+\beta-n_j} \,\HGF{2}{1}{-n_j,|\bm n|+\beta+\alpha_j+1}{\alpha_j+1}{x}
 , \quad j=1, \dots, r. 
\end{split}
\end{equation}
These identities allow us to follow the methodology described above to find the asymptotic zero distribution of these polynomials. To deal with the case of noninteger $\beta$, we appeal to an alternative expression for the reciprocals of the Type II Jacobi-Piñeiro polynomials, consequence of Lemma \ref{lemma:trick}:
\begin{proposition} \label{propReprJPTypeII}
For the Type II Jacobi-Piñeiro polynomials,
\begin{equation} \label{representationReversedJPII}
    \begin{split}
    P_{\bm n}^{(\bm{a} ; \beta)}(x) \simeq \, p^*(x) 
\boxtimes_{|\bm n|} \left[\HGF{1}{1}{-|\bm n|}{-\beta-|\bm n|+1}{ |\bm n|x} \boxplus_{|\bm n|}\left(
\HGF{1}{1}{-|\bm n|}{\beta+1}{|\bm n|x}\boxtimes_{|\bm n|} q\right)\right]^*,
    \end{split}
\end{equation}
where
\begin{equation}
    \label{defQinProp51}
p(x)= \HGF{2}{0}{-|\bm n|,1}{\cdot }{-\frac{x}{|\bm n|}} , \quad q(x)=\HGF{r+1}{r}{-|\bm n|,\ -|\bm n|-\bm \alpha}{-|\bm n|-\bm n-\bm \alpha}{-x}.
\end{equation}
\end{proposition}
\begin{proof}
By  \eqref{JacobiPineiroDef},  and using the hypergeometric representation of $(1-x)^\beta$, we have that
$$
P_{\bm n}^{(\bm{a} ; \beta)}(x) \simeq \HGF{1}{0}{\beta}{\cdot}{x}
\HGF{r+1}{r}{ -|\bm{n}|-\beta, \alpha_1+n_1+1, \ldots, \alpha_r+n_r+1 }{
\alpha_1+1, \ldots, \alpha_r+1}{x}. 
$$
Notice that $ P_{\bm{n}}^{(\bm \alpha;\beta)}(0)\neq 0$, so its reciprocals has the same degree. 
Applying Lemma \ref{lemma:trick}, we get that
\begin{align*}
\left(P_{\bm n}^{(\bm{a} ; \beta)}\right)^*(x) & \simeq \HGF{2}{0}{-|\bm n|,1}{\cdot}{x}
\\& \boxtimes_{|\bm n|}
\left[\HGF{1}{1}{-|\bm n|}{-\beta-|\bm n|+1}{ -x} \boxplus_{|\bm n|}
\HGF{r+1}{r+1}{-|\bm n|,\ -|\bm n|-\bm \alpha}{\beta+1,-|\bm n|-\bm n-\bm \alpha}{-x}\right].
\end{align*}
Using that $(x-1)^{|\bm n|} = (x+|\bm n|)^{|\bm n|} \boxtimes_{|\bm n|}(x+|\bm n|^{-1})^{|\bm n|} $ is an identity under the finite free multiplicative convolution (see \eqref{identityMult2}), we can rewrite $\left(P_{\bm n}^{(\bm{a} ; \beta)}\right)^*$ as the multiplicative convolution of two polynomials,
$$
p(x)=\HGF{2}{0}{-|\bm n|,1}{\cdot}{x}\boxtimes_{|\bm n|} (x+|\bm n|)^{|\bm n|}\simeq  
\HGF{2}{0}{-|\bm n|,1}{ \cdot }{-|\bm n|^{-1}x}$$ 
and
\begin{align*}
\widetilde p (x) & =(x+|\bm n|^{-1})^{|\bm n|} \boxtimes_{|\bm n|}\left[\HGF{1}{1}{-|\bm n|}{-\beta-|\bm n|+1}{ -x} \boxplus_{|\bm n|}
\HGF{r+1}{r+1}{-|\bm n|,\ -|\bm n|-\bm \alpha}{\beta+1,-|\bm n|-\bm n-\bm \alpha}{-x}\right]\\
& =  \HGF{1}{1}{-|\bm n|}{-\beta-|\bm n|+1}{ |\bm n| x } \boxplus_{|\bm n|}
\HGF{r+1}{r+1}{-|\bm n|,\ -|\bm n|-\bm \alpha}{\beta+1,-|\bm n|-\bm n-\bm \alpha}{ |\bm n| x} \\
& = \HGF{1}{1}{-|\bm n|}{-\beta-|\bm n|+1}{ |\bm n| x } \boxplus_{|\bm n|} \left[ \HGF{1}{1}{-|\bm n|}{\beta+1}{|\bm n|x}\boxtimes_{|\bm n|} q(x)\right],
\end{align*}
where $q$ was defined in \eqref{defQinProp51}.

It remains to observe that by the definition \eqref{coeffMultConv}, the operation $^*$ acts distributively on the multiplicative convolution $\boxtimes_{|\bm n|}$.
\end{proof}

\subsubsection{Real zeros: monotonicity and interlacing}

\begin{theorem}\label{interlacingJP}
    Let a multi-index $\bm n\in\N^{r}$, $\beta\in\Z_{\geq 0}$ and  $\bm \alpha=(\alpha_1,\dots,\alpha_r)$ satisfy \eqref{assumptionsonalphasTypeII}. For each $i\in\{1,\dots,r\}$ and  $0\le t \leq 2$  such that 
     \begin{equation}\label{conditioninterlacing}
         \alpha_i-\alpha_j+t\not \in\Z, \quad j\not=i,
     \end{equation}
     the following interlacing holds:
\begin{equation}
          P_{\bm n+\bm e_i}^{(\bm \alpha, \beta)}(x)\preccurlyeq P_{\bm n}^{(\bm \alpha , \beta)}\preccurlyeq P_{\bm n}^{(\bm \alpha+t \bm e_i, \beta)}          .
\end{equation}
\end{theorem}

Recall that $\bm e_i\in \N ^r$ is the multi-index whose only non-zero entry (equal to $1$) is in the position $i$. 

\begin{proof}
Take $i\in\{1,\dots,r\}$ and define the following polynomial, 
\begin{equation}\label{Qpolynomial}
     Q_{n_i}^{(\alpha_i,\beta)}(x)=\HGF{2}{1}{-n_i,|\bm n|+\beta+\alpha_i+1}{\alpha_i+1}{x}
\end{equation}
and using the decomposition \eqref{pPolyns1} of $ P_{\bm n}^{(\bm \alpha, \beta)}$ define the following polynomial,
\begin{equation}\label{qpolynomialJP}
    q_i(x):=p_{1}\boxtimes_{|\bm n|+\beta}\cdots\boxtimes_{|\bm n|+\beta}  p_{i-1}\boxtimes_{|\bm n|+\beta} p_{i+1} \boxtimes_{|\bm n|+\beta} \cdots\boxtimes_{|\bm n|+\beta} p_{r}
\end{equation}
 By hypothesis, we have, $\alpha_j, \beta>-1$ and  $\bm n\in \N^r$, thus
    \begin{equation*}
        \alpha_j+\beta+|\bm n|>\alpha_j+n_j-1, \qquad j\in\{1,\dots,r\}
    \end{equation*}
then we have by \cite[Theorem 1, \textit{(ii)}]{dominici2013real} that $Q_{n_i}^{\alpha_i,\beta},p_{j}\in\P_{|\bm n|+\beta}((0,1))$. Applying Proposition \ref{prop:realrootedness} to polynomials $q_i$ defined in \eqref{qpolynomialJP}, we get that $q_i\in\P_{|\bm n|+\beta}((0,1))$. 

Using the the interpretation of Jacobi Polynomials as $_2F_1$, we can extract from \cite[Theorem 2.2]{driver2008interlacing} that for $0<t \leq 2$,
\begin{equation*}
    Q_{n_i+1}^{\alpha_i,\beta}\prec Q_{ n_i}^{\alpha_i,\beta} \prec Q_{n_i}^{\alpha_i+t,\beta}
\end{equation*}
Using Proposition \ref{lem:preservinginterlacingMult}, we get that
\begin{equation*} 
\begin{split}
       (1-x)^{|\bm n|+\beta-n_i-1}  Q_{ n_i+1}^{\alpha_i,\beta}\boxtimes_{|\bm n|+\beta}q_i&\preccurlyeq  (1-x)^{|\bm n|+\beta-n_i} Q_{n_i}^{\alpha_i,\beta}\boxtimes_{|\bm n|+\beta}q_i\\ &\preccurlyeq  (1-x)^{|\bm n|+\beta-n_i} Q_{n_i}^{\alpha_i+t,\beta}\boxtimes_{|\bm n|+\beta}q_i.
\end{split}       
\end{equation*}
 
For $0<t\leq 2$ satisfying \eqref{conditioninterlacing}, we conclude that 
$$
       P_{\bm n+\bm e_i}^{(\bm \alpha, \beta)}(x)\preccurlyeq P_{\bm n}^{(\bm \alpha , \beta)}\preccurlyeq P_{\bm n}^{(\bm \alpha+t\bm e_i, \beta)}.
$$
\end{proof}
\begin{remark}
   From \cite[Theorem 2.2]{driver2008interlacing}, the polynomial defined in \eqref{Qpolynomial} satisfies
   \begin{equation*}
    Q_{n_i+1}^{\alpha_i,\beta}\prec Q_{ n_i}^{\alpha_i,\beta+s} \prec Q_{n_i}^{\alpha_i,\beta}  \quad \text{ for }  0\leq s\leq2   .
\end{equation*} 
Using similar ideas from the proof of Theorem \ref{interlacingJP}, we can obtain
\begin{equation}\label{Interlacingbeta}
          P_{\bm n+\bm e_i}^{(\bm \alpha, \beta)}(x)\preccurlyeq P_{\bm n}^{(\bm \alpha , \beta+s)}\preccurlyeq P_{\bm n}^{(\bm \alpha, \beta)}   \quad \text{ and } \quad s=0,1,2       .
\end{equation}
The restriction of $s$ to entire numbers is necessary in order to use \eqref{JacobiPineiro}.
The interlacing stated in Theorem \ref{interlacingJP} and in \eqref{Interlacingbeta} was partially proved in \cite[Theorem 2.2]{haneczok2012interlacing}, where it was shown that 
$$
 P_{\bm n+\bm e_i}^{(\bm \alpha, \beta)}(x)\preccurlyeq P_{\bm n}^{(\bm \alpha, \beta)}.
$$
\end{remark}

\subsubsection{Zero asymptotics of Type II Jacobi-Piñeiro polynomials}\

We can describe the weak zero asymptotics of a sequence of Type II Jacobi-Pi\~neiro polynomials, 
$$
p_{\bm n }=  P_{\bm n}^{(\bm{a} ; \beta)}(x),
$$
where $\bm \alpha_{\bm n} = \left(\alpha_1^{(\bm n)}, \dots, \alpha_r^{(\bm n)} \right)$, $\bm n =(n_1, \dots, n_r)\in \N^r$, under assumption \eqref{assumptionsonalphasTypeII}, that is,
\begin{equation}
    \label{asymptJPIIassumptions}
    \lim_{|\bm n|\to \infty} \frac{\alpha_i^{(\bm n)}}{|\bm n|}=A_i\ge 0, \quad \lim_{|\bm n|\to \infty} \frac{\beta^{(\bm n)}}{|\bm n|}=B\ge 0,  \quad \text{and} \quad \lim_{|\bm n|\to \infty} \frac{n_i}{|\bm n|}=\theta_i>0, \qquad i=1, \dots, r.
\end{equation}

Let us suppose initially that for all $\bm n$,
\begin{equation}
    \label{integerBeta}
\beta^{(\bm n)} \in \Z_{\ge 0}.
\end{equation}
Representation \eqref{JacobiPineiroDef} indicates that we need to consider the parameters 
\begin{equation}
    \label{defABfrakII}
\mathfrak a_j \isdef   (A_j+\theta_j)/(1+B),  \quad \text{and}\quad \mathfrak b_j \isdef  
     A_j   /(1+B),  \quad j=1, \dots, r.
\end{equation}
Restrictions $\mathfrak a_j \notin [-1,0)$ and $\mathfrak b_j\neq -1$ and the assumption that $\mathfrak a_i \neq \mathfrak b_j$ for $i\neq j$ hold automatically.

Let $\mu$ be a weak-* accumulation point $\mu$ of the normalized zero-counting measures $\chi(p_{\bm n, i})$, and denote
$$
\widetilde\mu = \frac{1}{1+B}\left( B \delta_0 +\mu \right).
$$
By Theorem~\ref{thm:StranfHyper2F1}, the $S$-transform of $\widetilde\mu$ is   
\begin{equation}
    \label{SfunctionForJPII}
   \mathcal S_{\widetilde\mu}(z)= \prod_{j=1}^r  \frac{ z+\mathfrak a_j^{(i)}+1 }{ z+\mathfrak b_j^{(i)}+1  } = \prod_{j=1}^r  \frac{ (1+B) z+  \theta_j+ A_j  +B+1 }{ (1+B) z+A_j  +B+1 } ,
\end{equation}
while, by Theorem \ref{cor:CauchyTransf}, 
$$
y = u \, \mathcal G_{\widetilde\mu}(u)= \frac{u}{1+B}\left( \frac{B}{u} + \mathcal G_\mu(u) \right)= \frac{1}{1+B}\left(  B  + u \, \mathcal G_\mu(u) \right)
$$
is a solution of 
$$  
    y \prod_{j=1}^r \left(y +\frac{A_j}{1+B} \right) = u (y-1)   \prod_{j=1}^r  \left(y +\frac{A_j+\theta_j}{1+B} \right).
$$  

Let us discuss the case when $\bm \alpha$ and $\beta$ do not depend on $\bm n$, so that $A=B=0$ and $\widetilde\mu =\mu$. In this case,
$$
  \mathcal S_\mu(z)=  \frac{1}{ (z+1)^r}  \,  \prod_{j=1}^r   \left(   z+  \theta_j +1 \right),
$$
and $y = u \, \mathcal G_\nu(u)$ is a solution of 
$$
y^{r+1} = u (y-1)   \prod_{j=1}^r  \left(y + \theta_j  \right).
$$
In the asymptotically diagonal situation, when all $\theta_j=1/r$, this equation boils down to 
$$
y^{r+1} = u (y-1)     \left(y + 1/r  \right)^r
$$
(compare it to \eqref{eqCauchyTrTypeIJPdiagonal}; these equations can be reduced to each other with the change $y\mapsto -y/r$; an equivalent equation appeared in \cite{MR3471160}). By the Sokhotski-Plemelj Formula, the density $\mu'$ of $\mu$ on $(0,1)$ can be recovered as
$$
\mu'(x)=\frac{1}{2\pi i x}  \left(y_- ( x)-y_+ ( x) \right) =- \frac{1}{ \pi x } \Im \left(y_+ ( x) \right)  , \quad x\in (0,1).
$$
\begin{example}\label{exampleJPII}
If $r=2$, and $\theta_1 = \theta\le 1/2$, the Cauchy transform $\mathcal G_\mu(u)$ of the limit measure $\mu$ is given by $y/u$, where $y=y(u)$ solves the equation
\begin{equation}
    \label{cubic2}
y^3 = u (y-1)    \left(y + \theta   \right)  \left(y +1- \theta   \right) = u \left( y^3- \left(1+ \nu  \right)   y+ \nu \right) , \quad \nu = \theta (\theta-1) \in (-1/4,0).
\end{equation}
This is the same cubic equation \eqref{cubic} we obtained in Example \ref{example44}, up to the modification $\theta \mapsto 1/\theta$, and the corresponding change of the value of $\nu$. Thus, we can use our previous calculations, selecting the branch such that $y(u)>0$ for $u<0$.  
Namely, if now 
$$
f(u)=\left(\frac{\theta (\theta-1)u}{2(1-u)}  \right)^{1/3}
$$
is taken analytic in $\C\setminus\left[0, +\infty\right)$, whose single-valued branch is fixed by requiring $f(x)>0$ for $x<0$,  then the desired solution of \eqref{cubic2} is
$$
y(u)=   f(u) \left( \left(1+\sqrt{1+  \frac{\kappa \,u}{1-u}}  \right)^{1/3} +\left(1-\sqrt{1+  \frac{\kappa \,u}{1-u}}  \right)^{1/3} \right) ,
$$
where
$$
\kappa =\kappa(\nu)=\frac{4}{27}\, \frac{(1+\nu)^3}{\nu^2} = \frac{4 (1+(\theta -1) \theta )^3}{27\theta ^2
   (1-\theta )^2 }\ge1
$$
and with the branches of the roots taken always positive for $u\in (-\varepsilon, 0)$ for a sufficiently small $\varepsilon>0$. Notice that $\kappa(\nu)$ is a strictly increasing function on $(-1/4,0))$, with $\kappa(-1/4)=1$.

For $u=x>0$, denote by $y_\pm$ the boundary values of $y$ on $(0,1)$. As before, 
the Sokhotski-Plemelj Formula gives us the density of the limit measure $\mu$ as
\begin{align*}
    \mu'(x) & =\frac{1}{2\pi i x} \left(y_- (x)-y_+ (x) \right) = -\frac{1}{ \pi x } \Im \left(y_+ (x) \right)  .
\end{align*}
We have that
$$
\left(\sqrt{1+  \frac{\kappa \,u}{1-u}}+1 \right)^{1/3}_+ \bigg|_{u=x} =  \left(\sqrt{1+ \frac{\kappa \, x}{1-x}}+1 \right)^{1/3} >0 \text{ for } x\in (0,1),
$$
and, taking into account that $1-\sqrt{1+u}=-\frac{u}{2}\left(1+\mathcal O(u)\right)$ as $u\to 0$, 
$$
\left(1-\sqrt{1+  \frac{\kappa \,u}{1-u}}  \right)^{1/3}_+ \bigg|_{u=x} =  \left(  \sqrt{1+ \frac{\kappa \, x}{1-x}}  -1\right)^{1/3}e^{-\pi i /3} \text{ for } x\in (0,1 ).
$$
Since
$$
f_+(x)=\sqrt[3]{\frac{-\nu  x}{2(1-x)}}\, e^{-\pi i /3}, \quad x\in(0,1),
$$
we get that on $(0,1)$,
$$
y_+(x)= \sqrt[3]{\frac{-\nu  x}{2(1-x)}}\, e^{-\pi i /3} \left( \left(\sqrt{1+ \frac{\kappa \, x}{1-x}}+1 \right)^{1/3} + \left(  \sqrt{1+ \frac{\kappa \, x}{1-x}}  -1\right)^{1/3}e^{-\pi i /3} \right) .
$$
Putting all together, we conclude that $\mu$ is supported on $[0,1]$, with
\begin{equation}
    \label{densityJPII}
    \begin{split}
     \mu'(x)   & =  \frac{\sqrt{3}}{2\pi }\, \sqrt[3]{\frac{\theta(1-\theta)}{2}}\,  
     \,\frac{ \left( \sqrt{1+(\kappa -1)x}+ \sqrt{1-x}  \right)^{1/3} + \left( \sqrt{1+(\kappa -1)x}- \sqrt{1-x}  \right)^{1/3}  }{x^{2/3} \sqrt{1-x}}.
     \end{split}
\end{equation}
Symbolic integration allows us to check that $\mu'(x)$ is the density of a positive probability measure on $[0,1]$, and that 
$$
\mu'(x)=\frac{\sqrt{3}
   \sqrt[3]{ \theta(1-\theta ) )}}{2 \pi \,  x^{2/3}}(1+o(1)) \text{ as } x\to 0+, \quad \text{and} \quad \mu'(x)=\frac{\sqrt{1-\theta(1-\theta) }
  }{\pi  \sqrt{1-x}}  (1+o(1)) \text{ as } x\to 1-.
$$

See Figure \ref{fig:3} for a histogram of the zeros for $\theta=1/3$.

In the diagonal case, when $\theta=1/2$ (and $\kappa=1$), this formula coincides with the expression 
\begin{equation}
    \label{densityJPIIDiag}
     \mu'(x)  =  \frac{\sqrt{3}}{4\pi }\,  
     \,\frac{ \left( 1+ \sqrt{1-x}  \right)^{1/3} + \left( 1- \sqrt{1-x}  \right)^{1/3}  }{x^{2/3} \sqrt{1-x}},
\end{equation}
found in \cite{MR3471160}, see also \cite[\S 8.4]{MR4560438}. On the other hand, as $\theta \to 0+$ or $\theta \to 1-$, $\mu$ converges to the equilibrium measure (arcsine distribution) on $[0,1]$. 
\begin{figure}
    \centering
    \includegraphics[width=.6\textwidth]{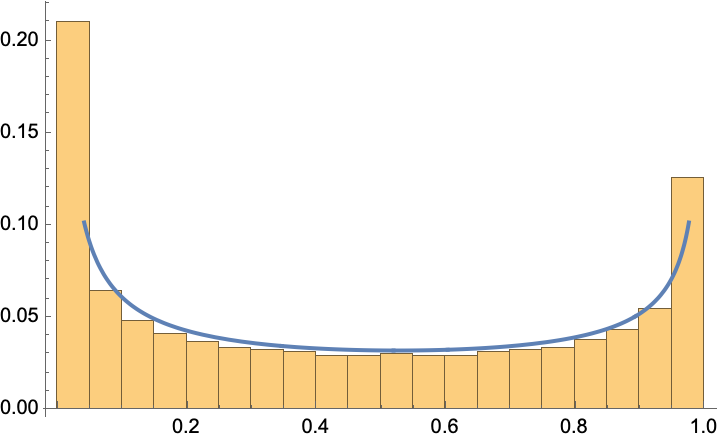}
    \caption{Histogram of the zeros of the polynomial $P^{(\bm \alpha ;\beta )}_{\bm{n}}$, with $\bm \alpha=(1/2, 3/7)$, $\beta=1$, and $\bm n = (300, 600)$, along with the corresponding density. }
    \label{fig:3}
\end{figure}
\end{example}

\begin{example}\label{exampleJPII2}
Still in the case of $r=2$ and asymptotically diagonal multi-indices ($\theta_1=\theta_2=1/2$), similar ideas allow us to tackle the case of varying parameters $\bm \alpha_{\bm n}$ and $\beta^{(\bm n)}$, satisfying \eqref{asymptJPIIassumptions}. We avoid performing cumbersome calculations here. Instead, we point out that if the case of $B=0$ and $A_1=A_2=A\ge 0$, the support of the limit measure $\mu$ is $[a^*,1]$, with
\begin{equation}
    \label{curveA}
a^* = \frac{A^3(A+1)}{\left(A+\frac{3}{2} \right)^3 \left(A+\frac{1}{2} \right)} \in [0,1).
\end{equation}
On the other hand, if $A_1=A_2=0$ and $B\ge 0$, the measure is supported on $[0,b^*]$, with
\begin{equation}
    \label{curveB}
b^* = \frac{27 (B+1)^2}{\left( 2B+3 \right)^3 } \in (0,1].
\end{equation}
See Figure~\ref{fig:5} for the result of some numerical experiments.
\begin{figure}
    \centering
    \includegraphics[width=.5\textwidth]{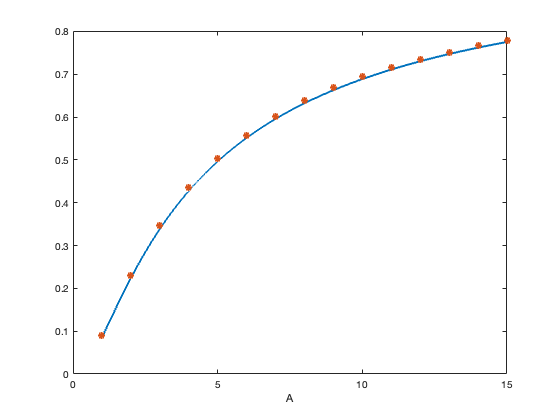}\hfill \includegraphics[width=.5\textwidth]{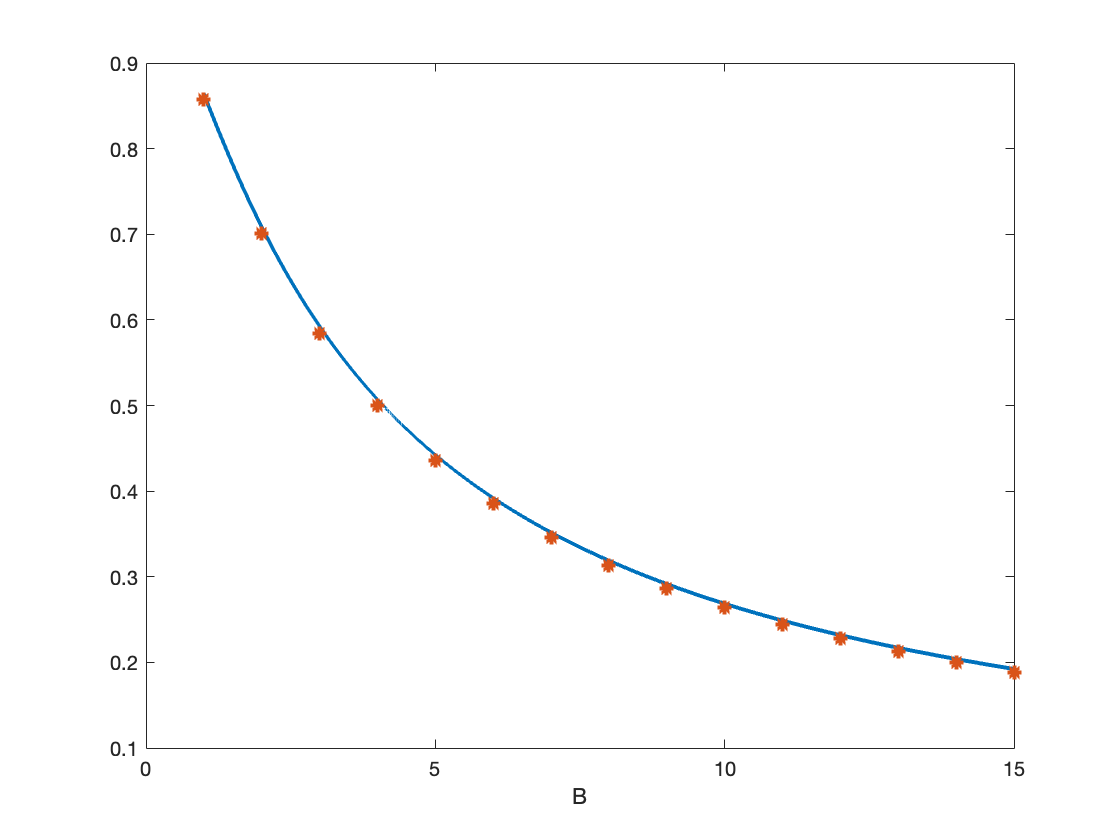}
    \caption{Left: plot of the smallest zeros of the polynomial $P^{(\bm \alpha ;\beta )}_{\bm{n}}$, with $\bm n = (300, 300)$,  $\beta=0$,  and $\bm \alpha=(600 A , 600 A)$, for $A=1, \dots, 15$, along with the predicted leftmost end-point of the support, given by \eqref{curveA}. Right: plot of the largest zeros of the polynomial $P^{(\bm \alpha ;\beta )}_{\bm{n}}$, with $\bm \alpha=(0 , 3/2)$, $\bm n = (300, 300)$, and $\beta=600 B$, for $B=1, \dots, 15$, along with the predicted rightmost end-point of the support, given by \eqref{curveB}.}
    \label{fig:5}
\end{figure}
\end{example}

All considerations above have been carried out under the assumption \eqref{integerBeta} that all $\beta$'s are integers. Consider the representation \eqref{representationReversedJPII} from Proposition~\ref{propReprJPTypeII}. We observe that for both polynomials in its right-hand side 
depending on $\beta$, their limit zero distribution 
depends only on the value of the limit $\beta/|\bm n|$, and not on the concrete values of $\beta$. Hence, the same is true for the expression of the $S$-transforms of the limit measures for $p$ and for the polynomial within the brackets on the right-hand side of \eqref{representationReversedJPII}. Finally, using \eqref{measureReversed}, we can conclude that the asymptotic zero distribution of $P_{\bm n}^{(\bm{a} ; \beta)}$ depends only on the limits \eqref{asymptJPIIassumptions} but not on the actual values of $\beta$. In other words, expression \eqref{densityJPII} is valid under the general assumptions \eqref{asymptJPIIassumptions}.

%%%%%%%%%%%%%%%%
\subsection{Type II Multiple Laguerre polynomials of the first kind}\ \label{subsec:MLTypeII}
%%%%%%%%%%%%%%%%

For $\bm{n}\in \N^r$, the Type II Multiple Laguerre polynomials of the first kind, $L_{\bm{n}}^{\bm{\alpha}}$, corresponding to the weights \eqref{Lkind1weights}, are polynomial of degree $|\bm n|-1$, satisfying 
$$
\int_0^{\infty} L_{\bm{n}}^{\bm{\alpha}}(x) x^{\alpha_j} e^{-x} x^k d x=0, \quad k=0,1, \ldots, n_j-1, \quad j=1, \dots, r,
$$
where parameters $\alpha_1, \ldots, \alpha_r>-1$ are such that $\alpha_i-\alpha_j \notin \mathbb{Z}$   whenever $ i \neq j $. These are hypergeometric polynomials (see \cite[\S 23.4.2]{MR2542683}):
\begin{equation} \label{MultLagerre1Def}
\begin{split}
(-1)^{|\bm{n}|} e^{-x} L_{\bm{n}}^{\bm{\alpha}}(x)=\prod_{j=1}^r\raising{\alpha_j+1}{n_j}\,   \HGF{r}{r}{\bm \alpha+\bm n +1} {\bm \alpha +1}{-x}.
\end{split}
\end{equation}
Notice that in this identity, neither the left-hand nor the right-hand side is a polynomial. However, the polynomials $L_{\bm{n}}^{\bm{\alpha}}$ do not vanish at the origin, and we can write their reciprocal as a finite convolution:

\begin{proposition}
    \label{prop:representationTypeIILaguerrekind1}
    For the reciprocal $p^*$ of the Type II Multiple Laguerre polynomials of the first kind $p(x)=L_{\bm{n}}^{\bm{\alpha}}$,
    \begin{equation}
        \label{eq:representationTypeIILaguerrekind1}
        p^*\simeq \HGF{2}{0}{-|\bm n|,1}{ \cdot }{x}\boxtimes_{|\bm n|}
\HGF{r+1}{r}{-|\bm n|,\ -|\bm n|-\bm \alpha}{-|\bm n|-\bm n-\bm \alpha}{x+1}.
\end{equation}
\end{proposition}
\begin{proof}
Since
$$
e^{x}= \HGF{0}{0}{\cdot}{\cdot}{x}, 
$$
we can rewrite \eqref{MultLagerre1Def} as 
$$
p(x) \simeq \HGF{0}{0}{\cdot}{\cdot}{x} \HGF{r}{r}{\bm \alpha+\bm n +1} {\bm \alpha +1}{-x}.
$$
Applying the Lemma \ref{lemma:trick}  we get that 
$$
p^*\simeq \HGF{2}{0}{-|\bm n|,1}{\cdot}{x}\boxtimes_{|\bm n|}\left(\HGF{1}{0}{-|\bm n|}{\cdot}{-x} \boxplus_{|\bm n|}
\HGF{r+1}{r}{-|\bm n|,\ -|\bm n|-\bm \alpha}{-|\bm n|-\bm n-\bm \alpha}{x}\right).
$$
On the other hand, 
$$
\HGF{1}{0}{-|\bm n|}{\cdot}{-x} \simeq (x+1)^{|\bm n|},
$$
and from \eqref{shift} we get \eqref{eq:representationTypeIILaguerrekind1}. 
\end{proof}

\subsubsection{Real zeros: monotonicity and interlacing}\

The Type II Multiple Laguerre polynomials of the first kind $ L_{\bm{n}}^{\bm{\alpha}}$ can be obtained from the Jacobi-Piñeiro polynomials $P_{\bm{n}}^{(\bm \alpha;\beta)}$ by the limit process,
$$
L_{\bm{n}}^{\bm{\alpha}}(x) \simeq \lim_{\beta\to +\infty} \beta^{|\bm n|} P_{\bm{n}}^{(\bm \alpha;\beta)}(x/\beta), 
$$
which can be easily established by taking the limit directly in the orthogonality relations \eqref{JPtypeII}. In consequence, interlacing properties of the zeros of $L_{\bm{n}}^{\bm{\alpha}}$ follow from Theorem~\ref{interlacingJP}:
 \begin{theorem}\label{interlacingLfist}
    Let a multi-index $\bm n\in\N^{r}$, and  $\bm \alpha=(\alpha_1,\dots,\alpha_r)$ satisfy \eqref{assumptionsonalphasTypeII}. For each $i\in\{1,\dots,r\}$ and  $0\le t \leq 2$  satisfying \eqref{conditioninterlacing}, 
     the following interlacing holds,
\begin{equation*}
          L_{\bm n+\bm e_i}^{(\bm \alpha)}(x)\preccurlyeq L_{\bm n}^{(\bm \alpha)}\preccurlyeq L_{\bm n}^{(\bm \alpha+t \bm e_i)}          ,
\end{equation*}
where $\bm e_i\in \N ^r$ is the multi-index whose only non-zero entry (equal to $1$) is in the position $i$. 
\end{theorem}

\subsubsection{Zero asymptotics of Type II Multiple Laguerre polynomials of the first kind}\

We can use Proposition~\ref{prop:representationTypeIILaguerrekind1} to describe the asymptotic zero distribution of the rescaled polynomials $L^{(\bm \alpha_{\bm n} )}_{\bm{n}} (|\bm n| \, x)$, under assumptions \eqref{asymptMLIassumptions}, that is,
\begin{equation}
    \label{asymptTypeIIMLIassumptions}
    \lim_{|\bm n|\to \infty} \frac{\alpha_i^{(\bm n)}}{|\bm n|}=A_i\ge 0,    \quad \text{and} \quad \lim_{|\bm n|\to \infty} \frac{n_i}{|\bm n|}=\theta_i>0, \qquad i=1, \dots, r,
\end{equation}
where $\bm \alpha_{\bm n} = \left(\alpha_1^{(\bm n)}, \dots, \alpha_r^{(\bm n)} \right)$, $\bm n =(n_1, \dots, n_r)\in \N^r$.
\begin{theorem}
    \label{prop:TypeIILag1KindAsymptitics}
Under the assumptions \eqref{asymptTypeIIMLIassumptions}, function $y= u \mathcal G_\mu(u)$, where $\mathcal G_\mu$ is the Cauchy transform of any limiting zero distribution $\mu$ of polynomials $L^{(\bm \alpha_{\bm n} )}_{\bm{n}} (|\bm n|\,  x) $, satisfies the algebraic equation
\begin{equation}
    \label{eqTypeILagKind1asympt1}
 u  \prod_{i=1}^r \left(    y + A_i+ \theta_i      -u      \right) 
 =   \left( u- y    \right)    \prod_{i=1}^r  \left(     y + A_i     -u   \right) . 
\end{equation}
\end{theorem}
In the particular case of  $\bm \alpha$ independent on $\bm n$, so that all $A_j=0$, it simplifies to
\begin{equation}
    \label{eqTypeILagKind1asympt2}
 u  \prod_{i=1}^r \left(    y  + \theta_i      -u      \right) 
 =  -   \left(     y      -u   \right)^{r+1} .  
\end{equation}
\begin{proof}
If  
$$
p_{\bm n}(x)= L^{(\bm \alpha_{\bm n} )}_{\bm{n}} (  x) ,
$$
then
$$
\left( L^{(\bm \alpha_{\bm n} )}_{\bm{n}} (|\bm n|\,  x)  \right)^* \simeq p^*_{\bm n}(x/|\bm n|).
$$
By \eqref{eq:representationTypeIILaguerrekind1},
\begin{equation}
    \label{representationpreciprLag}
        p^*(x/|\bm n|)\simeq \HGF{2}{0}{-|\bm n|,1}{\cdot}{x/|\bm n|}\boxtimes_{|\bm n|}
\HGF{r+1}{r}{-|\bm n|,\ -|\bm n|-\bm \alpha_{\bm n}}{-|\bm n|-\bm n-\bm \alpha_{\bm n}}{x+ 1}.
\end{equation}
By Corollary~\ref{cor:shifted}, the $S$-transform $\widetilde w=\mathcal S_{\widetilde  \mu}(z)$ of a weak-* limit $\widetilde \mu$ of the normalized zero-counting measures of the $_{r+1}F_r$ polynomial on the right-hand side satisfies the equation
\begin{equation}
    \label{SforIILag2shifted}
\widetilde w    \prod_{i=1}^r \left(   z \widetilde w +   z- A_i- \theta_i    \right)
 =    \left(    \widetilde w +1  \right)    \prod_{i=1}^r  \left(   z \widetilde w +  z- A_i \right).
\end{equation}

Taking into account \eqref{Sfor2F0} and the representation \eqref{representationpreciprLag}, we see that the $S$-transform $w=\mathcal S_{\mu^*}(z)$ of the normalized zero-counting measure of reversed scaled polynomials $\left( L^{(\bm \alpha_{\bm n} )}_{\bm{n}} (|\bm n|\,  x)  \right)^*$ is
$$
 \mathcal S_{\mu^*}(z)= (z+1) \mathcal S_{\widetilde  \mu}(z).
$$
Thus, making the change of variables $\widetilde w = w /  (z+1)$ in \eqref{SforIILag2shifted}, we get an equation for $w=\mathcal S_{\mu^*}(z)$:
\begin{equation}
    \label{SforIILag2shiftedBis}
 w    \prod_{i=1}^r \left(   z   w +  (z+1)\left(  z- A_i- \theta_i    \right)\right)
 =    \left(  w +z+1  \right)    \prod_{i=1}^r  \left(   z   w +  (z+1)\left( z- A_i \right)\right) .
\end{equation}
In particular, the equation for the shifted $S$-transform $w=\mathcal S_{\mu^*}(-z-1)$ is
$$
 w    \prod_{i=1}^r \left( -  (z+1)   w +   z\left( z+1+ A_i+ \theta_i    \right) \right)
 =    \left(  w -z  \right)    \prod_{i=1}^r  \left(  - (z+1)   w + z \left( z+1+ A_i \right)\right) .
$$
By \eqref{measureReversed}, the $S$-transform $w=\mathcal S_{\mu}(z)$ of the normalized zero-counting measure of the scaled polynomials $ L^{(\bm \alpha_{\bm n} )}_{\bm{n}} (|\bm n|\,  x)  $ satisfies the equation 
$$
   \prod_{i=1}^r \left(  \left( z+1+ A_i+ \theta_i    \right)z w  -(z+1)       \right) 
 =   \left(  1 -zw  \right)    \prod_{i=1}^r  \left(    \left( z+1+ A_i \right)z w -(z+1)   \right) .
$$
By \eqref{defStranform}, the equation for $u=\mathcal M_\mu^{-1}(z)$ is
$$
   \prod_{i=1}^r \left(  \left( z+1+ A_i+ \theta_i    \right) u  -1       \right) 
 =    \left(1- (z+1) u  \right)    \prod_{i=1}^r  \left(    \left( z+1+ A_i \right)  u -1   \right) . 
	$$
Recalling the definition \eqref{defMtransform} of the $M$-transform in terms of the Cauchy transform  $\mathcal G_\mu$ of $\mu$, we have that
$$
z=\frac{1}{u}  \mathcal G_\mu\left(\frac{1}{u}\right)-1;
$$
with $u\mapsto 1/u$, we arrive at the algebraic equation \eqref{eqTypeILagKind1asympt1} for $y= u \mathcal G_\mu(u)$.
\end{proof}

\begin{example}
    For $r=1$ (so that $\theta_1=1$) and $\bm \alpha$ independent on $\bm n$, we get the equation
    $$
    y^2-u y + u=0,
    $$
    which yields that
    $$
   \mathcal G_\mu(u)= \frac{1}{2}\left( 1- \sqrt{\frac{u-4}{u}}\right),
    $$
    which corresponds to the Marchenko-Pastur distribution, supported on $[0,4]$.
    
    For $r=2$ and $\bm \alpha$ independent of $\bm n$, denoting $\theta_1=\theta$, the equation \eqref{eqTypeILagKind1asympt2} reduces to
    $$
   y^3 -2 u y^2+ \left(u+1\right)u y +u \left(  \theta(1-\theta) -u \right)=0.
    $$
    The support of $\mu$ in this case is the interval $[0,c^*]$, where $c=c^*$ is the positive solution of the equation
    $$
    (1-2 \theta )^2 c^2-2 (2-9 (1-\theta  ) \theta ) c -27 (\theta -1)^2 \theta ^2 =0,
    $$
    \begin{figure}[ht]
    \centering
    \includegraphics[width=.6\textwidth]{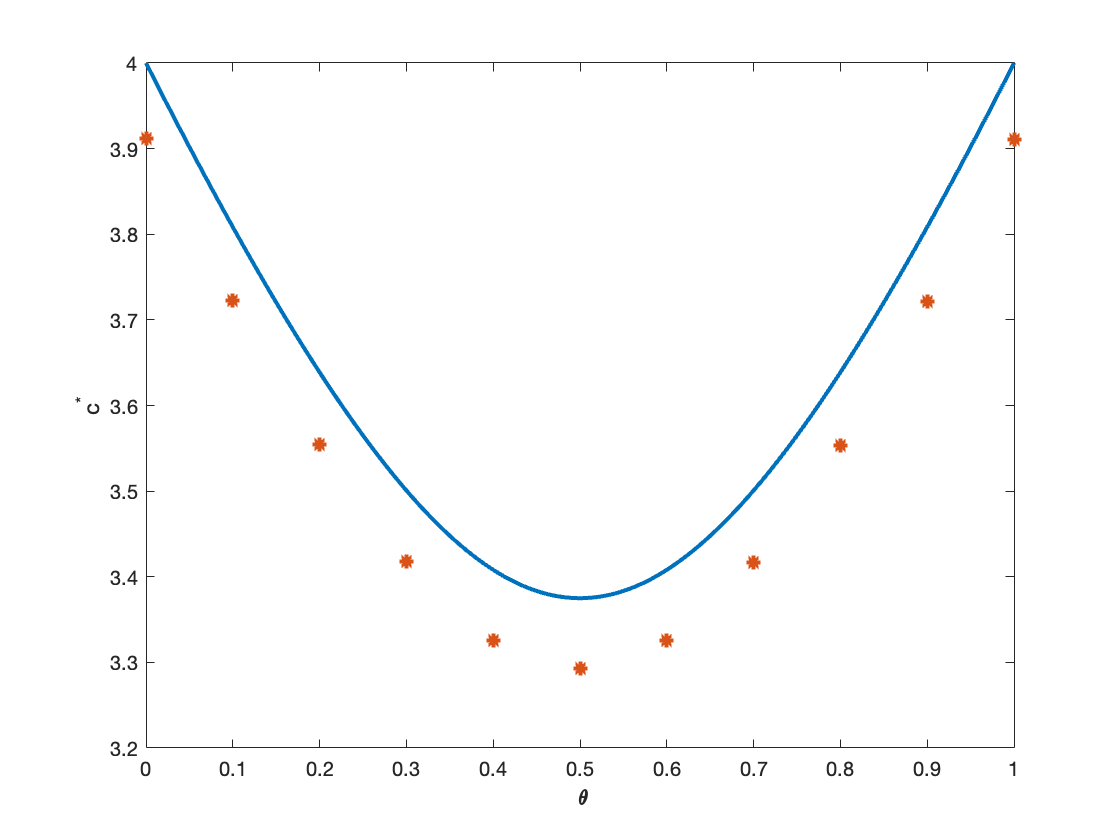}
    \caption{Graph of the predicted end-point $c^*$, given in \eqref{c*EndptIILag1k}, as function of $\theta$, along with the largest zero of the polynomial $L^{(\bm \alpha  )}_{\bm{n}} (500\,  x) $, with $\bm \alpha =(0,1/3)$ and   $\bm n = (500\, \theta  , 500\,  (1-\theta))$, for $\theta=k/10$, $k=0,1,\dots, 10$.   }
    \label{fig:EndptIILag1k}
\end{figure}
    that is
    \begin{equation}
        \label{c*EndptIILag1k}
    c^*=c^*(\theta)=\frac{27 \theta ^2(1-\theta  )^2 }{9 \theta(1-\theta ) 
   -2+2 (1-3  \theta (1-\theta ))^{3/2}}\in [27/8, 4],
\end{equation}
see Figure~\ref{fig:EndptIILag1k}.
\end{example}

    \subsection{Type II Multiple Laguerre of the second kind}\label{sec:T2ML2k}\

Let $\alpha>-1$ and  $\bm c =(c_1,\dots,c_r)\in \R_{>0}^r$ be such that $c_i\not= c_j$ for $i\not= j$.  For $\bm{n}\in \N^r$, the Type II Multiple Laguerre polynomials of the second kind, $\mathcal L_{\bm{n}}^{\alpha, \bm{c}}$, corresponding to the weights \eqref{Lkind2weights}, are polynomial of degree $|\bm n|-1$, satisfying 
$$
\int_0^{\infty} \mathcal L_{\bm{n}}^{\alpha, \bm{c}}(x) \, x^{k+\alpha} e^{-c_jx} \, d x=0, \quad k=0,1, \ldots, n_j-1, \quad j=1, \dots, r.
$$
Explicitly, see \cite[\S 23.4.2]{MR2542683},
\begin{equation}
\label{eq:TypeII.Laguerre.second}
    \mathcal    L_{\bm{n}}^{(\alpha,\bm{c})}(x)=\sum_{k_1=0}^{n_1}\cdots\sum_{k_r=0}^{n_r}(-1)^{|\bm{k}|}\binom{\abs{\bm{n}}+\alpha}{|\bm{k}|}\prod_{j=1}^r\binom{n_j}{k_j}\frac{|\bm{k}|!}{c_1^{k_1}\cdots c_r^{k_r}}\, x^{\abs{\bm{n}}-|\bm{k}|} .
\end{equation}
The change of variables $l_j=n_j-k_j$ for $j=1,\dots,r$, yields the equivalent expression (see also \cite[Eq (18)]{JPOP24} or \cite[\S 4]{BCVA2005}) 
\begin{align*}
\mathcal   L_{\bm{n}}^{(\alpha,\bm{c})}(x)&=\frac{\raising{\alpha+1}{\abs{\bm{n}}}}{c_1^{n_1}\cdots c_r^{n_r}}\sum_{k=0}^{\abs{\bm n}} \frac{x^{k}}{\raising{\alpha+1}{k}} \sum_{l_1+\cdots+l_r=k}\frac{\raising{-n_1}{l_1}c_1^{l_1}}{l_1!} \cdots \frac{\raising{-n_r}{l_r}c_j^{l_r}}{l_r!}  \\
&\simeq F_{1: 0 ;\dots; 0}^{0: 1 ;\dots;1} \left[\left.\begin{array}{c}
\quad \cdot \quad :\ - n_1\ ; \dots ;\ -n_r\  \\
\alpha+1 :\quad \cdot\quad ;\dots ; \quad \cdot \quad
\end{array} \right\rvert\,  c_1x,\dots,c_r x \right] .
\end{align*}

Finally, we can express them in terms of the finite free convolution of simpler polynomials:
\begin{theorem} \label{thm:factorization.LaguerreII.sec.kind}
For $\bm{n}\in \N^r$, $\alpha>-1$ and  $c_1,\dots,c_r>0$ such that $c_i\not= c_j$ for $i\not= j$. The Laguerre polynomials  of second kind of Type II have the following equivalent representations:
\begin{align}
\label{eq:factorization.LaguerreII.sec.kind}
  \mathcal   L_{\bm{n}}^{(\alpha,\bm{c})}(x)&= q^{(\alpha)} \boxtimes_{|\bm n|}(p_1\boxplus_{|\bm n|}p_2\boxplus_{|\bm n|}\dots \boxplus_{|\bm n|} p_r)\\
  \label{eq:factorization.LaguerreII.sec.kindBis}
  &= \falling{\alpha+|\bm n|}{|\bm n|} \HGF{1}{1}{-|\bm n|}{\alpha+1}{x}\boxtimes_{|\bm n|}\left(\prod_{j=1}^r\left(x-1/c_j \right)^{n_j} \right),
\end{align}

where 
\begin{equation*}
p_j= \frac{\falling{n_j}{|\bm n|}}{c_j^{|\bm n|}}\HGF{1}{1}{-|\bm n|}{n_j-|\bm n|+1}{c_jx}= (-x)^{|\bm n|-n_j}\frac{\falling{|\bm n|}{n_j}}{c_j^{n_j}}\HGF{1}{1}{-|\bm n|}{|\bm n|-n_j+1}{c_jx}
\end{equation*}
and
\begin{equation*}
    q^{(\alpha)}(x)= \frac{\falling{\alpha+|\bm n|}{|\bm n|}}{|\bm n|!} \HGF{2}{1}{-|\bm n|,\ 1}{\alpha+1}{x}.
\end{equation*}
\end{theorem}
\begin{proof}
For \eqref{eq:factorization.LaguerreII.sec.kind}, notice that in the notation \eqref{defFMULTCONV}, 
\begin{equation}\label{pcoefficient}
    e_k(p_j)=\begin{cases}
       \dfrac{\falling{|\bm n|}{k}\falling{n_j}{k}}{k!c_j^{k}} & \text{for } 0\leq k\leq n_j, \\
        0 & \text{for } n_j< k \leq |\bm n| .
    \end{cases}
\end{equation}
Therefore, by \eqref{coeffAdditiveConv},  
\begin{equation}\label{psum}
\begin{split}
    e_k(p_1\boxplus_{|\bm n|} p_2\boxplus_{|\bm n|} \cdots\boxplus_{|\bm n|} p_r)&=\falling{|\bm n|}{k}\sum_{k_1+\dots+k_r=k}\prod_{j=1}^{r} \frac{\falling{|\bm n|}{k_j}\falling{n_j}{k_j}}{\falling{|\bm n|}{k_j}k_j!c_j^{k_j}}\\
    &=\falling{|\bm n|}{k}\sum_{k_1+\dots+k_r=k}\prod_{j=1}^{r}\binom{n_j}{k_j}\frac{1}{c_j^{k_j}}.
\end{split}
\end{equation}
For the polynomial $q^{(\alpha)}$, we have
\begin{equation*}
    e_k(q^{(\alpha)})= \binom{|\bm n|}{k}\frac{\falling{|\bm n|+\alpha}{k}}{\falling{|\bm n|}{k}}.
\end{equation*}
By the formula \eqref{coeffMultConv}, for the coefficients of the finite free multiplicative convolution, we obtain after some simplifications
$$e_k(q^{(\alpha)} \boxtimes_{|\bm n|}(p_1\boxplus_{|\bm n|} p_2\boxplus_{|\bm n|}\dots \boxplus_{|\bm n|} p_r))=\falling{{|\bm n|}+\alpha}{k} \sum_{k_1+\dots+k_r=k}\prod_{j=1}^{r}\binom{n_j}{k}\frac{1}{c_j^{k_j}}.$$
It remains to compare these expressions with \eqref{eq:TypeII.Laguerre.second} to get \eqref{eq:factorization.LaguerreII.sec.kind}. 

For \eqref{eq:factorization.LaguerreII.sec.kindBis}, we use the decomposition for $q^{(\alpha)}$,
\begin{equation*}
q^{(\alpha)}= \frac{\falling{\alpha+|\bm n|}{|\bm n|}}{|\bm n|!}\HGF{1}{1}{-|\bm n|}{\alpha+1}{x}\boxtimes_{|\bm n|}\HGF{2}{0}{-|\bm n|,1}{\cdot}{x},
\end{equation*}
and substitute it in \eqref{eq:factorization.LaguerreII.sec.kind}:
\begin{equation}\label{stepdecomposition}
    \begin{split}
          \mathcal   L_{\bm{n}}^{(\alpha,\bm{c})}(x)=&\falling{\alpha+|\bm n|}{|\bm n|} \HGF{1}{1}{-|\bm n|}{\alpha+1}{x} \boxtimes_{|\bm n|}\\
          &\left[\frac{1}{|\bm n|!}\HGF{2}{0}{-|\bm n|,1}{\cdot}{x}\boxtimes_{|\bm n|}(p_1\boxplus_{|\bm n|}p_2\boxplus_{|\bm n|}\dots \boxplus_{|\bm n|} p_r)\right].
    \end{split}
\end{equation}
Observe that 
\begin{equation*}
    e_k\left(\frac{1}{|\bm n|!}\HGF{2}{0}{-|\bm n|,1}{\cdot}{x} \right)= \binom{|\bm n|}{k}\frac{1}{\falling{|\bm n|}{k}}.
\end{equation*}
Using it with \eqref{psum} and \eqref{coeffMultConv}, we get that
\begin{equation}
    e_k\left(\frac{1}{|\bm n|!}\HGF{2}{0}{-|\bm n|,1}{\cdot}{x}\boxtimes_{|\bm n|}(p_1\boxplus_{|\bm n|}p_2\boxplus_{|\bm n|}\dots \boxplus_{|\bm n|} p_r)\right)=\sum_{k_1+\dots+k_r=k}\prod_{j=1}^{r}\binom{n_j}{k_j}\frac{1}{c_j^{k_j}}.
\end{equation}
By \eqref{pcoefficient}, the non-vanishing coefficients of the sum in the right side  correspond to the indices such that $0\leq k_j\leq n_j$. Additionally,
\begin{equation*}
    e_k\left(\prod_{j=1}^r\frac{1}{c_j^{n_j}}\HGF{1}{0}{-n_j}{\cdot}{c_jx} \right)= \sum_{k_1+\dots+k_r=k}\prod_{j=1}^{r}\binom{n_j}{k_j}\frac{1}{c_j^{k_j}}.
\end{equation*}
Thus,
$$
\frac{1}{|\bm n|!}\HGF{2}{0}{-|\bm n|,1}{\cdot}{x}\boxtimes_{|\bm n|}(p_1\boxplus_{|\bm n|}p_2\boxplus_{|\bm n|}\dots \boxplus_{|\bm n|} p_r)=\prod_{j=1}^r\frac{1}{c_j^{n_j}}\HGF{1}{0}{-n_j}{\cdot}{c_jx},
$$
and it remains to use it in \eqref{stepdecomposition} to obtain \eqref{eq:factorization.LaguerreII.sec.kindBis}.
\end{proof}

\begin{theorem}
    Let a multi-index $\bm n\in\N^{r}$, $\alpha>-1$  and  $\bm c =(c_1,\dots,c_r)\in \R_{>0}^r$ be such that $c_i\not= c_j$ for $i\not= j$.  For  $0< t\leq 2$,      the following interlacing holds:
\begin{equation}
\mathcal   L_{\bm{n}}^{(\alpha,\bm{c})}(x)\preccurlyeq   \mathcal   L_{\bm{n}}^{(\alpha+t,\bm{c})}(x)        .
\end{equation}
\end{theorem}
\begin{proof}
   By \cite[Eq. (55)]{martinez2023real}, we have the following interlacing for $0<t\leq 2$,
\begin{equation}\label{interalcingLaguerre}
    \HGF{1}{1}{-|\bm n|}{\alpha+1}{x}\prec \HGF{1}{1}{-|\bm n|}{\alpha+1+t}{x},
\end{equation}
Since $c_i>0$ for each $i=1,\dots,r$,  
\begin{equation}
    \label{polync}
\prod_{j=1}^r\left(x-1/c_j \right)^{n_j}\in \P_n({\R_{>0}}).
\end{equation}
Hence, by Proposition \ref{lem:preservinginterlacingMult}, the multiplicative finite free convolution of the polynomials in \eqref{interalcingLaguerre} by \eqref{polync} preserves the interlacing, so that 
\begin{equation}
   \HGF{1}{1}{-|\bm n|}{\alpha+1}{x}\boxtimes_{|\bm n|}\prod_{j=1}^r\left(x-1/c_j \right)^{n_j}\preccurlyeq \HGF{1}{1}{-|\bm n|}{\alpha+1+t}{x}\boxtimes_{|\bm n|}\prod_{j=1}^r\left(x-1/c_j \right)^{n_j},
\end{equation}
which, by \eqref{eq:factorization.LaguerreII.sec.kindBis}, is equivalent to $\mathcal   L_{\bm{n}}^{(\alpha,\bm{c})}(x)\preccurlyeq   \mathcal   L_{\bm{n}}^{(\alpha+t,\bm{c})}(x)$.
\end{proof}

\subsubsection{Zero asymptotics of Type II Multiple Laguerre polynomials of the second kind}\

Let $\bm n =(n_1, \dots, n_r)\in \N^r$ and $\alpha_{\bm n} >-1$ be such that
\begin{equation}
    \label{asymptMLIIkind2assumptions}
    \lim_{|\bm n|\to \infty} \frac{\alpha_{\bm n}}{|\bm n|}=A\ge 0,    \quad \text{and} \quad \lim_{|\bm n|\to \infty} \frac{n_i}{|\bm n|}=\theta_i>0, \qquad i=1, 2.
\end{equation}
Moreover, as in in Section \ref{sec:asymptTypeILagueresecond}, assume that $\bm c_{\bm n}=\left(c_1^{(\bm n)}, \dots, c_r^{(\bm n)}\right)\in \R_{>0}^r$, with $c_i^{(\bm n)}\neq c_j^{(\bm n)}$ for $i\neq j$, depends on the multi-index $\bm n$ in such a way that the limit
\begin{equation}
    \label{asymptMLIIkind2assumptions2}
    \lim_{\bm n} \bm c_{\bm n} = \bm c =(c_1, \dots, c_r)\in \R_{>0}^r
\end{equation}
exists, with all $c_i\neq c_j$ for $i\neq j$. Recall (see Remark \ref{rem:scalingLaguerre}) that we can also easily handle the case when $\bm c_{\bm n}$ depend linearly on $|\bm n|$.

\begin{theorem}
Under the assumptions above, function $y=u \mathcal G_\mu(u)$, where $\mathcal G_\mu$ is the Cauchy transform  of the limiting distribution of zeros of the rescaled polynomials $ \mathcal   L_{\bm{n}}^{(\alpha_{\bm n},\bm{c}_{\bm n})}(|\bm n| x)$, satisfies the equation
\begin{equation}
    \label{eq:CauchyTypeIILagKind2}
\frac{y+A}{y-1} \sum_{j=1}^r\frac{\theta_j  }{c_j u - \left( y+A\right)   }=1.
\end{equation}
\end{theorem}
We can prove this assertion using any of the representations from Theorem~\ref{thm:factorization.LaguerreII.sec.kind}. Formula  \eqref{eq:factorization.LaguerreII.sec.kindBis} yields an expression in terms of a multiplicative convolution of a Marchenko-Pastur distribution with a discrete measure supported at the points $1/c_j$; this retrieves, in particular, the result of Hardy \cite[Theorem 3.7]{hardy2015} established using random matrix theory.

Below we use \eqref{eq:factorization.LaguerreII.sec.kind} to illustrate how to obtain asymptotics from an expression involving both multiplicative and additive finite convolutions. Our result is equivalent to \cite[Corollary 3.8]{hardy2015} after some algebraic manipulations.
\begin{proof}
By the representation \eqref{eq:factorization.LaguerreII.sec.kind},
$$
  \mathcal   L_{\bm{n}}^{(\alpha_{\bm n},\bm{c}_{\bm n})}(|\bm n| x)= q^{(\alpha_{\bm n})}(x) \boxtimes_{|\bm n|} \left[ (p_1\boxplus_{|\bm n|}p_2\boxplus_{|\bm n|}\dots \boxplus_{|\bm n|} p_r)(|\bm n | x) \right],
$$
while from \eqref{identityAdd3a} it follows that
$$
\dil{\frac{1}{|\bm n|}}(p_1\boxplus_{|\bm n|}p_2\boxplus_{|\bm n|}\dots \boxplus_{|\bm n|} p_r) \simeq (\dil{\frac{1}{|\bm n|}} p_1) \boxplus_{|\bm n|}(\dil{\frac{1}{|\bm n|}} p_2)\boxplus_{|\bm n|}\dots \boxplus_{|\bm n|} (\dil{\frac{1}{|\bm n|}} p_r).
$$
Now,
$$
[\dil{\frac{1}{|\bm n|}} p_j](x)= p_j(|\bm n | x)\simeq  \HGF{1}{1}{-|\bm n|}{n_j-|\bm n|+1}{ c_j^{(\bm n)} |\bm n | x}  ,
$$
or, equivalently, by \eqref{identityMult2},  
$$
p_j(|\bm n | x)\simeq  \left(x-\frac{1}{c_j^{(\bm n)}}\right)^{|\bm n|} \boxtimes_{|\bm n|}\HGF{1}{1}{-|\bm n|}{n_j-|\bm n|+1}{|\bm n|x}, \quad j=1, \dots, r.
$$
Under assumptions \eqref{asymptMLIIkind2assumptions}--\eqref{asymptMLIIkind2assumptions2}, the normalized zero-counting measure of the first polynomial in the right-hand side tends to $\delta_{1/c_j}$, whose $S$-transform is the constant $c_j\neq 0$. Thus, applying Theorem \ref{thm:StranfHyper2F1}, we conclude that the weak-* limit $\nu_j$ of the normalized zero-counting measures of the scaled polynomials $p_j(|\bm n | x)$ is a positive probability measure, compactly supported on the real line, for which 
\begin{equation}
   \mathcal S_{\nu_j}(z)= \frac{c_j}{z+\theta_j}.
\end{equation}
From \eqref{RtransformStranform} it follows that its $R$-transform is given by, 
$$
\mathcal R_{\nu_j}(z)=\frac{\theta_j}{c_j-z}.
$$
In the terminology of the free probability, $\nu_j$ is a free Poisson (or Marchenko-Pastur) distribution of rate $\theta_j$ and jump of size $\frac{1}{c_j}$ (see \cite[Definition 12.12]{MR2266879}).

Applying \eqref{def:additiveconvolutionmeasures}, we see that the normalized zero-counting measure of $(p_1\boxplus_{|\bm n|}p_2\boxplus_{|\bm n|}\dots \boxplus_{|\bm n|} p_r)(|\bm n | x) $ converges to a measure $\nu_p$, whose $R$-transform is\footnote{\, Notice that we can identify $\nu_p$ as compound free Poisson of rate $1$ and jump distribution $ \sum_{j=1}^r \theta_j \delta_{\frac{1}{c_j}}$,  see \cite[Definition 12.16]{MR2266879}.}
\begin{equation}
\label{eq:Rtrans.MLIIsecond}
\mathcal R_{\nu_p}(z)=\sum_{j=1}^r\frac{\theta_j}{c_j-z}.
\end{equation}
Using \eqref{RtransformStranform} again, we conclude that the $S$-transform $w=S_{\nu_p}(z)$ of $\nu_p$ satisfies the algebraic equation
\begin{equation}
    \label{eq:algebraicLag1}
1=\sum_{j=1}^r\frac{\theta_j w}{c_j-wz}.
\end{equation}

On the other hand, for the polynomial
$$
   q^{(\alpha_{\bm n})}(x)\simeq  \HGF{2}{1}{-|\bm n|,\ 1}{\alpha_{\bm n}+1}{x}
$$
we can use Proposition~\ref{prop:asymptotics2F1} to assure that for its limiting zero-counting measure $\nu_q$,
\begin{equation}
   \mathcal S_{\nu_q}(z)= \frac{z+1}{z+A+1}.
\end{equation}
Since the weak-* limit $\mu$ of the normalized zero-counting measure of the scaled polynomials $ \mathcal   L_{\bm{n}}^{(\alpha_{\bm n},\bm{c}_{\bm n})}(|\bm n| x)$ is given by
$$
\mu =\nu_p\boxtimes\nu_q,
$$
we have that its $S$-transform is 
\begin{equation}
   \mathcal S_{\mu}(z)= \mathcal S_{\nu_p}(z) \mathcal S_{\nu_q}(z)=  \frac{z+1}{ z+A+1} \, \mathcal S_{\nu_p}(z).
\end{equation}
Substituting this into \eqref{eq:algebraicLag1}, we get an algebraic equation for the $S$-transform $w=S_{\mu}(z)$:
\begin{equation}
    \label{eq:algebraicLag2}
1=(z+A+1)\sum_{j=1}^r\frac{\theta_j w}{c_j(z+1)-wz(z+A+1)}.
\end{equation}
With the definition \eqref{defStranform}, we can write it as
$$
1=\frac{z+A+1}{z} \sum_{j=1}^r\frac{\theta_j \, \mathcal M_\mu^{-1}(z)}{c_j-(z+A+1)  \mathcal M_\mu^{-1}(z)}.
$$
Recalling the definition \eqref{defMtransform} of the $M$-transform in terms of the Cauchy transform  $\mathcal G_\mu$ of $\mu$, we have that
$$
z=\frac{1}{u}  \mathcal G_\mu\left(\frac{1}{u}\right)-1;
$$
replacing $u\mapsto 1/u$ we can rewrite the equation above as
$$
1=\frac{u   \mathcal G_\mu\left(u\right)+A}{u   \mathcal G_\mu\left(u\right)-1} \sum_{j=1}^r\frac{\theta_j  }{c_j u - \left(u   \mathcal G_\mu\left(u\right)+A\right)   }.
$$
Denoting $y= u \mathcal G_\mu(u)$ we arrive at equation \eqref{eq:CauchyTypeIILagKind2}.
\end{proof}

\section{Further examples}

Recent research has revealed other (although not so many) families of multiple orthogonal polynomials that can be expressed in terms of generalized hypergeometric functions, to which the methodology explained here can be applied.

For instance, in \cite{MR4429042} Type II MOP with respect to a pair of weights ($r=2$) on $[0,1]$, 
$$
w_j(x)=\mathcal{W}(x ; a, b+j-1 ; c+j-1, d), \quad j=1, 2,
$$
where
$$
\mathcal{W}(x ; a, b ; c, d)=\frac{\Gamma(c) \Gamma(d)}{\Gamma(a) \Gamma(b) \Gamma(\delta)} x^{a-1}(1-x)^{\delta-1}{ }_2 F_1\left(\begin{array}{c}
c-b, d-b \\
\delta
\end{array} ; 1-x\right), \quad \delta=c+d-a-b.
$$
Note that $\mathcal{W}(x ; a, b ; c, d)$ is a positive function on the interval $(0,1)$, whenever $a, b, c, d \in \R_{>0}$ with $c>b, d>b$, and $\delta >0$. It was shown that the corresponding polynomial $P_n(x):=P_n(x ; a, b ; c, d)$, such that
$$
\int_0^1 x^k P_n(x) \mathcal{W}(x ; a, b+j ; c+j, d) d x  \begin{cases} = 0, & \text { if } n \geq 2 k+j+1, \\   \neq 0, & \text { if } n=2 k+j,\end{cases}
$$
is hypergeometric (see \cite[Theorem 6]{MR4429042}):
$$
P_n(x) \simeq { }_3 F_2\left(\begin{array}{c}
-n, c+\left\lfloor\frac{n}{2}\right\rfloor, d+\left\lfloor\frac{n-1}{2}\right\rfloor \\
a, b
\end{array} ; x\right).
$$
This means that we could use the arguments above to establish some monotonicity and interlacing properties of the zeros of $P_n$ (all real, simple, and on the interval $(0,1)$). As for the zero asymptotics, the authors in \cite{MR4429042} observe that $P_n$'s share it with the Type II Jacobi-Piñeiro polynomials on the step line, i.e., is given by \eqref{densityJPIIDiag}.

\medskip

Another example of Type II MOP appears in \cite{MR4154921}, this time with respect to two weights on $[0,+\infty)$ from the family
$$
w(x ; a, b ; c)=\frac{\Gamma(c)}{\Gamma(a) \Gamma(b)} \mathrm{e}^{-x} x^{a-1} U(c-b, a-b+1 ; x) ,
$$
expressed in terms of the confluent hypergeometric function of the second kind, $U(\alpha, \beta ; x)$, also known as the Tricomi function: for $\Re(\alpha)>0$ and $|\arg (x)|<\frac{\pi}{2}$,    
$$
U(\alpha, \beta ; x)=\frac{1}{\Gamma(\alpha)} \int_0^{\infty} t^{\alpha-1}(t+1)^{\beta-\alpha-1} e^{-t x} \mathrm{~d} t.
$$
Then, for two weights from this family, we can define Type II MOP as follows: for $a, b, c>0$ such that $c>\max \{a, b\}>0$ and $d \in\{0,1\}$, let $P_n=P_n^{[d]}$ satisfy
 $$
\int_0^{\infty} x^k P_n^{[d]}(x) w(x ; a, b ; c+d) \,d x  \begin{cases} =0, & \text { if } n \geq 2 k+1, \\  \neq 0, & \text { if } n=2 k,\end{cases}
$$
and
$$
\int_0^{\infty} x^k P_n^{[d]}(x) w(x ; a, b ; c+1-d) \, d x \begin{cases} =0, & \text { if } n \geq 2 k+2, \\   \neq 0, & \text { if } n=2 k+1.
\end{cases}
$$
We can identify $P_n$ with Type II MOP on the step-line by defining 
$$
P_{(n, n)}(x)= P_{2 n}(x) \quad \text { and } \quad P_{(n+1, n)}(x) = P_{2 n+1}(x).
$$
From \cite[Theorem 3.1]{MR4154921},
$$
P_n^{[d]}(x) \simeq \HGF{2}{2}{-n, c+\left\lfloor\frac{n+d}{2}\right\rfloor }{
a, b}{x}. 
$$

Hypergeometric polynomials arise also when considering multiple orthogonality, this time of both Type I and Type II, with respect to an exponential integral. Namely, \cite{vanassche} considered two weights on $[0, \infty)$,
$$
w_1(x)=x^\alpha e^{-x}, \quad w_2(x)=x^\alpha E_{\nu+1}(x) ,
$$
where
$$
E_\nu(x)=\int_1^{\infty} \frac{e^{-x t}}{t^\nu} d t,
$$
$\alpha>-1$ and $\alpha+\nu>-1$.  This pair of weights was shown to form a Nikishin system on $[0, \infty)$. 

If we consider the Type I MOP for $\left(w_1, w_2\right)$ that corresponds to the multi-index $\bm n =(n_1, n_2)$, that is, a vector of $r=2$ polynomials $\left(A_{\bm{n},1}, A_{\bm{n},2}\right)$ with $\deg A_{\bm{n},j} \leq n_j-1$, for
which the function \eqref{def:typeIfunction},
$$ 
    Q_{\bm{n}}(x) \isdef \sum_{j=1}^r A_{\bm{n},j}(x)w_j(x) 
$$
is orthogonal to all polynomials of degree $\leq |\bm{n}|-2$,
$$
  \int_0^{+\infty} x^{k}  Q_{\bm{n}}(x)\, dx =   0, \qquad 0 \leq k \leq |\bm{n}|-2,
$$
then it was proved in \cite[\S 2.2]{vanassche} that, provided that for $n_1+1 \geq n_2$ and $\nu \notin \mathbb{Z}$,
$$
A_{\bm{n},2}(x)  \simeq { }_2 F_2\left(\begin{array}{c}
-n_2+1, |\bm n|+\alpha+\nu \\
\nu+1, \alpha+\nu+1
\end{array} ;-x\right).
$$
On the other hand, for the corresponding Type II MOP $P_{\bm n}$ of degree $|\bm n|$, for the two weight functions $\left(w_1, w_2\right)$, satisfying  
\begin{align*}
\int_0^{\infty} P_{\bm n}(x) x^{k} w_j(x) \, d x & =0, \quad k=0, \ldots, n_j-1, \quad j=1,2, 
\end{align*}
it is established that 
\begin{equation}
    \label{VATypeI}
P_{n, m}(x) \simeq 
{}_2 F_2\left(\begin{array}{cc}
-|\bm n|, n_2+\alpha+\nu+1 \\
\alpha+1, \alpha+\nu+1 
\end{array} ; x\right),
\end{equation}
see \cite[\S 3.2]{vanassche}.

Finally, a general class of weights for which the moment-generating functions are hypergeometric series has been considered in the recent publication \cite{wolfs}. In particular, it was shown that in this setting, Type II MOPs have the general form
\begin{equation}
    \label{WolfypeII}
{}_{p+1} F_q\left(\begin{array}{cc}
-|\bm n|, \bm n+\bm \alpha+1 \\
\bm \beta+1 
\end{array} ; x\right),
\end{equation}
and thus, can be represented as the finite free multiplicative convolution of simpler building blocks.
 
In all these examples, the corresponding MOP are suitable for analyzing their zeros using this paper's methodology. For instance, the algebraic equation for the Cauchy transform of the limiting zero distribution of the appropriately rescaled polynomials \eqref{VATypeI} (see \cite[Lemma 6]{vanassche}) or \eqref{WolfypeII} (see \cite[Theorem 2.17 and Corollary 3.15]{wolfs}) are just a straightforward application of Theorem~\ref{cor:CauchyTransf} of this paper.

\appendix

\section{Outline of the proof of Theorem~\ref{prop:finiteAsymptotics.updated}} \label{sec:appendix}

Theorem \ref{prop:finiteAsymptotics.updated} is a key result for some of our applications and it was implicitly proved in \cite{arizmendi2018cumulants, arizmendi2021finite}. The connection between the results proved in those papers and the theorem written here uses the established theory of combinatorics in free probability that can be consulted in \cite{MR2266879}. The purpose of this section is to further clarify this connection.

First, recall from \cite[Remark 3.5]{arizmendi2018cumulants} that given a polynomial $p\in \pp_n$ one can define its degree $n$ finite free cumulants $(\kappa_j^{(n)}(p))_{j=1}^n$ as the values uniquely determined by the formulas:
$$e_j(p)=\frac{\falling{n}{j}}{d^j j!} \sum_{\pi\in \mathcal{P}(k)} n^{|\pi|} \mu(0_j,\pi) \kappa^{(n)}_\pi(p).$$
where $\mathcal{P}(j)$ is family of set partitions of $\{1,\dots,k\}$, $\mu$ is the Möbius function on the lattice $\mathcal{P}(k)$ (with the reversed refinement order), $|\pi|$ is the number of blocks in the partition $\pi$, and $\kappa^{(n)}_\pi(p):=\prod_{V\in\pi} \kappa^{(n)}_{|V|}(p)$. Finite free cumulants have the remarkable property that in the limit they tend to free cumulants, see \cite[Theorem 5.4]{arizmendi2018cumulants}, since the proof only relies in the combinatorial structure, the result can be readily updated as follows:

Let $(p_n)_{n\geq 0}$ be a sequence of polynomials such that $p_n\in \pp_n$ for $n\geq 0$. And let $(r_j)_{j\geq 1}$, $(m_j)_{j \geq 1}$ two sequences of complex numbers that satisfy the formulas:
\begin{equation}
 m_k= \sum_{\pi\in \mathcal{NC}(k)} r_\pi, \qquad \text{for } k=1,2,\dots,   
\end{equation}
where $\mathcal{NC}(k)$ is family of non-crossing set partitions of $\{1,\dots,k\}$, and $r_\pi:=\prod_{V\in\pi} r_{|V|}$.
Then, we have the equivalence:
$$\lim_{n\to\infty} m_j(p_n) = m_j, \quad \text{for }j=1,2,\dots \qquad \text{ if and only if } \qquad \lim_{n\to\infty} \kappa_j^{(n)}(p_n) = r_j\quad \text{for }j=1,2,\dots $$

Notice that the only difference with \cite[Theorem 5.4]{arizmendi2018cumulants} is that we do not require that $(m_j)_{j\geq 1}$ is the sequence of moments of some measure. If this was the case, then $(r_j)_{j\geq 1}$ would be the sequence of free cumulants of the given measure. 

Moreover, from \cite[Proposition 3.6]{arizmendi2018cumulants} we know that finite free cumulants linearize the finite free additive convolution:
$$\kappa^{(n)}_j(p\boxplus_nq)=\kappa^{(n)}_j(p)+\kappa^{(n)}_j(q) \qquad \text{for }j=1,\dots,n.$$
Furthermore, \cite[Theorem 1.2]{arizmendi2021finite} asserts that finite free cumulants of the multiplicative convolution satisfy the same relation as the free cumulants of a product in the limit: 
$$\kappa^{(n)}_j(p\boxtimes_nq)=\sum_{\pi\in \mathcal{NC}(j)} \kappa^{(n)}_\pi(\mathfrak{p})\kappa^{(n)}_{Kr(\pi)}(\mathfrak{q})+O(1/n)\qquad \text{for } j=1,2,\dots, $$
where, $Kr(\pi)$ denotes the Kreweras complement of a non-crossing partition and $O(1/n)$ is a term that tends to 0 when $n\to \infty$. When turning to the limiting behaviour, the previous results imply the following:

Let $\mathfrak{p}=(p_n)_{n\geq 0}$, $\mathfrak{q}=(p_n)_{n\geq 0}$ be a sequence of polynomials such that $p_n,q_n\in \pp_n$ for $n\geq 0$. And assume that
for all $j=1,2,\dots$ it holds that
$$\alpha_j:=\lim_{n\to\infty} m_j(p_n) \in\cc, \qquad \text{and}\qquad \beta_j:=\lim_{n\to\infty} m_j(q_n) \in\cc.$$
Then for all $j=1,2,\dots$ one has that
$$\gamma_j:=\lim_{n\to\infty} m_j(p_n\boxplus_n q_n)\in\cc\qquad \text{and,}$$
$$\theta_j:=\lim_{n\to\infty} m_j(p_n\boxtimes_n q_n)\in\cc.$$
Furthermore, the sequences $(\gamma_j)_{j\geq 1}$ and $(\theta_j)_{j\geq 1}$ can be computed as follows. For $\mu=\alpha,\beta, \gamma,\theta$ we let $(r_j(\mu))_{j\geq 1}$ be the sequences such that:
\begin{equation}
\label{eq:moment.cumulants}
 \mu_k= \sum_{\pi\in \mathcal{NC}(k)} r_\pi(\mu), \qquad \text{for } k=1,2,\dots.
\end{equation}
Then
\begin{align}
r_j(\gamma)&=r_j(\alpha)+ r_j(\beta) \qquad \text{ for } j=1,2,\dots, \text{ and } \label{eq:cumulants.additive} \\ 
r_j(\theta)&=\sum_{\pi\in \mathcal{NC}(j)} r_\pi(\alpha)r_{K(\pi)}(\beta)\qquad \text{for } j=1,2,\dots, r. \label{eq:cumulants.multiplicative}
\end{align}

It is a well-known fact that the previous relations between sequences are equivalent to the formal power series relations between the R-transform, S-transform and Cauchy transform. More specifically, 
Theorem \ref{prop:finiteAsymptotics.updated} follows from the following facts:
\begin{itemize}
\item Equation \eqref{eq:moment.cumulants} is true for every $k\geq 1$ if and only if the $r_j(\mu)$ are the coefficients of the $R$-transform series $\RR_\mu$ associated to the sequence $(\mu_j)_{j\geq 1}$, see \cite[Remark 16.18]{MR2266879} 
\item The equation $\RR_\gamma(z)=\RR_\alpha(z)+ \RR_\beta(z)$ at the level of formal power series is equivalent the same equality at each coefficient, namely Equation \eqref{eq:moment.cumulants}.
\item If $(\alpha_j)_{j\geq 1}, (\beta_j)_{j\geq 1},(\theta_j)_{j\geq 1}$ are sequences such that their associated $S$-transform satisfy $S_\theta(z)=S_\alpha(z) S_\beta(z)$, this is equivalent to the coefficients of their corresponding $R$-transforms satisfying Equation \eqref{eq:cumulants.multiplicative} for all $j\geq 1$, see \cite[Corollary 18.17]{MR2266879}.
\end{itemize}

 \section*{Acknowledgments}

The first author was partially supported by Simons Foundation Collaboration Grants for Mathematicians (grant 710499).
He also acknowledges the support of the project PID2021-124472NB-I00, funded by MCIN/AEI/10.13039/501100011033 and by ``ERDF A way of making Europe'', as well as the support of Junta de Andaluc\'{\i}a (research group FQM-229 and Instituto Interuniversitario Carlos I de F\'{\i}sica Te\'orica y Computacional). 

The third author was partially supported by the Simons Foundation via Michael Anshelevich's grant. He expresses his gratitude for the warm hospitality and stimulating atmosphere at Baylor University.  

This work has greatly benefited from our discussions with several colleagues, such as Octavio Arizmendi, Amilcar Branquinho, Juan E. F. Díaz, Ulises Fidalgo, Ana Foulqui\'e, Walter Van Assche, and Thomas Wolfs. We are also grateful to Grzegorz \'Swiderski, who brought the paper \cite{hardy2015} to our attention.
 
%\bibliographystyle{abbrv}
%\bibliography{References}

\end{document}